\documentclass[11pt,reqno]{amsproc}
\allowdisplaybreaks
\usepackage{graphicx}
\usepackage{subfigure}
\usepackage{morefloats}
\usepackage{mathtools}
\usepackage{amssymb}
\usepackage{stmaryrd}
\usepackage{caption}   % for \captionsetup and \caption*

\usepackage{enumerate}
\usepackage{fullpage}
\usepackage{textcomp}
\usepackage{lettrine}
\usepackage{booktabs}

\usepackage[debug=false, colorlinks=true, pdfstartview=FitV,
linkcolor=blue, citecolor=blue, urlcolor=blue]{hyperref}
\usepackage[semicolon,square,authoryear,sort]{natbib}

\usepackage[doipre={DOI:~}]{uri}

\usepackage{color}
\usepackage{colortbl}
\usepackage[dvipsnames,table]{xcolor}
\usepackage[most]{tcolorbox}

\newlength{\drop}
\definecolor{amethyst}{rgb}{0.6, 0.4, 0.8}
\definecolor{burgundy}{rgb}{0.5, 0.0, 0.13}

\usepackage{longtable} 
\usepackage{multirow} 
\usepackage{rotating} 
\usepackage{bigstrut} 
\usepackage{hhline} 
\usepackage{listings}
\usepackage{xcolor}

\lstdefinestyle{mypython}{
    language=Python,
    backgroundcolor=\color{gray!10},   % light gray background
    commentstyle=\color{green!60!black}, % comments in green
    keywordstyle=\color{blue},          % keywords in blue
    numberstyle=\tiny\color{gray},      % line numbers in gray
    stringstyle=\color{red!70!black},   % strings in red
    basicstyle=\ttfamily\footnotesize,  % code font
    breaklines=true,                    % automatic line breaking
    frame=single,                       % adds a frame around the code
    numbers=left,                       % line numbers on the left
    stepnumber=1,                       % number every line
    showstringspaces=false,             % hide spaces in strings
    tabsize=4,                          % tab width
    captionpos=t                        % caption at the bottom
}
\usepackage{amsthm}

\newtheorem{theorem}{Theorem}[section]
\newtheorem{lemma}[theorem]{Lemma}
\newtheorem{proposition}[theorem]{Proposition}

\newtheorem{remark}{Remark}[section]

\title{\textbf{An Adaptive Machine Learning Framework for Fluid Flow \\
in Dual-Network Porous Media}}

\author{\textbf{Venkat S.~Maduri} and \textbf{Kalyana B.~Nakshatrala} \\
  {\small Department of Civil and Environmental Engineering \\
  University of Houston, Houston, Texas 77204, USA.}\\
  {\small \textbf{Correspondence to:} ~knakshatrala@uh.edu (e-mail), +1-713-743-4418 (phone)}}


\begin{document}

%===========================;
%  Title page of the paper  ;
%===========================;
\begin{titlepage}
  \drop=0.1\textheight
  \centering
  \vspace*{0.5\baselineskip}
  \rule{\textwidth}{1.6pt}\vspace*{-\baselineskip}\vspace*{2pt}
  \rule{\textwidth}{0.4pt}\\[\baselineskip]
       {\Large \textbf{\color{burgundy}
       An Adaptive Machine Learning Framework for Fluid Flow \\[0.3\baselineskip] in Dual-Network Porous Media}}\\[0.3\baselineskip]
       \rule{\textwidth}{0.4pt}\vspace*{-\baselineskip}\vspace{3.2pt}
       \rule{\textwidth}{1.6pt}\\[0.5\baselineskip]
       \scshape
       An e-print of this paper is available on arXiv. \par
       \vspace*{0.25\baselineskip}
       Authored by \\[0.5\baselineskip]

       {\Large V.~S.~Maduri\par}
  {\itshape Graduate Student, Department of Civil \& Environmental Engineering \\
  University of Houston, Houston, Texas 77204.}\\[0.25\baselineskip]

  {\Large K.~B.~Nakshatrala\par}
  {\itshape Department of Civil \& Environmental Engineering \\
  University of Houston, Houston, Texas 77204. \\
  \textbf{phone:} +1-713-743-4418, \textbf{e-mail:} knakshatrala@uh.edu \\
  \textbf{website:} http://www.cive.uh.edu/faculty/nakshatrala}\\[0.75\baselineskip]

\vspace{-0.25in}
%----------------------;
%  Graphical abstract  ;
%----------------------;
   \begin{figure*}[ht]
        \centering
        \includegraphics[width=0.9\linewidth]{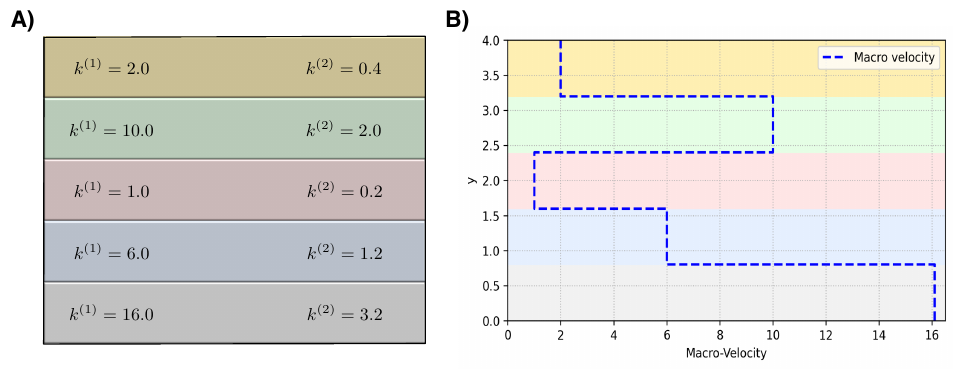}
        \captionsetup{labelformat=empty}
        \caption{\textbf{A)} The problem is formulated within the framework of the double porosity/permeability (DPP) model, consisting of multiple layers, each characterized by distinct macro- and micro-scale permeabilities, denoted by \( k^{(1)} \) and \( k^{(2)} \), respectively. The analytical solution within each layer yields constant horizontal components for both the macro- and micro-scale velocities, while the vertical components are zero. \textbf{B)} The macro-scale velocity field, computed using the proposed adaptive Physics-Informed Neural Network (PINN) framework, is visualized across the layers. The framework accurately captures discontinuities in the velocity field, significantly outperforming continuous mixed finite element methods and achieving accuracy levels typically associated with discontinuous Galerkin (DG) methods.}
  \end{figure*}

  \vfill
  {\scshape 2026} \\
  {\small Computational \& Applied Mechanics Laboratory} \par
\end{titlepage}

%=========================;
%  Abstract of the paper  ;
%=========================;
\begin{abstract}    
    Porous materials---natural or engineered---often exhibit dual pore-network structures that govern processes such as mineral exploration and hydrocarbon recovery from tight shales. Double porosity/permeability (DPP) mathematical models describe incompressible fluid flow through two interacting pore networks with inter-network mass exchange. Despite significant advances in numerical methods, there remains a need for computational frameworks that enable rapid forecasting, data assimilation, and reliable inverse analysis. To address this, we present a physics-informed neural network (PINN) framework for forward and inverse modeling of DPP systems. The proposed approach encodes the governing equations in mixed form, along with boundary conditions, directly into the loss function, with adaptive weighting strategies to balance their contributions. Key features of the framework include adaptive weight tuning, dynamic collocation point selection, and the use of shared trunk neural architectures to efficiently capture the coupled behavior of the dual pore networks. The proposed PINN framework offers several notable advantages. It is inherently mesh-free, making it well-suited for complex geometries typical of porous media. It accurately captures discontinuities in solution fields across layered domains without introducing spurious oscillations commonly observed in classical finite element formulations. Importantly, the framework is well-suited for inverse analysis, enabling robust parameter identification in scenarios where key physical quantities---such as the mass transfer coefficient in DPP models---are difficult to measure directly. In addition, a systematic convergence analysis is provided to rigorously assess the stability, accuracy, and reliability of the method. The effectiveness and computational advantages of the approach are demonstrated through a series of representative numerical experiments.
\end{abstract}

\maketitle

%==================================;
%  Include all the sections below  ;
%==================================;
\setcounter{figure}{0}   
    \vspace{-0.2in}
    \section*{Abbreviations}

\begin{center}
   \begin{tabular}{|l|l|} \hline 
       DG & Discontinuous Galerkin \\
       DPP & Double Porosity/Permeability \\
       FEM & Finite Element Method \\
       LBB & Ladyzhenskaya–Babu\v{s}ka–Brezzi \\
       LS & Least-Squares \\ 
       PDE & Partial Differential Equation \\
       PINN & Physics-Informed Neural Network \\ 
       RAR & Residual-based Adaptive Refinement \\\hline  
   \end{tabular}
\end{center}
    
    \newpage 
    
    %*********************************************;
%                                             ;
%  NAME                                       ;
%    S1_APINNs_Intro.tex                      ;
%                                             ;
%  WRITTEN BY                                 ;
%    Kalyana B. Nakshatrala                   ;
%                                             ;
%*********************************************;
\section{INTRODUCTION AND MOTIVATION}
\label{Sec:S1_APINNs_Intro}
\lettrine[findent=2pt]{\fbox{\textbf{I}}}n many natural porous materials---such as fractured carbonates and clay-rich soils---two distinct pore networks coexist \citep{warren1963behavior}. Synthetic materials, such as engineered composites, can also exhibit dual pore-network structures; in many cases, they are intentionally designed this way to achieve specific functionalities (e.g., sound absorption or thermal insulation) \citep{pornea2022}. Interestingly, dual-pore structures are not confined to a single scale; rather, they can appear at both the material and system levels, manifesting across multiple length scales \citep{Gerke1993}. This multiscale nature is particularly evident in natural systems composed of porous building blocks. For example, consider a geological formation consisting of a pile of rocks: each rock contains an internal pore network, while large voids exist between individual rocks (see \textbf{Fig.~\ref{Fig1:APINNs_Concept_figure}}). In this case, the intra-rock pores form one network, and the inter-rock voids form a second, distinct network. As a result, fluid flow within such systems exhibits clearly dual behavior \citep{Barenblatt1960}.

%------------------------------------;
%  Figure 1: Concept figure for DPP  ;
%------------------------------------;
\begin{figure}[h]
    \centering
    \includegraphics[width=0.4\linewidth]{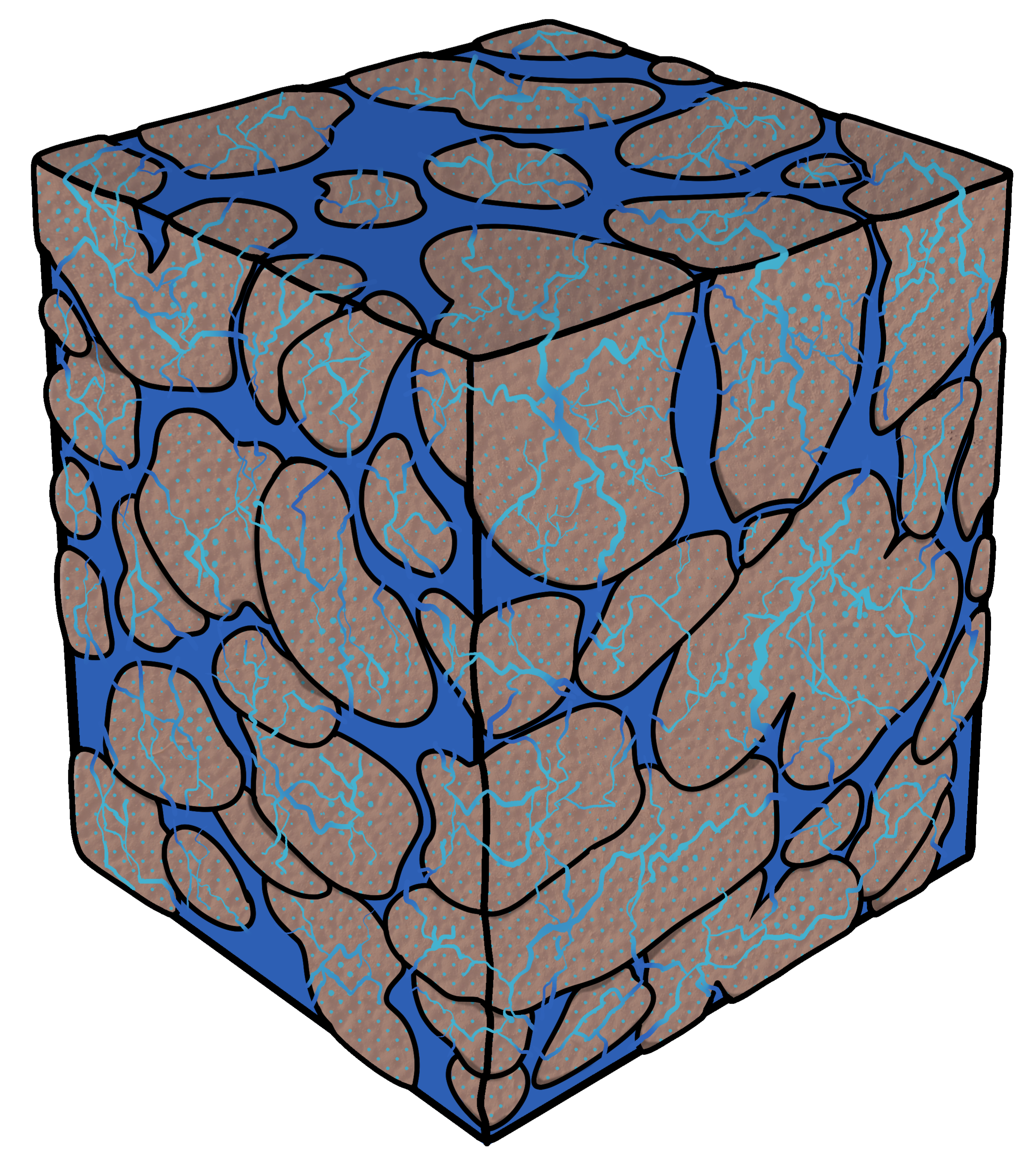}
    \caption{Fluid flow in a porous medium with dual pore networks, illustrating mass exchange between the networks. The macro-pore network is shown in dark blue and the micro-pore network in light blue. (For interpretation of the references to color in this figure legend, the reader is referred to the online version of this article.) \label{Fig1:APINNs_Concept_figure}}
\end{figure}

Beyond their structural coexistence, the two pore networks also differ fundamentally in their properties and roles. They vary in pore size and hydrodynamic behavior (e.g., permeability) \citep{Vogel2000}, and they serve distinct functions in both natural and engineered systems. This dual functionality is illustrated by the following two examples:
%---------------------------------------;
%  Enumerate: Two application examples  ;
%---------------------------------------;
\begin{enumerate}
    \item \textbf{Hydrocarbon recovery from tight shale formations.} In this context, macropores---such as fractures and large voids---act as highly permeable flow channels, while micropores primarily serve as storage zones with limited fluid exchange. The two networks are connected through fissures or conduits that enable continuous mass transfer between them \citep{warren1963behavior}.
    %%%
    \item \textbf{Critical mineral extraction and exploration.} This application area, which is of increasing national interest, also involves dual-pore networks \citep{nrc2008critical}. Such systems are characterized by two contrasting hydraulic domains: fast, fracture-dominated pathways that deliver lixiviants and transport dissolved metals, and a fine-grained matrix in which valuable elements reside and are released gradually through pressure- and concentration-driven exchange \citep{steefel2005reactive,molins2015reactive,deng2011dual}. Recent studies have also begun to explore physics-informed machine learning approaches for reactive transport modeling in critical minerals applications; see, for example, \citet{adhikari2025reactive}. 
\end{enumerate}

Traditional models based on the classical Darcy law assume a single, homogeneous pore network and therefore cannot capture the dual-scale flow behavior and the associated mass transfer between the interconnected networks. Recognizing this limitation, the \emph{double porosity/permeability (DPP)} models have been developed to provide a more realistic framework. As described by \citet{nakshatrala2018modeling}, the DPP model not only accounts for the simultaneous flow within both the macro- and micro-pore networks but also rigorously incorporates the mass exchange between them. This theoretical advancement offers enhanced predictive capabilities for pressure and velocity fields in heterogeneous porous media.

Despite its conceptual elegance, the DPP model introduces significant challenges for numerical simulation. In particular, its complex structure makes it difficult to discretize and solve using conventional finite element methods, requiring advanced stabilization techniques and specialized solver strategies; see \citep{joshaghani2019stabilized, joshaghani2019composable}. Although several such conventional numerical formulations exist, there remains a strong need for a machine learning–based modeling framework for DPP systems (as illustrated conceptually in \textbf{Fig.~\ref{Fig2:APINNs_Modeling_approaches}}). This need arises for three main reasons:
\begin{enumerate}
    \item \textbf{Incomplete or difficult-to-obtain model components:} While mathematical models for DPP systems exist, many of their components---particularly model parameters---are not fully understood or are challenging to obtain in practice. For instance, permeability fields are especially difficult to characterize in subsurface environments, as they often vary spatially. Moreover, the mass transfer coefficient, a critical input in plug-scale DPP models, is typically not derived from first principles but instead assumed or calibrated. At the same time, alternative data sources---such as well logs, in situ sensors, acoustic emission data, or even satellite imagery---are often available and can be leveraged in conjunction with mathematical models. In such cases, where the goal is to fuse physical models with limited or indirect data, machine learning tools---particularly hybrid or physics-informed approaches---offer significant promise (see Fig.~\ref{Fig2:APINNs_Modeling_approaches}).
    \item \textbf{Inverse problem estimation:} Building on the previous point, a common objective is to infer hard-to-measure parameters---such as the spatial permeability distribution or interpore volumetric transfer rates---through inverse modeling. This is a natural fit for machine learning methods, which excel when paired with data-driven inverse modeling frameworks.
    \item \textbf{Need for rapid forecasting:} In time-sensitive applications such as critical mineral exploration, real-time monitoring and decision making require fast and reliable forecasts. Once trained, machine-learning models can provide rapid evaluations, making them attractive as surrogates for applications requiring repeated predictions. This advantage pertains to online inference; the preceding offline training may remain computationally expensive. The present study therefore does not claim an overall computational speedup relative to conventional numerical solvers.
\end{enumerate}
Motivated by these advantages, this paper presents a modeling framework based on a \emph{physics-informed neural network (PINN)} to solve DPP problems.

%---------------------------------;
%  Figure 2: Modeling approaches  ;
%---------------------------------;
\begin{figure}
    \centering
    \includegraphics[width=0.60\linewidth]{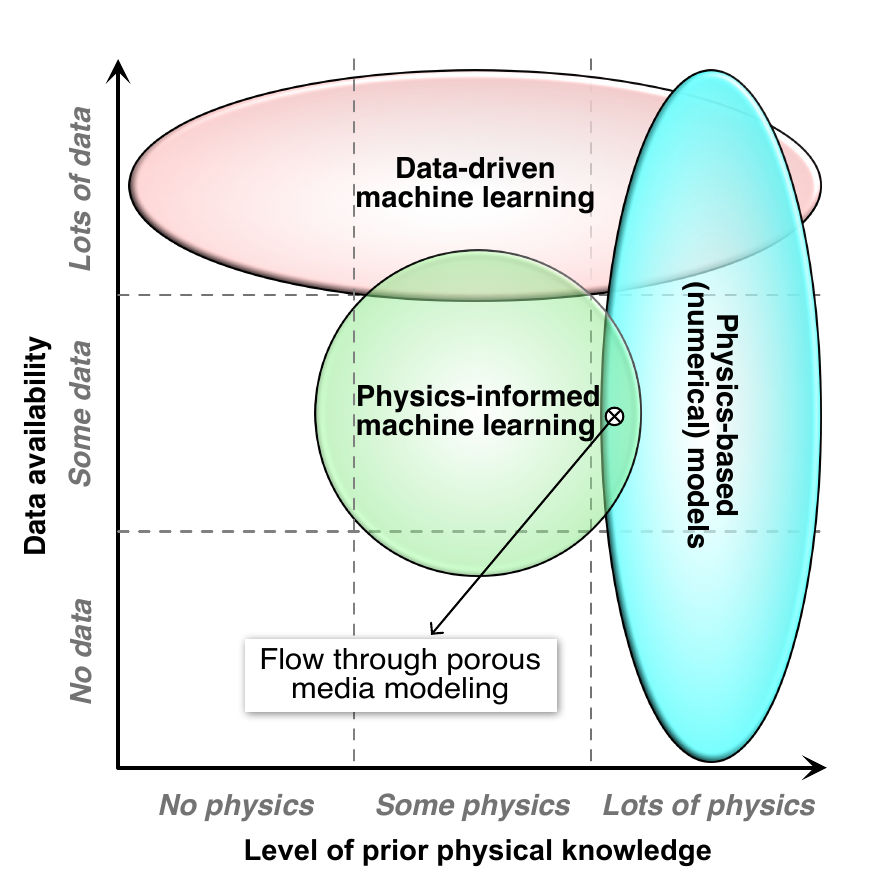}
    \caption{Modeling approaches across the spectrum from data-driven to physics-based methods, adapted from \citet{tartakovsky2020physics}. The position of flow through porous media modeling is highlighted, illustrating its potential to leverage physics-informed machine learning techniques.
    \label{Fig2:APINNs_Modeling_approaches}}
\end{figure}

%---------------------------------------------;
%  One paragraph on neural networks and PINN  ;
%---------------------------------------------;
Neural networks have evolved from early perceptron models to sophisticated deep learning systems capable of approximating complex functions. Building on decades of research---including the breakthrough backpropagation algorithm---these models have been applied to a wide array of challenges (e.g., \citep{goodfellow2016deep}). More recently, PINN-based methods have emerged as transformative methods for solving \emph{partial differential equations (PDEs)}. Their conceptual roots can be traced back to the pioneering work of \citet{dissanayake1994neural} and \citet{lagaris2000neural}, who first explored the use of neural networks for solving differential equations. PINNs, formally introduced and named by \citet{raissi2019physics}, provide a unified framework for embedding governing physical laws into neural-network training through PDE-residual terms in the loss function. This physics-informed formulation enables PINNs to approximate PDE solutions without requiring labeled solution data. By embedding governing physical laws into the optimization process, PINNs offer a compelling alternative to traditional numerical methods such as finite difference, finite element, and finite volume approaches---especially in high-dimensional and nonlinear problem domains. Numerous studies have highlighted the potential of PINNs in modeling complex phenomena such as fluid dynamics, heat transfer, and other systems where classical techniques face significant challenges \citep{raissi2020hidden,jagtap2020adaptive,kissas2020machine,mao2020physics,jagtap2023coolpinns}.

%-----------------------------------;
%  Challenges With classical PINNs  ;
%-----------------------------------;
Despite their conceptual elegance, classical PINNs face significant limitations in practical applications. Their performance deteriorates when applied to complex PDEs characterized by multiscale behavior, steep gradients, or discontinuities \citep{krishnapriyan2021characterizing}. This degradation often manifests as slow convergence, stagnating loss functions, and poor accuracy in regions with challenging physics. These issues have been attributed to several underlying factors: the spectral bias of neural networks toward smooth, low-frequency solutions; the ill-conditioning of PDE operators that leads to imbalanced gradient propagation during training; and the inefficiency of uniform sampling strategies, which fail to adequately resolve regions requiring fine-scale resolution \citep{wang2021when}.

These general challenges are further compounded in the case of DPP models, which introduce an additional layer of complexity due to the structure of their governing equations. Unlike standard Darcy-type models that can often be reduced to pressure-only formulations, DPP systems require the solution of coupled PDEs involving multiple interdependent fields—such as pressures and velocities across two distinct pore networks---governing mass conservation, momentum balance, and inter-pore exchange. Na\"ively assigning separate neural networks to each field variable not only increases the computational burden but also undermines the physical coherence of the model \citep{wang2021when}. This decoupling leads to inconsistent enforcement of constraints and deteriorated solution accuracy, particularly in regions where sharp gradients or multiscale dynamics are dominant \citep{krishnapriyan2021characterizing}.

To overcome the above challenges, the proposed modeling framework incorporates the following algorithmic advancements:
\begin{enumerate}
    %---------------------;
    %  Adaptive sampling  ;
    %---------------------;
    \item[(A1)] \textbf{Adaptive sampling:} To efficiently resolve regions requiring higher resolution, we employ adaptive sampling, which dynamically concentrates computational resources in areas with larger approximation errors. This approach has proven effective in addressing nonuniform solution behaviors, as demonstrated in previous works such as \citet{lu2021deepxde} and \citet{nabian2019deep}.
    %----------------------;
    %  Adaptive weighting  ;
    %----------------------;
    \item[(A2)] \textbf{Adaptive weighting:} While adaptive sampling improves the spatial distribution of training points, it does not account for the relative importance of different loss components. To address this, we adopt the normalized learning speed-based  adaptive weighting inspired by the work of \citet{gao2025physics}, which dynamically balances PDE-residual and boundary-condition terms by tracking their normalized learning speeds—thereby improving training efficiency and convergence.
    %----------------;
    %  Shared trunk  ;
    %----------------;
    \item[(A3)] \textbf{Shared trunk with slim head architecture:} 
    To effectively capture the coupled nature of the governing equations in the DPP system, we employ a shared trunk with slim head neural architecture. Originally introduced in the multi-task learning framework by \citet{caruana1997multitask} and further developed in works such as \citet{zou2023hydra}, this design consists of a common neural backbone that learns shared physical representations across all field variables. Lightweight, variable-specific output heads then branch from the trunk to predict distinct quantities such as pressure and velocity. This architecture reduces redundancy, maintains the inherent couplings between variables, and facilitates efficient training. Its effectiveness is further supported by recent applications of feedforward neural networks in phase-field fracture modeling---see, for instance, \citet{manav2024phase}---which highlight the advantages of shared representations in capturing complex, physics-driven interactions. In practice, when integrated with residual-based adaptive sampling and adaptive weighting strategies, the shared trunk architecture enables targeted learning in complex, multiscale regions and promotes balanced enforcement of all governing physical constraints.
\end{enumerate}

%---------------------------;
%  An outline of the paper  ;
%---------------------------;
Together, these advances provide a reliable and high-fidelity PINN-based framework for addressing complex DPP problems. Building on this foundation, the remainder of the paper is organized as follows. Section \ref{Sec:S2_APINNs_GE} introduces the governing equations defining the DPP model, which serve as the analytical basis for the subsequent developments. Section \ref{Sec:S3_APINNs_Framework} describes the proposed modeling framework grounded in scientific machine learning principles. Representative numerical results demonstrating the performance and capabilities of the framework are presented in Section \ref{Sec:S5_APINNs_NR}. Section \ref{Sec:S6_APINNs_IF} presents the inversion framework and illustrates its application through a representative boundary value problem for parameter identification. Finally, Section \ref{Sec:S7_APINNs_Closure} concludes the paper with a summary of key findings and a discussion of directions for future research.

    %*********************************************;
%                                             ;
%  NAME                                       ;
%    S2_APINNS_GE.tex                         ;
%                                             ;
%  WRITTEN BY                                 ;
%    Kalyana B. Nakshatrala                   ;
%                                             ;
%*********************************************;
\section{DOUBLE POROSITY/PERMEABILITY MODEL}
\label{Sec:S2_APINNs_GE}

We study the flow of an incompressible fluid through a rigid porous medium featuring two distinct, interconnected pore networks. These networks exchange mass via conduits or fissures, enabling flow between them. Such behavior can be effectively captured using a DPP model. In this work, we adopt the mathematical framework proposed by \cite{nakshatrala2018modeling}, hereafter referred to as the DPP model. We now introduce the necessary notation to present the governing equations of this model.
 
The porous medium is assumed to be a bounded domain $\Omega$  contained in $\mathbb{R}^{nd}$, where ``$nd$'' represents the spatial dimensions. The boundary $\partial\Omega$ is assumed to be piecewise smooth. Formally, $\partial \Omega = \overline{\Omega} \setminus \Omega$, where $\overline{\Omega}$ denotes the set closure of $\Omega$ \citep{evanspartial}. A spatial point is denoted by $\mathbf{x} \in \mathbb{R}^{nd}$. The corresponding spatial divergence and gradient operators are designated as $\mathrm{div}[\cdot]$ and $\mathrm{grad}[\cdot]$, respectively. The outward-directed unit normal vector at the boundary is symbolized by $\widehat{\mathbf{n}}(\mathbf{x})$. 

Henceforth, we label the two pore networks as the macro- and micro-pore networks, and designate the quantities associated with these two networks using parenthesized superscripts $1$ and $2$, respectively. The fluid pressures in the two networks are referred to as the macro- and micro-pressures, and are denoted by $p^{(1)}(\mathbf{x})$ and $p^{(2)}(\mathbf{x})$, respectively. The Darcy (or discharge) velocity vector fields in the macro- and micro-pore networks are represented by $\mathbf{u}^{(1)}(\mathbf{x})$ and $\mathbf{u}^{(2)}(\mathbf{x})$. The intrinsic density and viscosity of the fluid are denoted by $\gamma$ and $\mu$, respectively. For the two pore networks, the corresponding volume fractions are designated as $\phi^{(1)}(\mathbf{x})$ and $\phi^{(2)}(\mathbf{x})$. 

The permeability fields are denoted by $\mathbf{K}^{(1)}(\mathbf{x})$ and $\mathbf{K}^{(2)}(\mathbf{x})$, which are symmetric and, in general, anisotropic and spatially inhomogeneous. However, when the medium is isotropic and spatially homogeneous, the permeabilities take the following form:
%------------------------------------;
%  Equation: Isotropic permeability  ;
%------------------------------------;
\begin{align}
    \mathbf{K}^{(1)}(\mathbf{x}) = k^{(1)} \, \mathbf{I} 
    \quad \mathrm{and} \quad 
    \mathbf{K}^{(2)}(\mathbf{x}) = k^{(2)} \, \mathbf{I} 
\end{align}
where $k^{(1)}$ and $k^{(2)}$ are scalar constants, and $\mathbf{I}$ is the second-order identity tensor. For the mathematical analysis, we assume that the permeability tensors are uniformly elliptic and bounded. Specifically, there exist positive constants $0 < k_{\text{min}} \leq k_{\text{max}} < \infty$ such that
%-----------------------------------------------;
%  Equation: Uniformly elliptic permeabilities  ;
%-----------------------------------------------;
\begin{align}
    \label{Eqn:APIINS_permeabilities_uniformly_elliptic}
    k_{\text{min}} \, \mathbf{y} \bullet \mathbf{y}
    \leq \mathbf{y} \bullet \mathbf{K}^{(i)}(\mathbf{x}) \, \mathbf{y}
    \leq k_{\text{max}}  \, \mathbf{y} \bullet \mathbf{y}
    \quad \forall \mathbf{x} \in \Omega, 
    \forall \mathbf{y} \in \mathbb{R}^{nd}
\end{align}
where $i=1,2$.

For the macro-pore network, we partition the boundary into two parts: $\Gamma_{u}^{(1)}$ and $\Gamma_{p}^{(1)}$. $\Gamma_{u}^{(1)}$ designates that portion of the boundary on which the normal component of the macro-velocity is prescribed, while $\Gamma_{p}^{(1)}$ is the portion on which macro-pressure is prescribed. A similar notation holds for $\Gamma_{u}^{(2)}$ and $\Gamma_{p}^{(2)}$, but for the micro-pore network. For mathematical well-posedness, we require:
%--------------------------------;
%  Equation: Boundary partition  ;
%--------------------------------;
\begin{subequations}
    \label{Eqn:DEM_DPP_well_posedness}
    \begin{align}
        \Gamma_{u}^{(1)} \cup \Gamma_{p}^{(1)} = \partial\Omega \quad \text{and} \quad \Gamma_{u}^{(1)} \cap \Gamma_{p}^{(1)} = \emptyset\\
        \Gamma_{u}^{(2)} \cup \Gamma_{p}^{(2)} = \partial\Omega \quad \text{and} \quad \Gamma_{u}^{(2)} \cap \Gamma_{p}^{(2)} = \emptyset
    \end{align}
\end{subequations}

The governing equations for the DPP model take the following form:
%---------------------;
%  Equation: DPP BVP  ;
%---------------------;
\begin{subequations}
    \begin{alignat}{2}
        \label{Eqn:DEM_DPP_BoLM_1}
        &\mu\left(\mathbf{K}^{(1)}(\mathbf{x})\right)^{-1}\mathbf{u}^{(1)}(\mathbf{x}) 
        +  \mathrm{grad}\big[p^{(1)}(\mathbf{x})\big] 
        - \phi^{(1)}(\mathbf{x}) \, \gamma \, \mathbf{b}(\mathbf{x}) 
        = \mathbf{0}          
        &&  \qquad \mathrm{in} \; \Omega \\
        \label{Eqn:DEM_DPP_BoLM_2}
        &\mu\left(\mathbf{K}^{(2)}(\mathbf{x})\right)^{-1}\mathbf{u}^{(2)}(\mathbf{x}) 
        + \mathrm{grad}\big[p^{(2)}(\mathbf{x})\big] - \phi^{(2)}(\mathbf{x}) \, \gamma \, \mathbf{b}(\mathbf{x}) 
        = \mathbf{0}
        &&  \qquad \mathrm{in} \; \Omega \\
        \label{Eqn:DEM_DPP_BoM_1}
        &\mathrm{div}\big[\mathbf{u}^{(1)}(\mathbf{x})\big] 
        = +\chi(\mathbf{x}) = -\frac{\beta}{\mu} \, \big(p^{(1)}(\mathbf{x}) - p^{(2)}(\mathbf{x})\big)
        &&  \qquad \mathrm{in} \; \Omega \\
        \label{Eqn:DEM_DPP_BoM_2}
        &\mathrm{div}\big[\mathbf{u}^{(2)}(\mathbf{x})\big] = -\chi(\mathbf{x}) 
        = +\frac{\beta}{\mu} \, \big(p^{(1)}(\mathbf{x}) - p^{(2)}(\mathbf{x})\big)
        &&  \qquad \mathrm{in} \; \Omega \\
        \label{Eqn:DEM_DPP_v_BC_1}
        &\mathbf{u}^{(1)}(\mathbf{x}) \bullet \widehat{\mathbf{n}}(\mathbf{x}) = u^{(1)}_{n}(\mathbf{x}) 
       && \qquad \mathrm{on} \; \Gamma_{u}^{(1)} \\
       \label{Eqn:DEM_DPP_v_BC_2}
       &\mathbf{u}^{(2)}(\mathbf{x}) \bullet \widehat{\mathbf{n}}(\mathbf{x}) = u^{(2)}_{n}(\mathbf{x}) 
       && \qquad \mathrm{on} \; \Gamma_{u}^{(2)} \\
       \label{Eqn:DEM_DPP_p_BC_1}
       &p^{(1)}(\mathbf{x}) = p^{(1)}_{\mathrm{p}}(\mathbf{x}) 
       && \qquad \mathrm{on} \; \Gamma_{p}^{(1)} \\
       \label{Eqn:DEM_DPP_p_BC_2}
       &p^{(2)}(\mathbf{x}) = p^{(2)}_{\mathrm{p}}(\mathbf{x}) 
       && \qquad \mathrm{on} \; \Gamma_{p}^{(2)}  
    \end{alignat}
\end{subequations}
where $\mathbf{b}(\mathbf{x})$ is the specific body force, and $u^{(1)}_{n}(\mathbf{x})$ is the normal component of the Darcy velocity prescribed on the boundary of the macro-pore network. Similarly, $u^{(2)}_{n}(\mathbf{x})$ is the prescribed normal component of the velocity on the boundary of the micro-pore network. $p^{(1)}_{\mathrm{p}}(\mathbf{x})$ and $p^{(2)}_{\mathrm{p}}(\mathbf{x})$ are the pressures prescribed on the boundaries of macro- and micro-pore networks, respectively. 

$\chi(\mathbf{x})$ denotes the volumetric rate of fluid exchange, at the spatial point $\mathbf{x}$, from the micro-pore network to the macro-pore network, per unit volume of the porous medium---note the direction of the flow as indicated by the expressions given in Eqs.~\eqref{Eqn:DEM_DPP_BoM_1} and \eqref{Eqn:DEM_DPP_BoM_2}. From this point on, we will simply refer to $\chi(\mathbf{x})$ as the volumetric transfer rate. The rate of mass transfer at a spatial point $\mathbf{x}$ will then be 
\begin{align} 
    m(\mathbf{x}) = \gamma \, \chi(\mathbf{x})
\end{align}

The parameter $\beta$ denotes the inter-porosity mass transfer coefficient, which governs the rate of pressure-driven fluid exchange between the macro- and micro-pore networks and depends on the pore structure and material properties of the porous medium. For further details on the assumptions and derivation of the DPP model, refer to \citet{nakshatrala2018modeling}. 

A brief discussion on the mixed nature of the DPP model and its implications for solver design is warranted. For comparison, consider the classical Darcy equations. Although originally formulated in mixed form, these equations can be reformulated as a single-field diffusion equation expressed solely in terms of pressure. This transformation allows the application of well-established analytical solution methods \citep{strack2017analytical}, conventional computational solvers---such as the single-field pressure-based finite element formulation \citep{allen2021mathematics}, two-point flux approximation \citep{chen2006computational}, and multipoint flux approximation \citep{aavatsmark2002introduction} within the finite volume framework---and even machine learning-based solvers \citep{raissi2017physics, tartakovsky2020physics}, all tailored to diffusion-type problems. In contrast, the governing equations of the DPP model generally cannot be reduced to a pressure-only formulation. Instead, the model must be solved in its mixed form, involving four coupled solution fields. This key difference underscores the inherent complexity in the DPP model compared to the classical Darcy equations.
    %*********************************************;
%                                             ;
%  NAME                                       ;
%    S3_APINNs_Framework.tex                  ;
%                                             ;
%  WRITTEN BY                                 ;
%    Kalyana B. Nakshatrala                   ;
%                                             ;
%*********************************************;
\section{PROPOSED ADAPTIVE PINN-BASED MODELING FRAMEWORK}
\label{Sec:S3_APINNs_Framework}
To address the need for applying machine learning to flow in porous media---particularly those with dual pore networks---we propose a state-of-the-art PINN-based modeling framework. The DPP problem is inherently complex due to the strong interdependence among field variables, which rules out simple single-field formulations. Therefore, our framework embeds these couplings directly into a \emph{multi-field} PINN architecture, enabling the analysis and solution of coupled boundary value problems. To the best of our knowledge, this is the first framework to systematically integrate the essential couplings of DPP models into the network design, thereby improving both the accuracy and robustness of the solutions.

Our discussion in this section is structured as follows. We first introduce a general multi-field problem in an abstract setting, ensuring that the scope extends beyond the DPP model. At selected points, however, we revisit the DPP model to anchor the discussion with concrete examples. We then present the main elements of the proposed modeling framework, summarized in \textbf{Fig.~\ref{Fig:APINNs_NN_architecture}}:
%--------------------------;
%  Items in the framework  ;
%--------------------------;
\begin{enumerate}
    \item[(a)] A shared-trunk neural network architecture with task-specific slim heads, designed to efficiently capture the interactions among multiple fields.
    \item[(b)] Adaptive weighting and adaptive sampling strategies, introduced to enhance convergence and accurately resolve the coupled dynamics.
\end{enumerate}

%-----------------------------;
%  Figure 3: Shared trunk NN  ;
%-----------------------------;
\begin{figure}
    \centering
    \includegraphics[width=\linewidth]{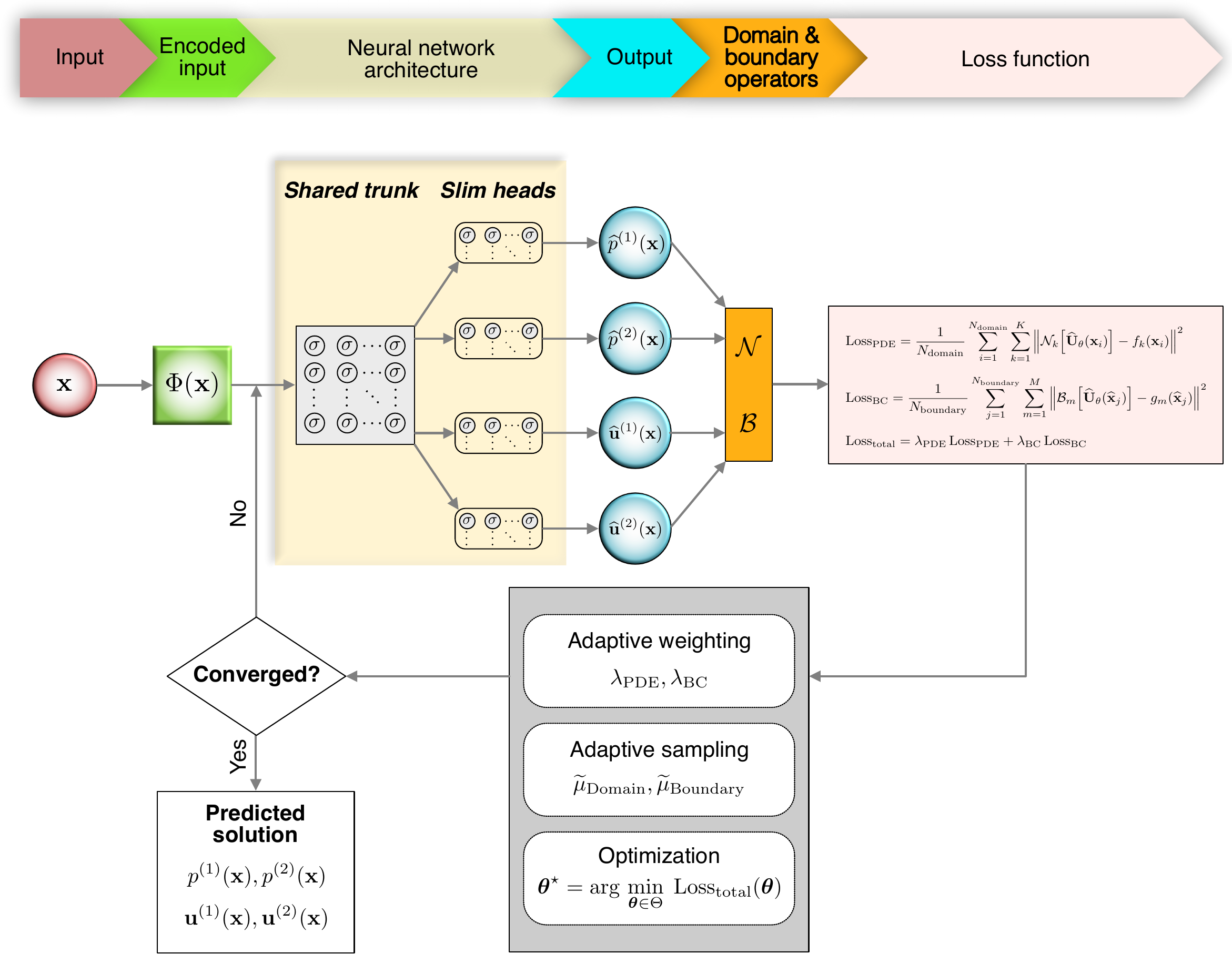}
    \caption{\textsf{Proposed adaptive modeling framework for the DPP model.} The architecture features a shared trunk, encoded inputs, and adaptive weighting and sampling strategies. \label{Fig:APINNs_NN_architecture}}
\end{figure}

%================================;
%  Subsection: Abstract setting  ;
%================================;
\subsection{Abstract setting}
Consider a mathematical model with $K$ field variables
\begin{align}
    u_k : \Omega \to \mathbb{R}^{m_k} \qquad k = 1, \dots, K
\end{align}
where $m_k \in \mathbb{N}$ specifies the dimension of the $k$-th field variable: $m_k=1$ corresponds to a scalar field, and $m_k = nd$ to a vector field with $nd$ components. For notational convenience, we define the ordered collection of all field variables as
%------------------------------------;
%  Equation: Set of trial functions  ;
%------------------------------------;
\begin{align}
    \label{Eqn:APINNS_set_of_trial_functions}
    \mathbb{U}(\mathbf{x}) = \big(u_1(\mathbf{x}),\, u_2(\mathbf{x}),\, \dots,\, u_K(\mathbf{x}) \big)
\end{align}

Let $\mathcal{U}_k$ denote the function space of $u_k$ (e.g., $L^2(\Omega)$ or $H^1(\Omega)$). The product space
\begin{align}
    \mathcal{U} = \prod_{k=1}^K \mathcal{U}_k
\end{align}
serves as the ambient space for the model, containing all tuples $\mathbb{U} = (u_1, \dots, u_K)$ with $u_k \in \mathcal{U}_k$.  
For each field variable, we associate a trace space $\mathcal{W}_m$, which allows solution fields defined in the (open) domain to be meaningfully extended to the boundary. The combined trace space is then
\begin{align}
    \mathcal{W} = \prod_{m=1}^M \mathcal{W}_m
\end{align}
containing all tuples of boundary traces $(w_1, \dots, w_M)$ with $w_m \in \mathcal{W}_m$.

%===================================;
%  Subsubsection: Extension to DPP  ;
%-----------------------------------;
\subsubsection{Contextualizing within the DPP model} 
We now interpret the abstract vector of unknown fields in terms of the physical variables of the DPP model:
%-----------------------------------;
%  Equation: DPP contextualization  ;
%-----------------------------------;
\begin{align}
    \mathbb{U}(\mathbf{x}) = \big(p^{(1)}(\mathbf{x}),\, p^{(2)}(\mathbf{x}),\, \mathbf{u}^{(1)}(\mathbf{x}),\, \mathbf{u}^{(2)}(\mathbf{x})\big)
\end{align}
where $p^{(1)}(\mathbf{x})$ and $p^{(2)}(\mathbf{x})$ denote the pressures in the macro- and micro-porous networks, respectively, and $\mathbf{u}^{(1)}(\mathbf{x})$ and $\mathbf{u}^{(2)}(\mathbf{x})$ are the corresponding velocity fields. In terms of field dimensions, we have $K = 4$ with
\begin{align}
    m_1 = 1, \quad m_2 = 1, \quad m_3 = n d, 
    \quad \mathrm{and} \quad m_4 = n d
\end{align}
Recall that ``$nd$'' represents the number of spatial dimensions.

The function space $\mathcal{U}$ for these fields is
%-------------------------------;
%  Equation: mathcal U for DPP  ;
%-------------------------------;
\begin{align}
    \mathcal{U} = H^{1}(\Omega) \times H^{1}(\Omega) \times H(\text{div};\Omega) \times H(\text{div};\Omega)
\end{align}
where the first two components correspond to the pressure fields and the last two to the velocity fields. Specifically, $p^{(1)}(\mathbf{x})$ and $p^{(2)}(\mathbf{x})$ lie in the Sobolev space $H^1(\Omega)$, which consists of square-integrable functions with square-integrable first derivatives, ensuring well-defined pressure gradients for flux computations. The velocity fields, $\mathbf{u}^{(1)}(\mathbf{x})$ and $\mathbf{u}^{(2)}(\mathbf{x})$, belong to the divergence-conforming space $H(\text{div};\Omega)$, containing vector fields whose components and divergence are square-integrable. This ensures that local mass conservation in each network is satisfied and that the divergence of velocities is properly defined.

The combined trace space for the DPP model is defined as
%----------------------------------;
%  Equation: Combined trace space  ;
%----------------------------------;
\begin{align}
    \mathcal{W} = H^{1/2}(\partial \Omega) 
    \times H^{1/2}(\partial \Omega) 
    \times H^{-1/2}(\partial \Omega) 
    \times H^{-1/2}(\partial \Omega)
\end{align}
where the first two components correspond to pressure traces and the last two to velocity traces. For velocity fields, the boundary condition is specified in terms of the normal component
\[
\mathbf{u}^{(i)}(\mathbf{x}) \bullet \widehat{\mathbf{n}}(\mathbf{x})
\]
which represents the flux through the boundary. This normal component naturally lies in $H^{-1/2}(\partial \Omega)$, the dual space of $H^{1/2}(\partial \Omega)$, ensuring that the flux is well-defined in the variational sense. 

For detailed discussions on functional spaces in the context of the DPP model, see \cite{joodat2018modeling}. For a broader introduction to functional analysis and related topics, refer to standard texts such as \citet{brezis2011functional, evanspartial, brezzi2012mixed, adams2003sobolev}.

%-----------------------------------;
%  Remark: Validity of weak spaces  ;
%-----------------------------------;
\begin{remark}
    The above functional spaces are defined in the weak sense, which remains consistent with the PINN-based framework. Although the PINN enforces the governing equations through pointwise collocation of the strong-form residuals, the neural network approximation generally attains Sobolev-like regularity (e.g., $H^1(\Omega)$) \citep{wang2024piratenets, katende2025stability}. This regularity allows the solution to be interpreted naturally within a weak framework. Moreover, the weak formulation mirrors the underlying variational principles of the governing equations, ensuring a mathematically consistent interpretation, as elaborated in the convergence analysis.
\end{remark}

%===============================;
%  Subsubsection: Abstract BVP  ;
%-------------------------------;
\subsubsection{Abstract boundary value problem}
We recast the boundary value problem in the following abstract setting: 
\begin{subequations}
    \begin{alignat}{2}
        % Abstract GE 
        \label{Eqn:APINNS_Abstarct_GE}
        &\mathcal{N}[\mathbb{U}(\mathbf{x})] = \mathbf{f} &&\quad \text{in} \,\, \Omega\\
        % Abstract BC 
        \label{Eqn:APINNS_Abstarct_BC}
        &\mathcal{B}[\mathbb{U}(\mathbf{x})] = \mathbf{g} &&\quad \text{on} \,\, \partial\Omega
    \end{alignat}
\end{subequations}

The operator $\mathcal{N}$ is vector-valued can be decomposed as follows:
\begin{align}
      \mathcal{N}[\mathbb{U}(\mathbf{x})] 
      = \begin{bmatrix}
      \mathcal{N}_1\big(u_1(\mathbf{x}), \cdots, u_K(\mathbf{x})\big) \\
      \mathcal{N}_2\big(u_1(\mathbf{x}), \cdots, u_K(\mathbf{x})\big) \\
      \vdots \\
      \mathcal{N}_K\big(u_1(\mathbf{x}), \cdots, u_K(\mathbf{x})\big)
  \end{bmatrix}
\end{align}
Typically, $\mathcal{N}_k$ can be expressed as follows:
\begin{align}
    \mathcal{N}_{k}[\mathbb{U}(\mathbf{x})] =  \underbrace{\sum_{j=1}^K \mathcal{A}_{kj}(u_j)}_{\mathcal{L}_k(\mathbb{U})} + \mathcal{F}_k(\mathbb{U})
\end{align}
where $\mathcal{L}_k(\mathbb{U})$ denotes the linear contribution, with differential operator $\mathcal{A}_{kj}$, while $\mathcal{F}_k$ represents the nonlinear/coupling part.  The source term $\mathbf{f}$ is given as $\mathbf{f} = (f_1, \dots, f_K)$.

Likewise, the boundary operator $\mathcal{B}$ can be decomposed as follows:
\begin{align}
      \mathcal{B}[\mathbb{U}(\mathbf{x})] 
      = \begin{bmatrix}
  \mathcal{B}_1\big(u_1(\mathbf{x}),\cdots,u_K(\mathbf{x})\big) \\
  \vdots \\
  \mathcal{B}_M\big(u_1(\mathbf{x}),\cdots,u_K(\mathbf{x})\big)
  \end{bmatrix}
\end{align}
Each $\mathcal{B}_m$ ($m = 1, \cdots, M)$ represents a boundary condition on a portion of $\partial\Omega$.  For example, if a Dirichlet boundary condition is imposed on $u_k$ at $\partial \Omega_m$, then $\mathcal{B}_m(\mathbb{U}) = u_k|_{\partial \Omega_m}$. For a Neumann boundary condition, $\mathcal{B}_m(\mathbb{U}) = \frac{\partial u_k}{\partial \widehat{\mathbf{n}}}|_{\partial \Omega_m}$, where $\widehat{\mathbf{n}}$ is the outward normal. 

%===========================================;
%  Subsection: Neural network architecture  ;
%===========================================;
\subsection{Neural network architecture}
Having established the abstract boundary value problem in Eqs.~\eqref{Eqn:APINNS_Abstarct_GE}--\eqref{Eqn:APINNS_Abstarct_BC}, we next introduce a neural representation of the solution fields that respects both the product structure and the regularity requirements of the problem.

%============================================;
%  Subsection: Neural network approximation  ;
%--------------------------------------------;
\subsubsection{Neural network approximation} 
We then define a neural network surrogate:
\begin{align}
\label{Eqn:NN_approximation}
    \hat{\mathbb{U}}_{\boldsymbol{\theta}}:\Omega \to \prod_{k=1}^{K}\mathbb{R}^{m_k}
\end{align}
parameterized by the trainable vector $\boldsymbol{\theta}$. The neural network $\hat{\mathbb{U}}_{\boldsymbol{\theta}}$ is designed to be smooth enough so that all derivatives required by Eqs.~\eqref{Eqn:APINNS_Abstarct_GE}--\eqref{Eqn:APINNS_Abstarct_BC} can be computed. 

The parameters $\boldsymbol{\theta}$ are obtained by minimizing a physics-informed objective that combines the residuals of the governing equations with deviations from the prescribed boundary conditions. When available, observational data can be incorporated through additional data-fidelity terms, and regularization on $\boldsymbol{\theta}$ can be included without altering the formulation. Automatic differentiation of $\hat{\mathbb{U}}_{\boldsymbol{\theta}}$ yields the required differential operators for end-to-end training while preserving inter-field couplings across the product space $\prod_{k=1}^{K}\mathbb{R}^{m_k}$.

Smooth coordinate-to-field surrogates, as in Eq.~\eqref{Eqn:NN_approximation}, are known to exhibit a low–frequency bias, whereas PDE solutions often display multi–scale structures, such as sharp layers, boundary or interior layers, and heterogeneous coefficients. Direct optimization in the ambient input space can therefore be ill-conditioned and sample-inefficient. To address this, we introduce a \emph{feature encoder} that lifts the inputs into a higher-dimensional, multi–scale feature space, where both high- and low-frequency content are simultaneously accessible to a moderate-size network. This transformation preserves the differentiability required by Eqs.~\eqref{Eqn:APINNS_Abstarct_GE}--\eqref{Eqn:APINNS_Abstarct_BC}, while improving training efficiency and accuracy for complex PDE solutions.

%==========================================;
%  Subsubsection: Fourier feature encoder  ;
%==========================================;
\subsubsection{Fourier feature encoder}
Multiple strategies have been proposed to enrich the representation space for PDE learning, ranging from \emph{data-driven latent manifolds} constructed from simulation data to \emph{coordinate-based embeddings} defined directly on the input domain. In the data-driven setting, \citet{kim2022fast} demonstrated that autoencoder architectures can construct latent manifolds that act as frequency-like bases learned from simulations. In this work, following the coordinate-based paradigm, we employ a randomized Fourier feature mapping
\begin{equation}
    \Phi: \mathbb{R}^{d} \to \mathbb{R}^{D}
\end{equation}
where $d$ denotes the dimensionality of the physical input domain (e.g., spatial or spatiotemporal coordinates), and $D$ represents the higher-dimensional embedding space that captures richer frequency information---as introduced by \citet{tancik2020fourier}.

Let \(\mathbf{L}=(L_1,\dots,L_d)\) denote the characteristic domain scales along each coordinate. For any \(\mathbf{x}\in\Omega\), we first normalize
\begin{align}
    \hat{\mathbf{x}} = \Big(\frac{x_1}{L_1}, \dots, \frac{x_d}{L_d}\Big)
\end{align}
We then draw a bank of \(n_{\mathrm{freq}}\) random frequencies \(\mathbf{B}\in\mathbb{R}^{n_{\mathrm{freq}}\times d}\) with independent entries \(B_{ij}\sim\mathcal{N}(0,\tau^2)\) (i.e., a normal distribution with zero mean and $\tau$ standard deviation) and define the phase vector:
\begin{align}
    \boldsymbol{\varphi}(\hat{\mathbf{x}}) = 2\pi\,\mathbf{B}\,\hat{\mathbf{x}} \;\in\; \mathbb{R}^{n_{\mathrm{freq}}}
\end{align}
The encoder \(\Phi\) maps \(\mathbf{x}\) to the concatenation of low- and high-frequency coordinates:
\begin{align}
    \Phi(\mathbf{x}) = \Big[\;\hat{\mathbf{x}} \;;\; \sin(\boldsymbol{\varphi}(\hat{\mathbf{x}})) \;;\; \cos(\boldsymbol{\varphi}(\hat{\mathbf{x}})) \;\Big] \;\in\; \mathbb{R}^{D}
\end{align}
where 
\begin{align}
    D = d + 2\,n_{\mathrm{freq}}
\end{align}

For planar problems (\(d=2\)), the parameter \(\tau>0\) controls the frequency spread, and we retain the raw $\hat{\mathbf{x}}$ block (``include input'') to stabilize low-frequency trends and improve conditioning.\footnote{Fourier encoding $\Phi(\mathbf{x})$ mitigates the low-frequency (``spectral'') bias of smooth MLPs by injecting multiscale $\sin/\cos$ features. This allows a moderate trunk $\mathcal{T}$ to capture sharp layers and fine gradients. In the shared-trunk/slim-heads design, these multiscale features are learned once by $\mathcal{T}$ and then reused by each head $\mathcal{H}_k$, improving sample efficiency and cross-field consistency.}

%=========================================;
%  Subsection: Shared-trunk architecture  ;
%=========================================;
\subsubsection{Shared-trunk architecture with slim heads}
With the Fourier encoder $\Phi(\mathbf{x})$ established, we employ a shared-trunk neural architecture designed to construct a unified latent representation of the domain, accessible to all field variables. Task-specific lightweight heads then map this latent representation to their respective outputs, reducing the number of trainable parameters and the computational cost of derivatives, while simultaneously promoting consistency and coherence across the different fields. Unlike training independent networks for each field, this architecture leverages shared structural information, improving data efficiency and generalization.

To formalize the shared-trunk component of the architecture, consider the encoded input 
\(\Phi(\mathbf{x}) \in \mathbb{R}^{D}\). The trunk maps this input to a latent vector 
\(\mathbf{z} \in \mathbb{R}^{H}\) through a depth-\(L\) composition:
\begin{align}
    \mathcal{T}:\mathbb{R}^{D} \to \mathbb{R}^{H}
    \quad \mathrm{where} \quad 
    \mathcal{T} = \mathcal{T}_{L} \circ \sigma \circ \mathcal{T}_{L-1} \circ \cdots \circ \sigma \circ \mathcal{T}_{1}
\end{align}
Each layer \(\mathcal{T}_{\ell} : \mathbb{R}^{H_{\ell-1}} \to \mathbb{R}^{H_{\ell}}\) is affine:
\begin{align}
    \mathcal{T}_{\ell}(\mathbf{v}) = \mathbf{W}_{\ell} \mathbf{v} + \mathbf{b}_{\ell}
\end{align}
where \(\mathbf{v}\) is the input to layer \(\ell\), \(\mathbf{W}_{\ell} \in \mathbb{R}^{H_{\ell} \times H_{\ell-1}}\), and \(\mathbf{b}_{\ell} \in \mathbb{R}^{H_{\ell}}\). The dimensions satisfy \(H_0 = D\) and \(H_\ell = H\) for \(\ell = 1,\dots,L\). An element-wise activation \(\sigma : \mathbb{R} \to \mathbb{R}\) is applied between layers, with no activation after \(\mathcal{T}_L\) unless explicitly stated. The resulting shared latent state is
\begin{align}
    \mathbf{z} = \mathcal{T}\big(\Phi(\mathbf{x})\big) \in \mathbb{R}^{H}
\end{align}

For each field \(u_k(\mathbf{x})\), \(k=1,\dots,K\), a slim head \(\mathcal{H}_k : \mathbb{R}^{H} \to \mathbb{R}^{m_k}\) maps \(\mathbf{z}\) to the codomain of the \(k\)-th field:
\begin{align}
    \mathcal{H}_k(\mathbf{z}) = \mathbf{W}_k^{\text{head}} \mathbf{z} + \mathbf{b}_k^{\text{head}}
\end{align}
with \(\mathbf{W}_k^{\text{head}} \in \mathbb{R}^{m_k \times H}\) and \(\mathbf{b}_k^{\text{head}} \in \mathbb{R}^{m_k}\). This slim-head design minimizes additional parameters and enables efficient derivative evaluation.

The neural network approximation of the field variables is given by
\begin{align}
    \label{Eqn:APINNS_NN_approximation_Trunk}
    \hat{\mathbb{U}}_{\boldsymbol{\theta}}(\mathbf{x})
    = \big(\hat{u}_{1}(\mathbf{x}), \ldots, \hat{u}_{K}(\mathbf{x})\big)
\end{align}
where each \(\hat{u}_k\) is obtained by applying a head network \(\mathcal{H}_k\) to the shared-trunk representation \(\mathcal{T}(\Phi(\mathbf{x}))\):
\begin{align}
    \label{Eqn:APINNS_NN_approximation_Trunk_Component}
    \hat{u}_{k}(\mathbf{x}) = \mathcal{H}_{k}\!\big(\mathcal{T}(\Phi(\mathbf{x}))\big)
\end{align}

The set of trainable parameters, \(\boldsymbol{\theta}\), includes all weights and biases from both the trunk \(\mathcal{T}\) and the heads \(\{\mathcal{H}_{k}\}_{k=1}^{K}\):
\begin{align}
    \boldsymbol{\theta} = 
    \Big(\{\mathbf{W}_{\ell}, \mathbf{b}_{\ell}\}_{\ell=1}^{L}, \ 
          \{\mathbf{W}_{k}^{\text{head}}, \mathbf{b}_{k}^{\text{head}}\}_{k=1}^{K}\Big)
\end{align}
By concatenating the outputs into a single vector 
\(\hat{\mathbf{u}}(\mathbf{x}) \in \mathbb{R}^{m_{\mathrm{tot}}}\), 
the network output can be written in a compact affine form:
\begin{align}
    \hat{\mathbf{u}}(\mathbf{x}) = \mathbf{A}\,\mathbf{z}(\mathbf{x}) + \mathbf{b}
\end{align}
where \(\mathbf{z}(\mathbf{x}) = \mathcal{T}(\Phi(\mathbf{x}))\),  
\(\mathbf{A} = \mathrm{blkdiag}(\mathbf{W}_{1}^{\text{head}}, \dots, \mathbf{W}_{K}^{\text{head}})\) is a block-diagonal matrix of head weights, and 
\(\mathbf{b} = (\mathbf{b}_{1}^{\text{head}}, \dots, \mathbf{b}_{K}^{\text{head}})\) collects the biases.  
Finally, the overall dimension of the network output, accounting for all fields, is
\begin{align}
    m_{\mathrm{tot}} = \sum_{k=1}^{K} m_{k}
\end{align}

%========================================================;
%  Subsubsection: Discretization and training objective  ;
%--------------------------------------------------------;
\subsubsection{Discretization and training objective} 
As mentioned earlier, the neural network defined in Eq.~\eqref{Eqn:APINNS_NN_approximation_Trunk} belongs to a finite-dimensional subspace $\mathcal{U}_{\boldsymbol{\theta}} \subset \mathcal{U}$, where $\mathcal{U}_{\boldsymbol{\theta}}$ denotes the space of functions representable by the network architecture parameterized by $\boldsymbol{\theta}$. The objective is to determine $\boldsymbol{\theta}$ such that the network approximation $\hat{\mathbb{U}}_{\boldsymbol{\theta}}(\mathbf{x})$ closely matches the true solution $\mathbb{U}(\mathbf{x}) \in \mathcal{U}$ by satisfying the following abstract boundary value problem:
\begin{subequations}
    \begin{alignat}{2}
        \label{Eqn:APINNS_NN_Abstarct_GE}
        &\mathcal{N}[\hat{\mathbb{U}}_{\boldsymbol{\theta}}(\mathbf{x})] = \mathbf{f}(\mathbf{x}) 
        &&\quad \text{in } \Omega \\
        \label{Eqn:APINNS_NN_Abstarct_BC}
        &\mathcal{B}[\hat{\mathbb{U}}_{\boldsymbol{\theta}}(\mathbf{x})] = \mathbf{g}(\mathbf{x}) 
        &&\quad \text{on } \partial\Omega
    \end{alignat}
\end{subequations}

For training, a finite set of collocation points is defined as follows:  
$\{ \mathbf{x}_i \}_{i=1}^{N_{\text{domain}}} \subset \Omega$ are points within the domain, and  
$\{ \mathbf{x}_j \}_{j=1}^{N_{\text{boundary}}} \subset \partial\Omega$ are points on the boundary.  

Using these collocation points, the loss function is constructed based on the governing equations given in Eqs.~(\ref{Eqn:APINNS_NN_Abstarct_GE})--(\ref{Eqn:APINNS_NN_Abstarct_BC}):
\begin{align}
    \label{Eqn:Abstarct_Total_Loss}
    \mathrm{Loss}_{\text{total}}(\boldsymbol{\theta}) 
    &= \lambda_{\text{PDE}} \, \mathrm{Loss}_{\text{PDE}}(\boldsymbol{\theta})
    + \lambda_{\text{BC}} \, \mathrm{Loss}_{\text{BC}}(\boldsymbol{\theta})
\end{align}
where $\lambda_{\text{PDE}}$ and $\lambda_{\text{BC}}$ are positive weighting factors that balance the contributions of the PDE and boundary residuals, respectively. These contributions are defined as
\begin{subequations}
    \begin{alignat}{2}
        \label{Eqn:APINNS_NN_Abstarct_GE_Loss}
        &\mathrm{Loss}_{\text{PDE}}(\boldsymbol{\theta}) 
        = \frac{1}{N_{\text{domain}}} 
        \sum_{i=1}^{N_{\text{domain}}} 
        \sum_{k=1}^{K} 
        \big\| \mathcal{N}_{k}[\hat{\mathbb{U}}_{\boldsymbol{\theta}}(\mathbf{x}_i)] - f_k(\mathbf{x}_i) \big\|^2 \\
        \label{Eqn:APINNS_NN_Abstarct_BC_Loss}
        &\mathrm{Loss}_{\text{BC}}(\boldsymbol{\theta}) 
        = \frac{1}{N_{\text{boundary}}} 
        \sum_{j=1}^{N_{\text{boundary}}} 
        \sum_{m=1}^{M} 
        \big\| \mathcal{B}_{m}[\hat{\mathbb{U}}_{\boldsymbol{\theta}}(\mathbf{x}_j)] - g_m(\mathbf{x}_j) \big\|^2
    \end{alignat}
\end{subequations}
The differential operators $\mathcal{N}$ and $\mathcal{B}$ are evaluated through automatic differentiation~\citep{paszke2017automatic, ketkar2021deep}.  

Given the neural network approximation in Eq.~\eqref{Eqn:NN_approximation} and the total loss in Eq.~\eqref{Eqn:Abstarct_Total_Loss}, the next step is to optimize the trainable parameters $\boldsymbol{\theta}$---comprising all weights and biases of the shared-trunk $\mathcal{T}$ and the task-specific heads $\{ \mathcal{H}_k \}_{k=1}^{K}$---to minimize the total loss. This optimization reduces the residuals of the governing PDEs and boundary conditions at the collocation points. The parameters $\boldsymbol{\theta}$ lie in the parameter space $\Theta \subseteq \mathbb{R}^p$, where $p$ denotes the total number of trainable weights and biases. Optimization can be performed using techniques such as gradient descent, momentum-based stochastic gradient descent, or quasi-Newton methods.

%==============================;
%  Subsubsection: DPP framing  ;
%------------------------------;
\subsubsection{Framing in the DPP context} Before establishing the optimization procedure, we frame the loss function under the DPP model (i.e., Eqs.~\eqref{Eqn:DEM_DPP_BoLM_1}--\eqref{Eqn:DEM_DPP_p_BC_2}). The loss due to the partial differential equations (excluding the boundary conditions) takes the following form: 
\begin{equation}
\mathrm{Loss}_{\text{PDE}}(\boldsymbol{\theta}) = \frac{1}{N_{\text{domain}}} \sum_{i=1}^{N_{\text{domain}}} 
\Big( 
    \|R_{1}(\mathbf{x}_i)\|_2^2 
  + \|R_{2}(\mathbf{x}_i)\|_2^2 
  + |R_{3}(\mathbf{x}_i)|^2 
  + |R_{4}(\mathbf{x}_i)|^2 
\Big)
\end{equation}
where $\|\cdot\|_{2}$ denotes the 2-norm. The individual residuals are: 
\begin{subequations}
    \begin{align}
        R_{1}(\mathbf{x}) &= \mu \, \mathbf{K}^{-1}_{1}\hat{\mathbf{u}}^{(1)}(\mathbf{x}) + \text{grad}\, [\hat{p}^{(1)}(\mathbf{x})] - \phi^{(1)}(\mathbf{x}) \, \gamma \, \mathbf{b}(\mathbf{x}) \\
       R_{2}(\mathbf{x}) &= \mu \, \mathbf{K}^{-1}_{2}\hat{\mathbf{u}}^{(2)}(\mathbf{x}) + \text{grad}\, [\hat{p}^{(2)}(\mathbf{x})] - \phi^{(2)}(\mathbf{x}) \, \gamma \, \mathbf{b}(\mathbf{x}) \\
        R_{3}(\mathbf{x}) &= \text{div}\, [\hat{\mathbf{u}}^{(1)}(\mathbf{x})] + \frac{\beta}{\mu} \left( \hat{p}^{(1)}(\mathbf{x}) - \hat{p}^{(2)}(\mathbf{x}) \right) \\
       R_{4}(\mathbf{x}) &= \text{div}\, [\hat{\mathbf{u}}^{(2)}(\mathbf{x})] -\frac{\beta}{\mu} \left( \hat{p}^{(1)}(\mathbf{x}) - \hat{p}^{(2)}(\mathbf{x}) \right)  
    \end{align}
\end{subequations}
The loss function due to the boundary conditions takes the following form:
\begin{align}
    \mathrm{Loss}_{\text{BC}}(\boldsymbol{\theta}) 
    &= \frac{1}{N_{\text{boundary}}} \sum_{j=1}^{N_{\text{boundary}}} \Bigg( 
    \sum_{m \in \mathcal{M}_{D}} \left[\left( \hat{p}^{(1)}(\mathbf{x}_j) 
    - g_{m, p^{(1)}}(\mathbf{x}_j) \right)^2 + \left( \hat{p}^{(2)}(\mathbf{x}_j) 
    - g_{m, p^{(2)}}(\mathbf{x}_j) \right)^2 \right] \nonumber \\
    &\hspace{0.5in} + \sum_{m \in \mathcal{M}_N} \left[ 
    \left( \hat{\mathbf{u}}^{(1)}(\mathbf{x}_j) \bullet \widehat{\mathbf{n}}_j 
    - g_{m,u^{(1)}} \right)^2 
    + \left( \hat{\mathbf{u}}^{(2)}(\mathbf{x}_j) \bullet \widehat{\mathbf{n}}_j 
    - g_{m, u^{(2)}}\right)^2 \right]\Bigg)
\end{align}

%========================================;
%  Subsubsection: Optimization strategy  ;
%----------------------------------------;
\subsubsection{Optimization strategy}
The optimization problem reads:
%----------------------------------;
%  Equation: Optimization problem  ;
%----------------------------------;
\begin{align}
    \label{Eqn:NN_Optimization_def}
    \boldsymbol{\theta}^{\star} = \arg\min_{\boldsymbol{\theta} \in \Theta} 
    \; \; \mathrm{Loss}_{\text{total}}(\boldsymbol{\theta})
\end{align}
The objective is to determine the optimal parameter vector $\boldsymbol{\theta}^{\star}$ that minimizes the total loss function $\mathrm{Loss}_{\mathrm{total}}(\boldsymbol{\theta})$. The optimization problem is solved as a two-stage procedure: We initially employ the Adam optimizer---a first-order method that adapts the learning rate using exponential moving averages of gradients and their squares \citep{kingma2014adam}. Following the Adam optimizer, the L-BFGS algorithm---a quasi-Newton method---is employed as a polishing step to refine the solution \citep{seo2024solving}.

%===================================;
%  Subsection: Adaptive strategies  ;
%===================================;
\subsection{Adaptive strategies}
The framework employs two adaptive strategies that enhance optimization by dynamically adjusting loss weights and collocation points. These methods balance PDE and boundary losses, and concentrate computation on high-residual regions, improving efficiency in complex multi-field systems.

%===================================;
%  Subsection: Adaptive strategies  ;
%-----------------------------------;
\subsubsection{Adaptive weighting}  
We combine the task-specific losses using initial weights $\lambda_j$ to define the total loss:
\begin{align}
    \text{Loss}_{\text{total}}(\boldsymbol{\theta}) = \sum_{j \in \mathcal{A}} \lambda_j \, \text{Loss}_j(\boldsymbol{\theta})
\end{align}
where $\boldsymbol{\theta}$ denotes the trainable parameters, $\mathcal{A}$ is the set of tasks (i.e., PDE, BC), and $\lambda_j>0$ are adaptive weights.

To balance optimization across tasks, these weights are adjusted dynamically based on the relative convergence rates of each task. The fractional loss reduction for task $j$ at iteration $n$ is defined as
\begin{align}
\label{Eqn: APINNS_Convergence_rates}
    S_j^{(n)} = \max\left(\frac{\mathrm{Loss}_j(\boldsymbol{\theta}_{n-1}) - \mathrm{Loss}_j(\boldsymbol{\theta}_n)}{\mathrm{Loss}_j(\boldsymbol{\theta}_{n-1}) + \delta},\, 0\right)
\end{align}
where a small stabilizer $\delta = 10^{-12}$ is included. Intuitively, $S_j^{(n)} = 0$ indicates a non-monotone step (i.e., the loss increased), which is counted as zero progress.
To reduce stochastic noise, these convergence rates are smoothed using a Simple Moving Average (SMA):
\begin{align}
    \left(\tilde{S}_j^{\text{SMA}}\right)^{(n)} = \frac{1}{W} \sum_{i=n-W+1}^n S_j^{(i)}
\end{align}
where $W$ is the SMA window size. We then compute the maximum and minimum averaged convergence rates across tasks:
\begin{align}
    s_{\max}^{(n)} = \max_{j \in \mathcal{A}} \tilde{S}_j^{(n)}
    \quad \mathrm{and} \quad 
    s_{\min}^{(n)} = \min_{j \in \mathcal{A}} \tilde{S}_j^{(n)}
\end{align}  

The reweighting mechanism is activated only after each task has accumulated a full window of $W$ samples and the spread of convergence rates is significant:
\begin{align}
    \frac{s_{\max}^{(n)}}{\max(s_{\min}^{(n)}, \varepsilon)} > \rho
\end{align}  
where $\varepsilon>0$ prevents division by zero and $\rho$ is a user-chosen threshold. Once activated, the weight for each task $j \in \mathcal{A}$ is updated as
\begin{align}
    \lambda_j^{(n+1)} = 1 + \alpha \, \frac{s_{\max}^{(n)} - \tilde{S}_j^{(n)}}{s_{\max}^{(n)} - s_{\min}^{(n)}}
\end{align}
with $\alpha > 0$ a scaling parameter, which is specified by the user. 
Tasks that achieve the maximal averaged convergence rate retain the baseline weight $\lambda_j^{(n+1)} = 1$. Tasks with slower convergence receive larger weights:
\begin{align}
    1 < \lambda_j^{(n+1)} \le 1 + \alpha
\end{align}
with the lower bound reached when $\tilde{S}_j^{(n)} = s_{\max}^{(n)}$ and the upper bound $\lambda_j^{(n+1)} = 1 + \alpha$ when $\tilde{S}_j^{(n)} = s_{\min}^{(n)}$.

The computer implementation is provided in the appendix (\S\ref{Subsec:APINNS_Adaptive_weighting}).

%====================================;
%  Subsubsection: Adaptive sampling  ;
%====================================;
\subsubsection{Adaptive sampling} 
Adaptive sampling is a round-end enrichment strategy that augments the collocation set by adding points in regions where the current network approximation exhibits large residuals. Training proceeds in refinement rounds: at the end of each round $r$, we take the current approximation $\hat{\mathbb{U}}_{\boldsymbol{\theta}^{(r)}} $ and draw a random pool of candidates from the domain and the boundary. We then evaluate the pointwise residual of the training objective at these candidates, rank them, admit a fixed number of top-scoring points under a simple capacity cap, and continue optimization on the enlarged collocation set.

Formally, the interior and boundary collocation sets at round $r$ are denoted by $\mathcal{X}^{(r)}_{\mathrm{int}} \subset \Omega$ and $\mathcal{X}^{(r)}_{\mathrm{bnd}} \subset \partial \Omega$. The corresponding residual indicators are defined as
\begin{align}
    &\eta_{\mathrm{int}}^{(r)}(\mathbf{x}) 
    = \Bigg(\sum_{k=1}^K \big\|\mathcal{N}_k[\hat{\mathbb{U}}_{\boldsymbol{\theta}^{(r)}}](\mathbf{x}) - f_k(\mathbf{x}) \big\|_2^2 \Bigg)^{1/2} 
    &&\mathbf{x} \in \Omega \\
    &\eta_{\mathrm{bnd}}^{(r)}(\mathbf{x}) 
    = \Bigg(\sum_{m=1}^M \big\|\mathcal{B}_m[\hat{\mathbb{U}}_{\boldsymbol{\theta}^{(r)}}](\mathbf{x}) - g_m(\mathbf{x}) \big\|_2^2 \Bigg)^{1/2}
    &&\mathbf{x} \in \partial \Omega
\end{align}

At each round, we generate a candidate pool of interior collocation points. Let $\kappa \in \mathbb{N}$ be the number of points to admit per round, and let $\gamma > 1$ be an oversampling factor that controls the size of the candidate pool relative to $\kappa$. Specifically, we sample
\begin{align}
    M_n = \gamma \kappa
\end{align}
points from the domain, forming the candidate set
\begin{align}
    \mathcal{X}_{\mathrm{cand}}^{(r)} = \{\tilde{\mathbf{x}}_1, \dots, \tilde{\mathbf{x}}_{M_n}\} \subset \Omega
\end{align}
For each candidate point, we compute the corresponding residual indicator:
\begin{align}
    v_\ell = \eta_{\mathrm{int}}^{(r)}(\tilde{\mathbf{x}}_\ell)
\end{align}
for indices
\begin{align}
    \ell = 1, \dots, M_n
\end{align}
The candidate points are then ranked according to these residual indicators, and the top $\kappa$ points with the largest values are selected for inclusion in the collocation set:
\begin{align}
    \mathcal{S}^{(r)}_{\mathrm{int}}
    = \operatorname{Top}_{\kappa}\!\big(\{v_\ell\}_{\ell=1}^{M_n}\big)
\end{align}
Finally, the interior collocation set is updated by appending the selected points:
\begin{align}
    \mathcal{X}^{(r+1)}_{\mathrm{int}} 
    = \mathcal{X}^{(r)}_{\mathrm{int}} \cup \mathcal{S}^{(r)}_{\mathrm{int}}
\end{align}

The boundary collocation set can be treated analogously. In this work, however, the boundary points are fixed throughout training: 
\begin{align}
    \mathcal{X}_{\mathrm{bnd}}^{(r+1)} = \mathcal{X}_{\mathrm{bnd}}^{(r)}
\end{align}

For computer implementation details, see the code snippet provided in the appendix (\S\ref{Subsec:APINNS_Adaptive_sampling}).
    %*********************************************;
%                                             ;
%  NAME                                       ;
%    S4_APINNs_Convergence.tex                ;
%                                             ;
%  WRITTEN BY                                 ;
%    Kalyana B. Nakshatrala                   ;
%                                             ;
%*********************************************;
\section{CONVERGENCE ANALYSIS}
\label{Sec:S4_APINNs_Convergence}

We adopt the standard abbreviation ``a.e." for almost everywhere. A function is said to vanish almost everywhere if it is equal to zero except on a set of measure zero \citep{Axler2020_MeasureIntegrationRealAnalysis}.
For the sake of completeness, we now enumerate the assumptions underpinning the subsequent mathematical analysis. Although some have appeared previously, they are restated here in consolidated form.
%-------------------------;
%  Enumerate assumptions  ;
%-------------------------;
\begin{enumerate}
    \item[(i)] $\Omega \subset \mathbb{R}^{nd}$ is a bounded Lipschitz domain.
    \item[(ii)] The material parameters satisfy $\mu > 0$ and $\beta \ge 0$.
    \item[(iii)] The permeabilities are essentially bounded (i.e., $\mathbf{K}^{(i)}(\mathbf{x})\in L^{\infty}(\Omega;\mathbb{R}^{nd\times nd}))$ and uniformly elliptic (cf. Eq.~\eqref{Eqn:APIINS_permeabilities_uniformly_elliptic}). Furthermore, the permeabilities are symmetric. 
    \item[(iv)] The pressure is prescribed on a subset of the boundary of positive measure for at least one pore network: either $\operatorname{meas}(\Gamma^{(1)}_p) > 0$ or $\operatorname{meas}(\Gamma^{(2)}_p) > 0$. This condition fixes the pressure datum.
\end{enumerate}

%--------------------------------;
%  Remark: Scope of convergence  ;
%--------------------------------;
\begin{remark}[Scope of the convergence analysis]
\label{Rem:Scope_convergence_analysis}
    The analysis below first establishes the well-posedness and stability of the continuous least-squares formulation and then relates its neural-network and empirical-collocation approximations to the exact solution. The resulting convergence statement holds under the stated approximation, sampling-consistency, and optimization-accuracy assumptions. It does not establish unconditional convergence of the individual Adam--L-BFGS iterates or of the coupled adaptive-weighting and residual-based refinement procedure. Rather, the adaptive algorithm generates the approximations to which the analysis applies when these assumptions are satisfied.
\end{remark}

%============================================;
%  Subsection: Infinite-dimensional setting  ;
%============================================;
\subsection{Infinite-dimensional setting} To carry out the convergence analysis, we first introduce an infinite-dimensional analogue of the discrete optimization problem \eqref{Eqn:NN_Optimization_def} presented in the preceding section. For clarity and conciseness, we introduce the notation for the operators defining linear momentum and mass balance (i.e., the continuity equation). 

We define the \emph{linear momentum operator} associated with the $i$-th pore network as follows: 
%-------------------------------;
%  Equation: Momentum operator  ;
%-------------------------------;
\begin{align}
    \label{Eqn:APINNS_Momentum_operator}
    \mathcal{M}^{(i)}\big(p^{(i)}(\mathbf{x}),\mathbf{u}^{(i)}(\mathbf{x})\big) 
    := \mu \, \big(\mathbf{K}^{(i)}(\mathbf{x}) \big)^{-1} 
    \mathbf{u}^{(i)}(\mathbf{x}) + \mathrm{grad}\big[p^{(i)}(\mathbf{x})\big]
\end{align}
Then, the corresponding residual of the linear momentum equation is given by
%-------------------------------;
%  Equation: Momentum residual  ;
%-------------------------------;
\begin{align}
    \label{Eqn:APINNS_Momentum_residual}
    \mathbf{R}_{1}^{(i)}(\mathbf{x}) = \mathcal{M}^{(i)}\big(p^{(i)}(\mathbf{x}),\mathbf{u}^{(i)}(\mathbf{x})\big)
    - \phi^{(i)}(\mathbf{x}) \, \gamma \, \mathbf{b}(\mathbf{x})
\end{align}
Likewise, we introduce the \emph{continuity operator} for the $i$-th pore network\footnote{The sign of the transfer term is positive for the first pore network and negative for the second. Using the symbol $\pm$ within the continuity operator or in the mass-balance residual should therefore not cause any ambiguity.}:
%---------------------------------;
%  Equation: Continuity operator  ;
%---------------------------------;
\begin{align}
    \label{Eqn:APINNS_Continuity_operator}
    \mathcal{C}^{(i)}\big(p^{(1)}(\mathbf{x}),p^{(2)}(\mathbf{x}),\mathbf{u}^{(i)}(\mathbf{x})\big) 
    := \mathrm{div}\big[\mathbf{u}^{(i)}(\mathbf{x})\big] 
    \pm \frac{\beta}{\mu} \, \big(p^{(1)}(\mathbf{x}) - p^{(2)}(\mathbf{x})\big)
\end{align}
The residual of the mass-balance equation is therefore 
%-------------------------------;
%  Equation: Momentum residual  ;
%-------------------------------;
\begin{align}
    \label{Eqn:APINNS_Continuity_residual}
    \mathrm{R}_{2}^{(i)}(\mathbf{x}) 
    = \mathcal{C}^{(i)}\big(p^{(1)}(\mathbf{x}),p^{(2)}(\mathbf{x}),\mathbf{u}^{(i)}(\mathbf{x})\big)
\end{align}

The \emph{infinite-dimensional optimization problem} can be stated as follows: 
%-------------------------------------------------------;
%  Equation: Infinite-dimensional optimization problem  ;
%-------------------------------------------------------;
\begin{align}
    \label{Eqn:APINNS_Infinite_dimensional_optimization_problem}
    \inf_{\mathbb{U}(\mathbf{x}) \in \mathcal{U}} 
    \Pi_{\text{LS}}\big[\mathbb{U}(\mathbf{x})\big] 
\end{align}
The associated \emph{least-squares objective functional} is defined as
\begin{align}
    \Pi_{\text{LS}}\big[\mathbb{U}(\mathbf{x})\big] 
    &:= \sum_{i=1}^{2} \Bigg(\frac{1}{2} \, \big\|\mathbf{R}^{(i)}_{1}(\mathbf{x})\big\|^{2}_{L^{2}(\Omega)}
    + \frac{1}{2} \, \big\|\mathrm{R}^{(i)}_{2}(\mathbf{x})\big\|^{2}_{L^{2}(\Omega)}
    \nonumber \\ 
    &\hspace{0.7in} + \frac{1}{2} \, \big\|\mathbf{u}^{(i)}(\mathbf{x}) \bullet \widehat{\mathbf{n}}(\mathbf{x}) - 
    u_n^{(i)}(\mathbf{x})\big\|^{2}_{L^{2}(\Gamma_u^{(i)})}  
    + \frac{1}{2} \big\|p^{(i)}(\mathbf{x}) - 
    p^{(i)}_{\mathrm{p}}(\mathbf{x})\big\|^{2}_{L^{2}(\Gamma_p^{(i)})}
    \Bigg) 
\end{align}

%===================================;
%  Subsubsection: Weak formulation  ;
%===================================;
\subsubsection{Weak formulation}
Let $\mathbb{W}(\mathbf{x})$ denote the set of test functions corresponding to $\mathbb{U}(\mathbf{x})$, which is the ordered set of trial functions or solution fields, defined in Eq.~\eqref{Eqn:APINNS_set_of_trial_functions}. Similar to $\mathbb{U}(\mathbf{x})$, $\mathbb{W}(\mathbf{x})$ takes the following form:
%-----------------------------------;
%  Equation: Set of test functions  ;
%-----------------------------------;
\begin{align}
    \label{Eqn:APINNS_set_of_test_functions}
    \mathbb{W}(\mathbf{x}) = 
    \big(q^{(1)}(\mathbf{x}), 
    q^{(2)}(\mathbf{x}), \mathbf{w}^{(1)}(\mathbf{x}), \mathbf{w}^{(2)}(\mathbf{x})
    \big) 
\end{align}

A necessary condition for a solution to the infinite-dimensional optimization problem \eqref{Eqn:APINNS_Infinite_dimensional_optimization_problem} is that it satisfies the following weak formulation:
%------------------------------;
%  Equation: Weak formulation  ;
%------------------------------;
\begin{align}
    \label{Eqn:APINNS_continuous_LS_weak_form}
    \mathcal{B}\big(\mathbb{W}(\mathbf{x});\mathbb{U}(\mathbf{x})\big) 
    = l\big(\mathbb{W}(\mathbf{x})\big) \quad \forall 
    \mathbb{W}(\mathbf{x}) \in \mathcal{U}
\end{align}
where the bilinear form is defined as follows:
%---------------------------;
%  Equation: Bilinear form  ;
%---------------------------;
\begin{align}
    \label{Eqn:APINNS_continuous_LS_bilinear_form}
    \mathcal{B}\big(\mathbb{W}(\mathbf{x});\mathbb{U}(\mathbf{x})\big) 
    &:= \sum_{i=1}^{2} \Bigg(\int_{\Omega} 
    \mathcal{M}^{(i)}\big(q^{(i)}(\mathbf{x}),\mathbf{w}^{(i)}(\mathbf{x})\big) \bullet 
    \mathcal{M}^{(i)}\big(p^{(i)}(\mathbf{x}),\mathbf{u}^{(i)}(\mathbf{x})\big) \, \mathrm{d} \Omega 
    \nonumber \\ 
    &\hspace{0.55in} +\int_{\Omega} 
    \mathcal{C}^{(i)}\big(q^{(1)}(\mathbf{x}),q^{(2)}(\mathbf{x}),\mathbf{w}^{(i)}(\mathbf{x})\big) \bullet 
    \mathcal{C}^{(i)}\big(p^{(1)}(\mathbf{x}),p^{(2)}(\mathbf{x}),\mathbf{u}^{(i)}(\mathbf{x})\big) \, \mathrm{d} \Omega 
    \nonumber \\ 
    &\hspace{0.55in} + \int_{\Gamma^{(i)}_u} \Big(\mathbf{w}^{(i)}(\mathbf{x}) \bullet \widehat{\mathbf{n}}(\mathbf{x})\Big) 
    \, \Big(\mathbf{u}^{(i)}(\mathbf{x}) \bullet \widehat{\mathbf{n}}(\mathbf{x})\Big) \, \mathrm{d} \Gamma 
   \nonumber \\ 
    &\hspace{0.55in} + \int_{\Gamma^{(i)}_p} q^{(i)}(\mathbf{x}) \, 
    p^{(i)}(\mathbf{x}) \, \mathrm{d} \Gamma 
    \Bigg) 
\end{align}
The linear form is defined as follows: 
%-------------------------;
%  Equation: Linear form  ;
%-------------------------;
\begin{align}
    \label{Eqn:APINNS_continuous_LS_linear_form}
    l\big(\mathbb{W}(\mathbf{x})\big) 
    &:= \sum_{i=1}^{2} \Bigg( 
    \int_{\Omega} \mathcal{M}^{(i)}\big(q^{(i)}(\mathbf{x}),\mathbf{w}^{(i)}(\mathbf{x})\big) 
    \bullet \phi^{(i)}(\mathbf{x}) \, \gamma \, \mathbf{b}(\mathbf{x}) \, \mathrm{d} \Omega 
    \nonumber \\ 
    &\hspace{0.5in} 
    + \int_{\Gamma^{(i)}_u} \mathbf{w}^{(i)}(\mathbf{x}) 
    \bullet \widehat{\mathbf{n}}(\mathbf{x}) 
    \, u_n^{(i)}(\mathbf{x}) \, \mathrm{d} \Gamma 
    + \int_{\Gamma^{(i)}_p} q^{(i)}(\mathbf{x}) \, p_{\mathrm{p}}^{(i)}(\mathbf{x}) \, \mathrm{d} \Gamma 
    \Bigg) 
\end{align}

We define the \emph{least-squares energy functional} as follows: 
%-----------------------------------------;
%  Equation: Definition of LS functional  ;
%-----------------------------------------;
\begin{align}
    \label{Eqn:APINNS_LS_energy_functional_defnition}
    \big\|\mathbb{U}(\mathbf{x})\big\|_{\text{LS}} 
    := \sqrt{\mathcal{B}\big(\mathbb{U}(\mathbf{x});\mathbb{U}(\mathbf{x})\big)}
\end{align}
The quantity $\|\mathbb{U}(\mathbf{x})\|_{\text{LS}}$ can be shown to define a norm; norms obtained as the square root of a bilinear form are commonly referred to as energy norms. In the present work, however, we do not establish that $\|\mathbb{U}(\mathbf{x})\|_{\text{LS}}$ satisfies all the axioms of a norm. Instead, we show only that the associated quadratic functional has a trivial kernel, as demonstrated below. For this reason, we refer to $\|\mathbb{U}(\mathbf{x})\|_{\text{LS}}$ as the least-squares energy \emph{functional}, rather than the least-squares energy \emph{norm}.

%========================================;
%  Subsection: Existence and uniqueness  ;
%========================================;
\subsection{Existence and uniqueness of the infinite-dimensional problem} 
We now show that the infinite-dimensional problem is well-posed; in particular, $\mathbb{U}(\mathbf{x})$ exists and is unique. Our strategy is to invoke the Lax--Milgram theorem \citep{brenner2008mathematical}, which requires establishing the coercivity and boundedness of the associated bilinear form. The inter-pore mass-transfer term in the mass-balance equations introduces technical challenges that prevent us from verifying coercivity directly using standard inequalities (e.g., Friedrichs, Poincaré, or trace inequalities).

To overcome this difficulty, we take an alternative approach that (a) establishes that the least-squares energy functional possesses a trivial kernel, (b) derives an intermediate estimate reminiscent of a Gårding-type inequality, and (c) employs some standard analytical results from Banach spaces (e.g., Banach--Alaoglu theorem, compact Sobolev embedding, and Rellich--Kondrachov theorem). 

%=================================;
%  Subsubsection: Trivial kernel  ;
%---------------------------------;
\subsubsection{Trivial kernel of the least-squares energy functional} A key observation is that the least-squares energy functional vanishes only when all residuals in the governing equations are zero, which forces every component of the solution field to vanish as well. In other words, the only function whose mass-balance, momentum-balance, and auxiliary residuals are all identically zero is the zero solution itself. 

%===============================;
%  Proposition: Trivial kernel  ;
%-------------------------------;
\begin{proposition}[Trivial kernel]
    \label{Prop:APINNS_trivial_kernel_of_LS_functional}
    The least-squares energy functional $\|\mathbb{U}(\mathbf{x})\|_{\mathrm{LS}}$ 
    has a trivial kernel on $\mathcal{U}$; that is,
    \begin{align}
        \|\mathbb{U}(\mathbf{x})\|_{\mathrm{LS}} = 0
        \quad \Longleftrightarrow \quad
        \mathbb{U}(\mathbf{x}) = \mathrm{O}
    \end{align}
\end{proposition}
%---------;
%  Proof  ;
%---------;
\begin{proof}
    We begin by observing that $\mathcal{B}(\mathbb{U}(\mathbf{x}); \mathbb{U}(\mathbf{x}))$ is a sum of nonnegative contributions. Hence, if
    \[
        \|\mathbb{U}(\mathbf{x})\|_{\mathrm{LS}}
        \equiv \sqrt{\mathcal{B}(\mathbb{U}(\mathbf{x}); \mathbb{U}(\mathbf{x}))}    
        = 0
    \]
    then every term in this sum must vanish almost everywhere. Thus,
    \begin{subequations} 
        \label{Eqn:APINNS_Trivial_kernel_ae_equations}
        \begin{alignat}{2}
            \label{Eqn:APINNS_Trivial_kernel_ae_equations_Eqa}
            &\mu \, \big(\mathbf{K}^{(i)}(\mathbf{x})\big)^{-1}\mathbf{u}^{(i)}(\mathbf{x}) + \mathrm{grad}\big[p^{(i)}(\mathbf{x})\big] 
            = \mathbf{0} 
            &&\qquad \mathrm{a.e.} \\
            \label{Eqn:APINNS_Trivial_kernel_ae_equations_Eqb}
            &\mathrm{div}\big[\mathbf{u}^{(1)}(\mathbf{x})\big] 
            + \tfrac{\beta}{\mu}\big(p^{(1)}(\mathbf{x}) 
            - p^{(2)}(\mathbf{x})\big) = 0 
            &&\qquad \mathrm{a.e.} \\ 
            \label{Eqn:APINNS_Trivial_kernel_ae_equations_Eqc}
            &\mathrm{div}\big[\mathbf{u}^{(2)}(\mathbf{x})\big] 
            - \tfrac{\beta}{\mu}\big(p^{(1)}(\mathbf{x}) 
            - p^{(2)}(\mathbf{x})\big) = 0
            &&\qquad \mathrm{a.e.}
        \end{alignat}
    \end{subequations}
with homogeneous boundary conditions: $\mathbf{u}^{(i)}(\mathbf{x}) \bullet \widehat{\mathbf n}(\mathbf{x})=0$ on $\Gamma^{(i)}_u$ and $p^{(i)}(\mathbf{x})=0$ on $\Gamma^{(i)}_p$.

Testing the $i$-th divergence equation with $p^{(i)}(\mathbf{x})$, integrating over $\Omega$ (i.e., Eqs.~\eqref{Eqn:APINNS_Trivial_kernel_ae_equations_Eqb} and \eqref{Eqn:APINNS_Trivial_kernel_ae_equations_Eqc}), and summing the two identities, we proceed as follows:
\begin{align}
    0 = \int_{\Omega} p^{(1)}(\mathbf{x}) \, \mathrm{div}\big[\mathbf{u}^{(1)}(\mathbf{x})\big] \,\mathrm{d}\Omega
    + \int_{\Omega} p^{(2)}(\mathbf{x}) \, \mathrm{div}\big[\mathbf{u}^{(2)}(\mathbf{x})\big] \,\mathrm{d}\Omega 
    + \int_{\Omega}\frac{\beta}{\mu} \, \big(p^{(1)}(\mathbf{x}) - p^{(2)}(\mathbf{x})\big)^2 \,\mathrm{d}\Omega 
\end{align}
Invoking the Green's identity and noting that the boundary integral vanishes by the homogeneous traces on $\Gamma^{(i)}_u$ and $\Gamma^{(i)}_p$, we arrive at the following: 
\begin{align}
    0 = -\int_{\Omega} \mathbf{u}^{(1)}(\mathbf{x}) \bullet 
    \mathrm{grad}\big[p^{(1)}(\mathbf{x})\big] \,\mathrm{d}\Omega
   -\int_{\Omega} \mathbf{u}^{(2)}(\mathbf{x}) \bullet \mathrm{grad}\big[p^{(2)}(\mathbf{x})\big] \,\mathrm{d}\Omega
   + \int_{\Omega}\frac{\beta}{\mu} \, \big(p^{(1)}(\mathbf{x}) -p^{(2)}(\mathbf{x})\big)^{2}\,\mathrm{d}\Omega 
\end{align}

    Likewise, testing the $i$-th momentum equation with $\mathbf{u}^{(i)}(\mathbf{x})$, integrating over $\Omega$, and summing over the index, we arrive at the following:
    \begin{align}
        0 &=\sum_{i=1}^{2}\int_{\Omega}\Big(\mu\,\mathbf u^{(i)}(\mathbf{x}) \bullet \big(\mathbf K^{(i)}(\mathbf{x})\big)^{-1}\mathbf u^{(i)}(\mathbf{x}) + \mathbf{u}^{(i)}(\mathbf{x}) \bullet \mathrm{grad}\big[p^{(i)}(\mathbf{x})\big] \Big)\, \mathrm d\Omega 
    \end{align}
    By adding the previous two equations, we establish the following: 
    \begin{align}
        \sum_{i=1}^{2}\int_{\Omega}\mu \, \mathbf u^{(i)}(\mathbf{x}) \bullet \big(\mathbf K^{(i)}(\mathbf{x})\big)^{-1}\mathbf u^{(i)}(\mathbf{x}) \, \mathrm d\Omega
        + \int_{\Omega}\frac{\beta}{\mu} \, \big(p^{(1)}(\mathbf{x}) - p^{(2)}(\mathbf{x})\big)^2\,\mathrm d\Omega
        = 0
    \end{align}
Since $\mathbf K^{(i)}(\mathbf{x})$ is uniformly elliptic (and hence positive definite), $\beta > 0$ and $\mu > 0$, each term is nonnegative, and hence, each is zero. Therefore, 
\begin{align} 
    \label{Eqn:APINNS_Trivial_kernel_u_zero}
    \mathbf u^{(i)}(\mathbf{x}) = \mathbf 0
    \quad \mathrm{and} \quad 
    p^{(1)}(\mathbf{x}) - p^{(2)}(\mathbf{x}) = 0
    \qquad \text{a.e. in} \; \Omega
\end{align}

From Eq.~$\eqref{Eqn:APINNS_Trivial_kernel_u_zero}_{1}$, together with the momentum equation for each pore and invoking $\text{Eq}.~\eqref{Eqn:APINNS_Trivial_kernel_u_zero}_{2}$, we deduce that
\begin{align}
    p^{(1)}(\mathbf{x}) = p^{(2)}(\mathbf{x}) = C
\end{align}
where $C$ is a constant. Since at least one of the boundary subsets satisfies $\text{meas}(\Gamma^{(1)}_p) > 0$ or $\text{meas}(\Gamma^{(2)}_p) > 0$, the homogeneous trace condition applies on the non-empty portion of the pressure boundary associated with the corresponding pore network. Consequently, we have
\begin{align}
    p^{(i)}(\mathbf{x}) = 0 
    \quad \text{a.e. in} \; \Omega  
\end{align} 
and therefore we conclude that $\mathbb{U} = \mathrm{O}$ almost everywhere.

For the reverse implication, suppose that $\mathbb{U}(\mathbf{x}) = \mathrm{O}$. Then $p^{(i)}(\mathbf{x}) = 0$ and $\mathbf{u}^{(i)}(\mathbf{x}) = \mathbf{0}$, which implies that all integrals in $\mathcal{B}(\cdot;\cdot)$ vanish (see Eq.~\eqref{Eqn:APINNS_continuous_LS_bilinear_form}). Consequently, $\|\mathbb{U}(\mathbf{x})\|_{\text{LS}} = 0$.
\end{proof}

%==================================;
%  Lemma: Garding-type inequality  ;
%----------------------------------;
\begin{lemma}[G{\aa}rding-type inequality for the DPP model]
    \label{Lemma:APINNS_Garding_type_inequality}
    There exist two constants $\alpha_0 > 0$ and $\alpha_1 > 0$ such that 
    \begin{align}
        \alpha_0 \, \big\|\mathbb{U}(\mathbf{x})\big\|^{2}_{\mathcal{U}} 
        \leq \mathcal{B}\big(\mathbb{U}(\mathbf{x});\mathbb{U}(\mathbf{x})\big) 
        + \alpha_1 \, \big\|p^{(1)}(\mathbf{x}) - p^{(2)}(\mathbf{x})\big\|^2
    \end{align}
\end{lemma}
%------------------------------------;
%  Proof of Garding-type inequality  ;
%------------------------------------;
\begin{proof} The strategy is to derive estimates for $p^{(i)}(\mathbf{x})$ in $H^{1}(\Omega)$ and $\mathbf{u}^{(i)}(\mathbf{x})$ in $H(\mathrm{div};\Omega)$. 

%-----------------------------------;
%   Step 1: Control of u in H(div)  ;
%-----------------------------------;
\textbf{Step 1: An estimate for $\mathbf{u}^{(i)}(\mathbf{x})$ in $H(\mathrm{div};\Omega)$.}
The Friedrichs inequality when applied to $\mathbf{u}^{(i)}(\mathbf{x})$ takes the following form: 
\begin{align}
    \|\mathbf{u}^{(i)}(\mathbf{x})\|
    &\le 
    C_{\mathrm{div}}\Big( 
    \big\|\mathrm{div}[\mathbf{u}^{(i)}(\mathbf{x})]\big\|
        + \big\|\mathbf{u}^{(i)}(\mathbf{x}) \bullet \widehat{\mathbf{n}}(\mathbf{x})\big\|_{L^{2}(\Gamma_{u}^{(i)})}
    \Big)
\end{align}
Applying triangle inequality on the continuity operator \eqref{Eqn:APINNS_Continuity_operator}, separately for each network, we write 
\begin{align}
    \label{Eqn:APINNS_ui_div_bound}
    \big\|\mathrm{div}[\mathbf{u}^{(i)}(\mathbf{x})]\big\|
    &\le 
    \big\|\mathcal{C}^{(i)}\big(p^{(1)}(\mathbf{x}),p^{(2)}(\mathbf{x}),\mathbf{u}^{(i)}(\mathbf{x})\big)\big\|
    + \frac{\beta}{\mu} \, \big\|p^{(1)}(\mathbf{x}) - p^{(2)}(\mathbf{x})\big\|
\end{align}
The above two inequalities provide the following estimate for $\mathbf{u}^{(i)}(\mathbf{x})$:
\begin{align}
    \label{Eqn:APINNS_ui_final_inequality}
    \|\mathbf{u}^{(i)}(\mathbf{x})\|
    &\le 
    C_{\mathrm{div}}\Big(
        \big\|\mathcal{C}^{(i)}\big(p^{(1)}(\mathbf{x}),p^{(2)}(\mathbf{x}),\mathbf{u}^{(i)}(\mathbf{x})\big)\big\|
        + \frac{\beta}{\mu} \, \big\|p^{(1)}(\mathbf{x}) - p^{(2)}(\mathbf{x})\big\|
        \nonumber \\ 
        &\hspace{2.56in} + \big\|\mathbf{u}^{(i)}(\mathbf{x}) \bullet \widehat{\mathbf{n}}(\mathbf{x})\big\|_{L^{2}(\Gamma_{u}^{(i)})}
    \Big)
\end{align}

%-------------------------------;
%   Step 2: Control of p in H1  ;
%-------------------------------;
\textbf{Step 2: An estimate for $p^{(i)}(\mathbf{x})$ in $H^{1}(\Omega)$.} Applying the Poincar\'e inequality to the scalar field $p^{(i)}(\mathbf{x})$, we establish
\begin{align}
    \label{Eqn:APINNS_pi_PF_inequality}
    \big\|p^{(i)}(\mathbf{x})\big\|
    \leq C_{\mathrm{P}} \left(
    \big\|\mathrm{grad}[p^{(i)}(\mathbf{x})]\big\|
    + \big\|p^{(i)}(\mathbf{x})\big\|_{L^{2}(\Gamma_{p}^{(i)})}
    \right) 
\end{align}
Applying triangle inequality on the linear momentum operator \eqref{Eqn:APINNS_Momentum_operator}, separately for each network, we get
\begin{align}
    \label{Eqn:APINNS_grad_pi_inequality}
    \big\| \text{grad} [p^{(i)}(\mathbf{x})]\big\| 
    &\le \big\|\mathcal{M}_1^{(i)}\big(p^{(i)}(\mathbf{x}),\mathbf{u}^{(i)}(\mathbf{x})\big)\big\| 
    + \big\|\mu \, \big(\mathbf{K}^{(i)}(\mathbf{x})\big)^{-1}\mathbf{u}^{(i)}(\mathbf{x})\big\| 
    \nonumber \\ 
    &\le  \big\|\mathcal{M}^{(i)}\big(p^{(i)}(\mathbf{x}),\mathbf{u}^{(i)}(\mathbf{x})\big)\big\| + C_K \big\|\mathbf{u}^{(i)}(\mathbf{x})\big\|  
\end{align}
where $C_{K} := \mu \, k^{-1}_{\text{min}}$. 
Multiplying Eq.~\eqref{Eqn:APINNS_grad_pi_inequality} by the factor $2C_{\mathrm{P}}$ and adding the resulting quantity to Eq.~\eqref{Eqn:APINNS_pi_PF_inequality}, we conclude
\begin{align}
    \big\|p^{(i)}(\mathbf{x})\big\| 
    + \big\| \text{grad} [p^{(i)}(\mathbf{x})]\big\| 
    &\le 2C_{\mathrm{P}} \, \big\|\mathcal{M}^{(i)}\big(p^{(i)}(\mathbf{x}),\mathbf{u}^{(i)}(\mathbf{x})\big)\big\| 
    + 2 C_{\mathrm{P}} C_K \, \big\|\mathbf{u}^{(i)}(\mathbf{x})\big\| 
    \nonumber \\ 
    &\hspace{2in} + C_{\mathrm{P}} \big\|p^{(i)}(\mathbf{x})\big\|_{L^{2}(\Gamma_{p}^{(i)})}
\end{align}

Multiplying Eq.~\eqref{Eqn:APINNS_ui_final_inequality} by the factor $(2C_{\mathrm{P}}C_K + 1)$ and adding the resulting quantity to the previous equation, we arrive at the following: 
\begin{align}
    \big\|p^{(i)}(\mathbf{x})\big\| 
    + \big\| \text{grad} [p^{(i)}(\mathbf{x})]\big\| 
    + \big\|\mathbf{u}^{(i)}(\mathbf{x})\big\| 
    &\le 2C_{\mathrm{P}} \, \big\|\mathcal{M}^{(i)}\big(p^{(i)}(\mathbf{x}),\mathbf{u}^{(i)}(\mathbf{x})\big)\big\| 
    \nonumber \\
    &\qquad + \big(2C_{\mathrm{P}}C_K + 1\big) C_{\mathrm{div}}
        \big\|\mathcal{C}^{(i)}\big(p^{(1)}(\mathbf{x}),p^{(2)}(\mathbf{x}),\mathbf{u}^{(i)}(\mathbf{x})\big)\big\|
    \nonumber \\ 
    &\qquad + C_{\mathrm{P}} \big\|p^{(i)}(\mathbf{x})\big\|_{L^{2}(\Gamma_{p}^{(i)})} 
    \nonumber \\
    &\qquad + \big(2C_{\mathrm{P}}C_K + 1\big) C_{\mathrm{div}}
    \big\|\mathbf{u}^{(i)}(\mathbf{x}) \bullet \widehat{\mathbf{n}}(\mathbf{x})\big\|_{L^{2}(\Gamma_{u}^{(i)})}
    \nonumber \\
    &\qquad + \big(2C_{\mathrm{P}}C_K + 1\big) C_{\mathrm{div}}
        \frac{\beta}{\mu} \, \big\|p^{(1)}(\mathbf{x}) - p^{(2)}(\mathbf{x})\big\|
\end{align}
Adding the previous inequality to Eq.~\eqref{Eqn:APINNS_ui_div_bound}, we arrive at the following estimate: 
\begin{align}
    &\big\|p^{(i)}(\mathbf{x})\big\| 
    + \big\| \text{grad} [p^{(i)}(\mathbf{x})]\big\| 
    + \big\|\mathbf{u}^{(i)}(\mathbf{x})\big\| 
    + \big\|\mathrm{div}[\mathbf{u}^{(i)}(\mathbf{x})]\big\| 
    \nonumber \\ 
    &\hspace{1in} \le 2C_{\mathrm{P}} \, \big\|\mathcal{M}^{(i)}\big(p^{(i)}(\mathbf{x}),\mathbf{u}^{(i)}(\mathbf{x})\big)\big\|
    \nonumber \\
    &\hspace{1.2in}+ \big(2C_{\mathrm{P}}C_K + 2\big) C_{\mathrm{div}}
    \big\|\mathcal{C}^{(i)}\big(p^{(1)}(\mathbf{x}),p^{(2)}(\mathbf{x}),\mathbf{u}^{(i)}(\mathbf{x})\big)\big\|
    \nonumber \\ 
    &\hspace{1.2in} + C_{\mathrm{P}} \big\|p^{(i)}(\mathbf{x})\big\|_{L^{2}(\Gamma_{p}^{(i)})} 
    + \big(2C_{\mathrm{P}}C_K + 1\big) C_{\mathrm{div}}
    \big\|\mathbf{u}^{(i)}(\mathbf{x}) \bullet \widehat{\mathbf{n}}(\mathbf{x})\big\|_{L^{2}(\Gamma_{u}^{(i)})}
    \nonumber \\
    &\hspace{1.2in} + \Big(\big(2C_{\mathrm{P}}C_K + 1\big) C_{\mathrm{div}} + 1\Big) 
        \frac{\beta}{\mu} \, \big\|p^{(1)}(\mathbf{x}) - p^{(2)}(\mathbf{x})\big\|
\end{align}
Next, for notational convenience, introduce a single constant that upper bounds all coefficients appearing on the right-hand side. Define
\begin{align}
    C_{\text{max}} := \max\left[2C_{\mathrm{P}}, 2C_{\mathrm{P}}C_K + 2,\big((2C_{\mathrm{P}}C_K + 1) C_{\mathrm{div}} + 1 \big) \frac{\beta}{\mu} \right]
\end{align}
With this choice, each term on the right-hand side of the previous inequality can be bounded above by $C_{\max}$ times its corresponding norm. Therefore, we obtain the compact estimate
\begin{align}
    &\big\|p^{(i)}(\mathbf{x})\big\| 
    + \big\| \text{grad} [p^{(i)}(\mathbf{x})]\big\| 
    + \big\|\mathbf{u}^{(i)}(\mathbf{x})\big\| 
    + \big\|\mathrm{div}[\mathbf{u}^{(i)}(\mathbf{x})]\big\| 
    \nonumber \\ 
    &\hspace{0.5in} \le C_{\mathrm{max}} \, \Bigg(\big\|\mathcal{M}^{(i)}\big(p^{(i)}(\mathbf{x}),\mathbf{u}^{(i)}(\mathbf{x})\big)\big\|
    + \big\|\mathcal{C}^{(i)}\big(p^{(1)}(\mathbf{x}),p^{(2)}(\mathbf{x}),\mathbf{u}^{(i)}(\mathbf{x})\big)\big\|
    \nonumber \\ 
    &\hspace{1.5in} + \big\|p^{(i)}(\mathbf{x})\big\|_{L^{2}(\Gamma_{p}^{(i)})} 
    + \big\|\mathbf{u}^{(i)}(\mathbf{x}) \bullet \widehat{\mathbf{n}}(\mathbf{x})\big\|_{L^{2}(\Gamma_{u}^{(i)})} 
    + \big\|p^{(1)}(\mathbf{x}) - p^{(2)}(\mathbf{x})\big\|
    \Bigg) 
\end{align}
We note that the quantities on either side of the previous inequality are positive; squaring both sides and using the inequality
$(a_1+\cdots+a_5)^2 \le 5(a_1^2+\cdots+a_5^2)$, we obtain
\begin{align}
    &\big\|p^{(i)}(\mathbf{x})\big\|^2 
    + \big\| \text{grad} [p^{(i)}(\mathbf{x})]\big\|^2 
    + \big\|\mathbf{u}^{(i)}(\mathbf{x})\big\|^2
    + \big\|\mathrm{div}[\mathbf{u}^{(i)}(\mathbf{x})]\big\|^2 
    \nonumber \\
    &\qquad \leq
    \Big(\big\|p^{(i)}(\mathbf{x})\big\| 
    + \big\| \text{grad} [p^{(i)}(\mathbf{x})]\big\| 
    + \big\|\mathbf{u}^{(i)}(\mathbf{x})\big\| 
    + \big\|\mathrm{div}[\mathbf{u}^{(i)}(\mathbf{x})]\big\|\Big)^2 
    \nonumber \\
    &\qquad \le 5C_{\max}^2 \Big(
    \big\|\mathcal{M}^{(i)}\big(p^{(i)}(\mathbf{x}),\mathbf{u}^{(i)}(\mathbf{x})\big)\big\|^2
    + \big\|\mathcal{C}^{(i)}\big(p^{(1)}(\mathbf{x}),p^{(2)}(\mathbf{x}),\mathbf{u}^{(i)}(\mathbf{x})\big)\big\|^2
    \nonumber \\ 
    &\hspace{1.25in} + \big\|p^{(i)}(\mathbf{x})\big\|^2_{L^2(\Gamma_p^{(i)})}
    + \big\|\mathbf{u}^{(i)}(\mathbf{x})\!\bullet\!\widehat{\mathbf{n}}(\mathbf{x})\big\|^2_{L^2(\Gamma_u^{(i)})}
    + \big\|p^{(1)}(\mathbf{x}) - p^{(2)}(\mathbf{x})\big\|^2\Big)
\end{align}
Summing over $i=1,2$ yields
\begin{align}
    \label{Eqn:Ui_sum}
    \big\|\mathbb{U}(\mathbf{x})\big\|_{\mathcal{U}}^2
    &\le 5C_{\max}^2 \sum_{i=1}^2 \Big(
    \big\|\mathcal{M}^{(i)}\big(p^{(i)}(\mathbf{x}),\mathbf{u}^{(i)}(\mathbf{x})\big)\big\|^2
    + \big\|\mathcal{C}^{(i)}\big(p^{(1)}(\mathbf{x}),p^{(2)}(\mathbf{x}),\mathbf{u}^{(i)}(\mathbf{x})\big)\big\|^2 
    \nonumber \\ 
    &\hspace{1in} 
    + \big\|p^{(i)}(\mathbf{x})\big\|^2_{L^2(\Gamma_p^{(i)})}
    + \big\|\mathbf{u}^{(i)}(\mathbf{x})\!\bullet\!\widehat{\mathbf{n}}(\mathbf{x})\big\|^2_{L^2(\Gamma_u^{(i)})}
    \Big) 
    \nonumber \\ 
    &\hspace{1.5in} + 10 C_{\max}^2 \big\|p^{(1)}(\mathbf{x})-p^{(2)}(\mathbf{x})\big\|^2 
\end{align}

By definition of the least-squares bilinear form,
\begin{align}
    \mathcal{B}\big(\mathbb{U}(\mathbf{x});\mathbb{U}(\mathbf{x})\big)
    &= \sum_{i=1}^2 \Big(
    \big\|\mathcal{M}^{(i)}\big(p^{(i)}(\mathbf{x}),\mathbf{u}^{(i)}(\mathbf{x})\big)\big\|^2
    + \big\|\mathcal{C}^{(i)}\big(p^{(1)}(\mathbf{x}),p^{(2)}(\mathbf{x}),\mathbf{u}^{(i)}(\mathbf{x})\big)\big\|^2 
    \nonumber \\ 
    &\hspace{1in} 
    + \big\|p^{(i)}(\mathbf{x})\big\|^2_{L^2(\Gamma_p^{(i)})}
    + \big\|\mathbf{u}^{(i)}(\mathbf{x})\!\bullet\!\widehat{\mathbf{n}}(\mathbf{x})\big\|^2_{L^2(\Gamma_u^{(i)})}
    \Big) 
\end{align}
Combining the above estimates and dividing by $5C_{\max}^2$ gives
\begin{align}
\frac{1}{5C_{\max}^2}\,
\big\|\mathbb{U}(\mathbf{x})\big\|_{\mathcal{U}}^2
\le
\mathcal{B}\big(\mathbb{U}(\mathbf{x});\mathbb{U}(\mathbf{x})\big)
+ 2\,\big\|p^{(1)}(\mathbf{x})-p^{(2)}(\mathbf{x})\big\|^2 
\end{align}
Hence, the result follows with
\begin{align}
\alpha_0 := \frac{1}{5C_{\max}^2}
\quad \mathrm{and} \quad 
\alpha_1 := 2 
\end{align}
\end{proof}

%=======================;
%  Theorem: Coercivity  ;
%-----------------------;
\begin{theorem}[Coercivity]
    The bilinear form $\mathcal{B} : \mathcal{U}\times\mathcal{U}\to\mathbb{R}$ is strictly coercive on $\mathcal{U}$. That is, there exists $\alpha_0 > 0$ such that 
    \begin{align}
        \label{Eqn:APINNS_Coercivity_def}
        \alpha_0 \,\|\mathbb{U}(\mathbf{x})\|_{\mathrm{LS}}^{2} \leq   \mathcal{B}\big(\mathbb{U}(\mathbf{x});\mathbb{U}(\mathbf{x})\big)
        \qquad \forall\mathbb{U}(\mathbf{x})\in\mathcal{U}
    \end{align}
\end{theorem}
%-----------------------;
%  Proof of coercivity  ;
%-----------------------;
\begin{proof}
    Suppose, for the sake of contradiction, that $\mathcal{B}(\cdot;\cdot)$ is \emph{not} coercive. Then the negation of Eq.~\eqref{Eqn:APINNS_Coercivity_def} asserts that
    \begin{equation}
        \label{Eqn:APINNS_Negation_coercivity}
        \forall\,\alpha_0 > 0\;\; \exists\,\mathbb{V}(\mathbf{x}) \in\mathcal{U}\setminus\{\mathrm{O}\}
        \;\text{ such that }\;
        \mathcal{B}\big(\mathbb{V}(\mathbf{x});\mathbb{V}(\mathbf{x})\big)
        < \alpha_0 \,\big\|\mathbb{V}(\mathbf{x})\big\|_{\mathcal{U}}^{2}
    \end{equation}
    For each $n\in\mathbb{N}$, apply Eq.~\eqref{Eqn:APINNS_Negation_coercivity} with
    $\alpha_0 = 1/n$. This yields a nonzero element
    $\mathbb{V}_n(\mathbf{x})\in\mathcal{U}\setminus\{\mathrm{O}\}$ such that
    \begin{equation}
        \label{Eqn:APINNS_Vn_inequality}
        \mathcal{B}\big(\mathbb{V}_n(\mathbf{x});\mathbb{V}_n(\mathbf{x})\big)
        \;\le\;
        \frac{1}{n}\,\big\|\mathbb{V}_n(\mathbf{x})\big\|_{\mathcal{U}}^{2}
    \end{equation}  
    We now normalize $\mathbb{V}_n(\mathbf{x})$ and define
    \begin{align}
        \mathbb{U}_n(\mathbf{x})
        := \frac{\mathbb{V}_n(\mathbf{x})}{\|\mathbb{V}_n(\mathbf{x})\|_{\mathcal{U}}}
        \in\mathcal{U}
    \end{align}
    Then $\|\mathbb{U}_n(\mathbf{x})\|_{\mathcal{U}} = 1$ for all $n$, and
    inequality \eqref{Eqn:APINNS_Vn_inequality} implies
    \begin{equation}\label{eq:Un_energy_small}
        \mathcal{B}\big(\mathbb{U}_n(\mathbf{x});\mathbb{U}_n(\mathbf{x})\big)
        \;=\;
        \frac{\mathcal{B}\big(\mathbb{V}_n(\mathbf{x}),\mathbb{V}_n(\mathbf{x})\big)}
         {\|\mathbb{V}_n(\mathbf{x})\|_{\mathcal{U}}^{2}}
        \;\le\;
        \frac{1}{n}
    \end{equation}
    With the notation
    \begin{align}
        \mathcal{J}\big(\mathbb{U}(\mathbf{x})\big)
        := \mathcal{B}\big(\mathbb{U}(\mathbf{x});\mathbb{U}(\mathbf{x})\big)
        \quad \mathrm{and} \quad 
        \mathbb{U}_n(\mathbf{x}) 
        \equiv \Big(p^{(1)}_n(\mathbf{x}),
        p^{(2)}_n(\mathbf{x}), 
        \mathbf{u}^{(1)}_n(\mathbf{x}), 
        \mathbf{u}^{(2)}_n(\mathbf{x})\Big) 
    \end{align}
    we, therefore, have constructed a sequence
    $\{\mathbb{U}_n(\mathbf{x})\}_{n\in\mathbb{N}}\subset\mathcal{U}$ such that
    \begin{align}
        \label{Eqns:APINNS_Coercivity_bad_seqeuence}
        \big\|\mathbb{U}_n(\mathbf{x})\big\|_{\mathcal{U}} = 1
        \quad \mathrm{and} \quad 
        \mathcal{J}\big(\mathbb{U}_n(\mathbf{x})\big) 
        := \mathcal{B}\big(\mathbb{U}_n(\mathbf{x});\mathbb{U}_n(\mathbf{x})\big) \le \frac{1}{n}
    \end{align}
    Hence,
    \begin{align}
        \mathcal{J}\big(\mathbb{U}_n(\mathbf{x})\big) 
        \longrightarrow 0
        \qquad \text{as } n \to \infty
    \end{align}

    Since $\mathcal{U}$ is a Hilbert space, it is reflexive. Therefore, by the \emph{Banach--Alaoglu theorem} (reflexive weak compactness) \citep{brezis2011functional}, every bounded sequence in $\mathcal{U}$ admits a \emph{weakly} convergent subsequence. Hence, from the boundedness of $\{\mathbb{U}_n(\mathbf{x})\}\subset\mathcal{U}$, there exist a subsequence and an element $\mathbb{U}_\infty(\mathbf{x}) \in \mathcal{U}$ such that
    \begin{align}
        \mathbb{U}_{n_k}(\mathbf{x}) \rightharpoonup \mathbb{U}_\infty(\mathbf{x}) 
        \qquad\text{weakly in }\mathcal{U}
    \end{align}
    By the product structure of $\mathcal{U}$, we may write 
    \begin{align}
        \mathbb{U}_\infty(\mathbf{x})
        =
        \Big(p^{(1)}_\infty(\mathbf{x}),
        p^{(2)}_\infty(\mathbf{x}),
        \mathbf{u}^{(1)}_\infty(\mathbf{x}),
        \mathbf{u}^{(2)}_\infty(\mathbf{x})\Big)
    \end{align}
    Since $\mathcal{J}(\cdot)$ is \emph{weakly lower semi-continuous} on $\mathcal{U}$, we obtain
    \begin{align}
        \mathcal{B}\big(\mathbb{U}_\infty(\mathbf{x});\mathbb{U}_\infty(\mathbf{x})\big)
        = \mathcal{J}\big(\mathbb{U}_\infty(\mathbf{x})\big)
        \;\le\;
        \liminf_{k\to\infty} \mathcal{J}\big(\mathbb{U}_{n_k}(\mathbf{x})\big)
        = \liminf_{k\to\infty} \; \mathcal{B}\big(\mathbb{U}_{n_k}(\mathbf{x});\mathbb{U}_{n_k}(\mathbf{x})\big)
        \;\le\; 0
    \end{align}
    where we used $\mathcal{J}\big(\mathbb{U}_n(\mathbf{x})\big)\le 1/n$ from
    Eq.~\eqref{Eqns:APINNS_Coercivity_bad_seqeuence} in the last step.
    Since $\mathcal{J}\big(\mathbb{U}(\mathbf{x})\big)\ge 0$ for all $\mathbb{U}(\mathbf{x})\in\mathcal{U}$,
    it follows that
    \begin{align}
        \mathcal{B}\big(\mathbb{U}_\infty(\mathbf{x});\mathbb{U}_\infty(\mathbf{x})\big)
        = \mathcal{J}\big(\mathbb{U}_\infty(\mathbf{x})\big) = 0
    \end{align}
    Using Proposition \ref{Prop:APINNS_trivial_kernel_of_LS_functional} (i.e., \emph{trivial kernel of the LS functional}), we conclude that 
    \begin{align}
        \label{Eqn:APINNS_Coercivity_Uinfty_zero}
        \mathbb{U}_\infty(\mathbf{x}) = \mathrm{O}
    \end{align}
    Thus far, we have established that 
    \begin{align}
        p^{(i)}_{n_k}(\mathbf{x}) \rightharpoonup p^{(i)}_{\infty}(\mathbf{x}) = 0 
        \quad \mathrm{and} \quad 
        \mathbf{u}^{(i)}_{n_k}(\mathbf{x}) \rightharpoonup 
        \mathbf{u}^{(i)}_{\infty}(\mathbf{x}) = \mathbf{0} 
    \end{align}
    as $k \to \infty$, for $i = 1,2$.

    We now recall the \emph{compact Sobolev embedding}
     \begin{align}
        H^1(\Omega) \hookrightarrow \hookrightarrow L^2(\Omega)
    \end{align}
    valid for bounded Lipschitz domains $\Omega$;
    see, e.g., \citep[Theorem~5.6]{evanspartial} 
    or \citep[Theorem~6.3]{adams2003sobolev}.
    Since the sequence $\{p_{n_k}^{(i)}(\mathbf{x})\}\subset H^{1}(\Omega)$ is bounded, the \emph{Rellich--Kondrachov theorem} \citep{evanspartial} implies that, for each $i=1,2$, there exists a further subsequence (not relabeled) such that
    \begin{equation}
        p_{n_k}^{(i)}(\mathbf{x}) \longrightarrow p_{\infty}^{(i)}(\mathbf{x})
        = 0\qquad \text{strongly in } L^{2}(\Omega)
    \end{equation}
    Therefore,
    \begin{align}
        \|p_{n_k}^{(i)}(\mathbf{x})\|_{L^{2}(\Omega)} \longrightarrow 0 
        \quad \text{as} \; k\to\infty
    \end{align}

    % Use Gårding inequality
    From the \emph{G\aa rding-type inequality}, we have 
    \begin{equation}
        \label{eq:garding}
        \alpha_0\big\|\mathbb{U}(\mathbf{x})\big\|_{\mathcal{U}}^2 
        \; \leq \; 
        \mathcal{B}\big(\mathbb{U}(\mathbf{x});\mathbb{U}(\mathbf{x})\big) 
        + \alpha_1 \, 
        \big\|p^{(1)}(\mathbf{x}) - p^{(2)}(\mathbf{x})\big\|_{L^2(\Omega)}^2
        \qquad\forall\mathbb{U}(\mathbf{x})\in\mathcal{U}
    \end{equation}
    Apply inequality \eqref{eq:garding} to $\mathbb{U}_n(\mathbf{x})$ and use $\|\mathbb{U}_n(\mathbf{x})\|_{\mathcal{U}}=1$:
    \begin{align}
        \alpha_0
        \;\le\;
        \mathcal{B}\big(\mathbb{U}_n(\mathbf{x});\mathbb{U}_n(\mathbf{x})\big)
        +
        \alpha_1 \, 
        \big\|p^{(1)}_n(\mathbf{x}) - p^{(2)}_n(\mathbf{x})\big\|_{L^2(\Omega)}^2
    \end{align}
    Using Eq.~\eqref{Eqns:APINNS_Coercivity_bad_seqeuence}, we have
    \begin{align}
        \label{Eqn:APINNS_Coercivity_intermediate_inequality}
        \alpha_0 \;\le\;
        \frac{1}{n} + \alpha_1 \, 
        \big\|p^{(1)}_n(\mathbf{x}) -p^{(2)}_n(\mathbf{x})\big\|_{L^2(\Omega)}^2
    \end{align}

    Passing to the limit $k\to\infty$ in Eq.~\eqref{Eqn:APINNS_Coercivity_intermediate_inequality} yields
    \begin{align}
        \alpha_{0} \le 0
    \end{align}
    which contradicts the assumption $\alpha_{0} > 0$. Hence, $\mathcal{B}(\cdot;\cdot)$ is strictly coercive on $\mathcal{U}$.
    %%%
\end{proof}

%==================================;
%  Remark: Weighted least-squares  ;
%----------------------------------;
\begin{remark}[Weighted least--squares]
    Most practical least--squares methods introduce positive weights $\omega_1,\dots,\omega_4>0$ and use the modified functional
    \[
        J_\omega\big(\mathbb{U}(\mathbf{x})\big) := 
        \sum_{j=1}^4 \omega_j\,\big\|R_j\big(\mathbb{U}(\mathbf{x})\big)\big\|_{L^2(\Omega)}^2,
        \qquad
        \mathcal{B}_\omega\big(\mathbb{W}(\mathbf{x});\mathbb{U}(\mathbf{x})\big)
        := \sum_{j=1}^4 \omega_j\,\Big(R_j\big(\mathbb{W}(\mathbf{x})\big),R_j(\mathbb{U}(\mathbf{x})\big)\Big)
    \]
    All results of the theorem remain valid: symmetry is preserved, continuity holds with a possibly different constant, positive definiteness is unchanged, and coercivity still holds because $\mathcal{B}_\omega\big(\mathbb{U}(\mathbf{x});\mathbb{U}(\mathbf{x})\big)$ is equivalent to $\mathcal{B}\big(\mathbb{U}(\mathbf{x});\mathbb{U}(\mathbf{x})\big)$. Thus the weighted least--squares formulation is also well--posed.
    %%%
\end{remark}

%========================;
%  Theorem: Boundedness  ;
%------------------------;
\begin{theorem}[Boundedness]
    The bilinear form $\mathcal{B}(\cdot;\cdot)$ is bounded (continuous) with respect to $\|\cdot\|_{\mathcal{U}}$. That is, there exists a constant $C_{\mathcal{B}}>0$ such that
    \begin{align}
        \label{Eqn:B_boundedness_goal}
        \left|\mathcal{B}\big(\mathbb{W}(\mathbf{x});\mathbb{U}(\mathbf{x})\big)\right|
        \le C_{\mathcal{B}}\, \|\mathbb{W}(\mathbf{x})\|_{\mathcal{U}}\, \|\mathbb{U}(\mathbf{x})\|_{\mathcal{U}}
        \qquad \forall\, \mathbb{W}(\mathbf{x}), 
        \mathbb{U}(\mathbf{x}) \in \mathcal{U}
    \end{align}
\end{theorem}
%------------------------;
%  Proof of boundedness  ;
%------------------------;
\begin{proof}

\textbf{Step 1: Bounds for the operators.}
For each $i=1,2$, using the definition of $\mathcal{M}^{(i)}$ in Eq.~\eqref{Eqn:APINNS_Momentum_operator},
the Cauchy--Schwarz inequality, and the boundedness of $\big(\mathbf{K}^{(i)}(\mathbf{x})\big)^{-1}$ in $L^\infty(\Omega)$,
we obtain
\begin{align}
    \label{Eqn:Mi_bound}
    \|\mathcal{M}^{(i)}\big(q^{(i)}(\mathbf{x}),\mathbf{w}^{(i)}(\mathbf{x})\big)\|_{L^2(\Omega)}
    &= \big\|\mu \big(\mathbf{K}^{(i)}(\mathbf{x})\big)^{-1}\mathbf{w}^{(i)}(\mathbf{x})
    + \mathrm{grad}[q^{(i)}(\mathbf{x})]
    \big\|_{L^2(\Omega)} 
    \nonumber\\
    &\le \mu \big\|\big(\mathbf{K}^{(i)}(\mathbf{x})\big)^{-1}\big\|_{L^\infty(\Omega)} \,\|\mathbf{w}^{(i)}(\mathbf{x})\|_{L^2(\Omega)}
    + \|\mathrm{grad}[q^{(i)}(\mathbf{x})]\|_{L^2(\Omega)} \nonumber\\
    &\le C_{M}^{(i)}\Big(\|q^{(i)}(\mathbf{x})\|_{H^1(\Omega)} 
    + \|\mathbf{w}^{(i)}(\mathbf{x})\|_{H(\mathrm{div};\Omega)}\Big)
\end{align}
where one may take $C_M^{(i)}:=\max\!\left\{\mu\big\|\big(\mathbf{K}^{(i)}(\mathbf{x})\big)^{-1}\big\|_{L^\infty(\Omega)},\,1\right\}$.

Likewise, for the continuity operator \eqref{Eqn:APINNS_Continuity_operator}, we have the following estimate:
\begin{align}
    \label{Eqn:Ci_bound}
    \|\mathcal{C}^{(i)}\big(q^{(1)},q^{(2)},\mathbf{w}^{(i)}\big)\|_{L^2(\Omega)}
    &= \big\|\mathrm{div}[\mathbf{w}^{(i)}(\mathbf{x})] \pm \frac{\beta}{\mu}\big(q^{(1)}(\mathbf{x}) - q^{(2)}(\mathbf{x})\big)\big\|_{L^2(\Omega)} 
    \nonumber\\
    &\le \|\mathrm{div}[\mathbf{w}^{(i)}(\mathbf{x})]\|_{L^2(\Omega)}
    + \frac{\beta}{\mu}\Big(\|q^{(1)}(\mathbf{x})\|_{L^2(\Omega)} + \|q^{(2)}(\mathbf{x})\|_{L^2(\Omega)}\Big) 
    \nonumber\\
    &\le C_C\Big(\|\mathbf{w}^{(i)}(\mathbf{x})\|_{H(\mathrm{div};\Omega)} + \|q^{(1)}(\mathbf{x})\|_{H^1(\Omega)} + \|q^{(2)}(\mathbf{x})\|_{H^1(\Omega)}\Big)
\end{align}
with $C_C:=\max\!\left\{1,\frac{\beta}{\mu}\right\}$.

\textbf {Step 2: Volume terms.}
Using the Cauchy--Schwarz inequality together with Eqs.~\eqref{Eqn:Mi_bound}--\eqref{Eqn:Ci_bound}, we write
% \begin{align}
%     \label{Eqn:volume_terms_bound}
%     \Big|\int_\Omega \mathcal{M}^{(i)}(q^{(i)},\mathbf{w}^{(i)})\cdot \mathcal{M}^{(i)}(p^{(i)},\mathbf{u}^{(i)})\,\mathrm{d}\Omega\Big|
%     %
%     &\le \|\mathcal{M}^{(i)}(q^{(i)},\mathbf{w}^{(i)})\|_{L^2(\Omega)}\, \|\mathcal{M}^{(i)}(p^{(i)},\mathbf{u}^{(i)})\|_{L^2(\Omega)} \nonumber\\
%     %
%     &\le (C_M^{(i)})^2 \Big(\|q^{(i)}\|_{H^1}+\|\mathbf{w}^{(i)}\|_{H(\mathrm{div})}\Big)
%     \Big(\|p^{(i)}\|_{H^1}+\|\mathbf{u}^{(i)}\|_{H(\mathrm{div})}\Big)
%     \\[0.25em]
%     \Big|\int_\Omega \mathcal{C}^{(i)}(q^{(1)},q^{(2)},\mathbf{w}^{(i)})\, \mathcal{C}^{(i)}(p^{(1)},p^{(2)},\mathbf{u}^{(i)})\,\mathrm{d}\Omega\Big|
% &\le \|\mathcal{C}^{(i)}(q^{(1)},q^{(2)},\mathbf{w}^{(i)})\|_{L^2(\Omega)}\, \|\mathcal{C}^{(i)}(p^{(1)},p^{(2)},\mathbf{u}^{(i)})\|_{L^2(\Omega)} \nonumber\\
% &\le C_C^2 \\  
% \Big(\|\mathbf{w}^{(i)}\|_{H(\mathrm{div})}+\|q^{(1)}\|_{H^1}+\|q^{(2)}\|_{H^1}\Big)\Big(\|\mathbf{u}^{(i)}\|_{H(\mathrm{div})}+\|p^{(1)}\|_{H^1}+\|p^{(2)}\|_{H^1(\Omega)}\Big).
% \end{align}
\begin{align}
\label{Eqn:volume_terms_bound}
\Big|\int_\Omega \mathcal{M}^{(i)}(q^{(i)},\mathbf{w}^{(i)})\cdot
\mathcal{M}^{(i)}(p^{(i)},\mathbf{u}^{(i)})\,\mathrm{d}\Omega\Big|
&\le \|\mathcal{M}^{(i)}(q^{(i)},\mathbf{w}^{(i)})\|_{L^2(\Omega)}
   \|\mathcal{M}^{(i)}(p^{(i)},\mathbf{u}^{(i)})\|_{L^2(\Omega)} \nonumber\\
&\le (C_M^{(i)})^2
\Big(\|q^{(i)}\|_{H^1}+\|\mathbf{w}^{(i)}\|_{H(\mathrm{div})}\Big) \nonumber\\
&\quad \times
\Big(\|p^{(i)}\|_{H^1}+\|\mathbf{u}^{(i)}\|_{H(\mathrm{div})}\Big)
\\[0.5em]
\Big|\int_\Omega \mathcal{C}^{(i)}(q^{(1)},q^{(2)},\mathbf{w}^{(i)})
\,\mathcal{C}^{(i)}(p^{(1)},p^{(2)},\mathbf{u}^{(i)})\,\mathrm{d}\Omega\Big|
&\le \|\mathcal{C}^{(i)}(q^{(1)},q^{(2)},\mathbf{w}^{(i)})\|_{L^2(\Omega)}
   \|\mathcal{C}^{(i)}(p^{(1)},p^{(2)},\mathbf{u}^{(i)})\|_{L^2(\Omega)} \nonumber\\
&\le C_C^2
\Big(\|\mathbf{w}^{(i)}\|_{H(\mathrm{div})}
     +\|q^{(1)}\|_{H^1}
     +\|q^{(2)}\|_{H^1}\Big) \nonumber\\
&\quad \times
\Big(\|\mathbf{u}^{(i)}\|_{H(\mathrm{div})}
     +\|p^{(1)}\|_{H^1}
     +\|p^{(2)}\|_{H^1(\Omega)}\Big)
\end{align}
%-------------------------;
%  Step3: Boundary terms  ;
%-------------------------;
\textbf{Step 3: Boundary terms.}
Assuming $\Omega$ is Lipschitz, the trace operator is continuous:
\[
\|q^{(i)}\|_{L^2(\Gamma_p^{(i)})}\le C_{\mathrm{tr},p}^{(i)} \|q^{(i)}\|_{H^1(\Omega)}
\]
Moreover, the normal trace on $H(\mathrm{div};\Omega)$ is continuous into $H^{-1/2}(\partial\Omega)$; in particular,
its restriction to $\Gamma_u^{(i)}$ yields the estimate
\[
\|\mathbf{w}^{(i)}\!\bullet\!\widehat{\mathbf{n}}\|_{L^2(\Gamma_u^{(i)})}
\le C_{\mathrm{tr},u}^{(i)} \|\mathbf{w}^{(i)}\|_{H(\mathrm{div};\Omega)}
\]
provided the normal component is square-integrable on $\Gamma_u^{(i)}$ (which is built into
$H(\mathrm{div};\Omega,\Gamma_u^{(i)})$). Hence, by Cauchy--Schwarz,
\begin{align}
\label{Eqn:boundary_terms_bound}
\Big|\int_{\Gamma_u^{(i)}} (\mathbf{w}^{(i)}\!\bullet\!\widehat{\mathbf{n}})\,(\mathbf{u}^{(i)}\!\bullet\!\widehat{\mathbf{n}})\,\mathrm{d}\Gamma\Big|
&\le \|\mathbf{w}^{(i)}\!\bullet\!\widehat{\mathbf{n}}\|_{L^2(\Gamma_u^{(i)})}\,
     \|\mathbf{u}^{(i)}\!\bullet\!\widehat{\mathbf{n}}\|_{L^2(\Gamma_u^{(i)})} \nonumber\\
&\le (C_{\mathrm{tr},u}^{(i)})^2 \|\mathbf{w}^{(i)}\|_{H(\mathrm{div};\Omega)}\,
                         \|\mathbf{u}^{(i)}\|_{H(\mathrm{div};\Omega)}
\\[0.25em]
\Big|\int_{\Gamma_p^{(i)}} q^{(i)}\,p^{(i)}\,\mathrm{d}\Gamma\Big|
&\le \|q^{(i)}\|_{L^2(\Gamma_p^{(i)})}\,\|p^{(i)}\|_{L^2(\Gamma_p^{(i)})} \nonumber\\
&\le (C_{\mathrm{tr},p}^{(i)})^2 \|q^{(i)}\|_{H^1(\Omega)}\,\|p^{(i)}\|_{H^1(\Omega)}
\end{align}

\textbf{Step 4: Collecting the estimates.}
Combining \eqref{Eqn:volume_terms_bound} and \eqref{Eqn:boundary_terms_bound}, summing over $i=1,2$, and using the
elementary inequality $(a+b)(c+d)\le (a^2+b^2)^{1/2}(c^2+d^2)^{1/2}$ together with the definition of $\|\cdot\|_{\mathcal{U}}$, we obtain
\begin{align}
    \big|\mathcal{B}\big(\mathbb{W}(\mathbf{x});\mathbb{U}(\mathbf{x})\big)\big|
    \le C_{\mathcal{B}}\, \big\|\mathbb{W}(\mathbf{x})\big\|_{\mathcal{U}} \, \big\|\mathbb{U}(\mathbf{x})\big\|_{\mathcal{U}}
\end{align}
where $C_{\mathcal{B}}>0$ depends only on
$\mu$, $\beta/\mu$, the $L^\infty$ bounds of $\big(\mathbf{K}^{(i)}(\mathbf{x})\big)^{-1}$, the trace constants for $\Gamma_u^{(i)}$
and $\Gamma_p^{(i)}$, and the domain regularity. This proves the boundedness claim \eqref{Eqn:B_boundedness_goal}.
\qedhere
\end{proof}

By the Lax--Milgram theorem, there exists a unique solution
\(\mathbb{U}(\mathbf{x}) \in \mathcal{U}\) such that
\begin{align}
    \big\|\mathbb{U}(\mathbf{x})\big\|_{\mathcal{U}}
    \leq \frac{1}{\alpha_0}\, \big\|l\big\|_{\mathcal{U}'}
\end{align}
where \(\mathcal{U}'\) denotes the dual space of \(\mathcal{U}\), and \(\alpha_0 > 0\) is the coercivity constant of the associated bilinear form.

The existence and uniqueness of the exact minimizer $\mathbb{U}(\mathbf{x})$ for the continuous infinite-dimensional least-squares problem have been established above. Once this property is available, the convergence of the neural-network and empirical-collocation approximations is established under the preceding assumptions as follows:
\begin{itemize}
    \item \textbf{Step 1.} Introduce the admissible neural-network class $\mathcal N \subset \mathcal U$ and define the corresponding restricted continuous minimizer $\mathbb U_{\mathrm{dl}} \in \mathcal N$ as the minimizer of the continuous least-squares functional over $\mathcal N$.

    \item \textbf{Step 2.} Replace the continuous least-squares functional by its empirical counterpart $\Pi_M$, obtained from the collocation or sampling procedure, and define the associated trained minimizer $\mathbb U_M \in \mathcal N$ as the minimizer of $\Pi_M$.

    \item \textbf{Step 3.} Decompose the total error $\|\mathbb U_M-\mathbb{U}\|_{\mathcal U}$ into an approximation component, arising from the restriction to the network class $\mathcal N$, and a discretization component, arising from the empirical approximation of the continuous functional.

    \item \textbf{Step 4.} Use approximation properties of the neural-network class together with consistency of the empirical functional to control these two error contributions at the level of the least-squares objective.

    \item \textbf{Step 5.} Invoke stability of the continuous least-squares formulation to convert bounds on functional suboptimality into convergence estimates in the solution norm.
\end{itemize}

Finally, while the present least-squares functional differs from that considered in \citet{maduri2026deepLS}, the proof strategy remains unchanged in structure. For the sake of brevity, we omit the technical details and refer the reader to Section~4 of \citet{maduri2026deepLS}.

    %*********************************************;
%                                             ;
%  NAME                                       ;
%    S5_APINNs_NR.tex                         ;
%                                             ;
%  WRITTEN BY                                 ;
%    Kalyana B. Nakshatrala                   ;
%                                             ;
%*********************************************;
\section{REPRESENTATIVE NUMERICAL RESULTS}
\label{Sec:S5_APINNs_NR}
This section presents a comprehensive evaluation of the proposed modeling framework. We showcase results on a set of well-established benchmark problems from the literature, covering a wide range of physical scenarios and boundary conditions. 

%========================================;
%  Subsection: One-dimensional problems  ;
%========================================;
\subsection{One-dimensional problems}
To begin, we consider two distinct one-dimensional problems. The first is a pressure-driven boundary value problem, in which the solution fields under the DPP model resemble those predicted by the classical Darcy equations. Specifically, there is no mass transfer between the pore networks, and the flow in each network is governed by Darcy’s law. The second is a mixed boundary value problem, where the flow behavior differs fundamentally from that described by the Darcy model; in this case, mass transfer occurs between the pore networks.

%=========================================;
%  Subsubsection: 1D pressure-driven BVP  ;
%=========================================;
\subsubsection{1D pressure-driven boundary value problem} 
\textbf{Figure~\ref{Fig:APINNs_1D_Pressure_BVP}} illustrates pressure-driven flow through a porous medium with a dual-pore network. Macro- and micro-pressures, prescribed on the left and right boundaries, drive the flow. This problem is motivated by the well-known challenges of traditional methods, such as mixed finite element formulations, which often require stabilization when equal-order interpolation is used for all field variables to satisfy the Ladyzhenskaya–Babuška–Brezzi (LBB) stability condition \citep{brezzi2012mixed}. This naturally raises the question of whether the proposed framework is likewise susceptible to such stability constraints, or whether it inherently bypasses them through its architecture and training dynamics.

To investigate this question, we implement the proposed framework using the following neural network architecture and training strategy. The neural network used here comprises six hidden layers with 64 neurons each, employs the hyperbolic tangent (\texttt{tanh}) activation function, and is trained initially with the Adam optimizer (learning rate $1 \times 10^{-3}$), followed by the quasi-Newton L-BFGS algorithm to refine the solution. Table~\ref{Tab:APINNS_1D_Pressure_BVP} lists the material properties, geometric parameters, and boundary conditions used in the numerical simulation.

\textbf{Figure~\ref{fig:APINNs_1D_Pressure_Result}} shows that the solution fields computed with the proposed framework closely match the numerical solution obtained with the stabilized mixed formulation reported in \citet{joodat2018modeling}.

%------------------------------------;
%  Figure 4: 1D BVP pressure-driven  ;
%------------------------------------;
\begin{figure}[h!]
    \centering
    \includegraphics[width=0.5\linewidth]{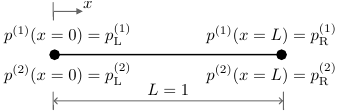}
    \caption{\textsf{1D pressure-driven boundary value problem.} The macro/micro-level pressure boundary conditions are prescribed at both ends.}
    \label{Fig:APINNs_1D_Pressure_BVP}
\end{figure}

%----------------------------------------------;
%  Table 1: Parameters for 1D pressure-driven  ;
%----------------------------------------------;
\begin{table}[htbp]
    \centering
    \caption{\textsf{1D pressure-driven boundary value problem.} Numerical simulation parameters. \label{Tab:APINNS_1D_Pressure_BVP}}
    \begin{tabular}{|c|c|}
        \hline
        \textbf{Parameter} & \textbf{Value} \\ \hline
        $L$         & 1 \\ 
        $\beta$     & 1 \\ 
        $\mu$       & 1 \\ 
        $\gamma \mathbf{b}$  & 0 \\ 
        $k^{(1)}$   & 1    \\ 
        $k^{(2)}$   & 0.01 \\ 
        $p_{\mathrm{L}}^{(1)}$ & 10 \\ %\hline
        $p_{\mathrm{R}}^{(1)}$ & 1 \\ %\hline
        $p_{\mathrm{L}}^{(2)}$ & 10 \\ %\hline
        $p_{\mathrm{R}}^{(2)}$ & 1 \\ \hline
    \end{tabular}
\end{table}

%--------------------------------------------;
%  Figure 5: 1D BVP pressure-driven results  ;
%--------------------------------------------;
\begin{figure}[h!]
    \centering
    \includegraphics[width=0.90\linewidth]{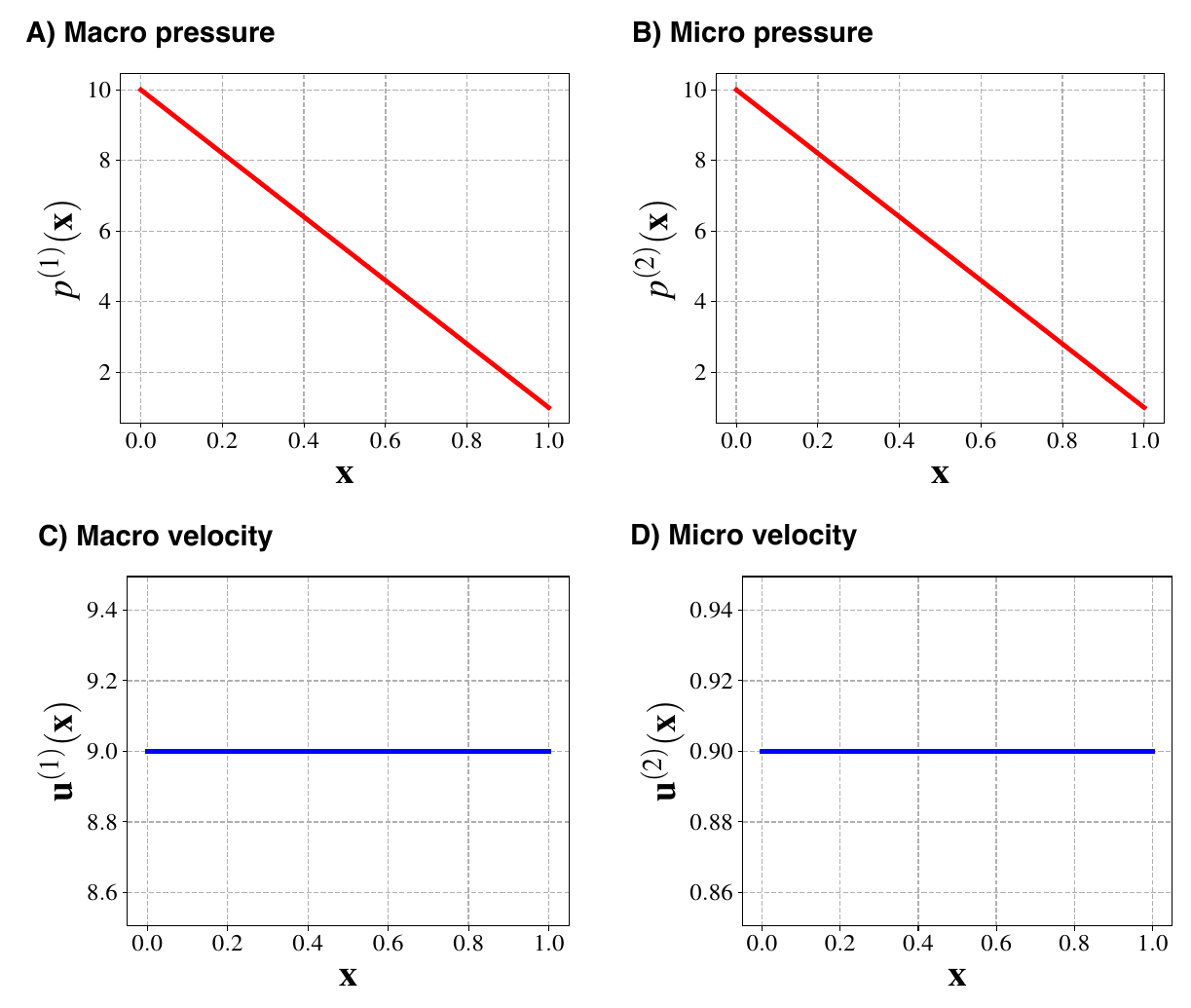}
    \caption{\textsf{1D pressure-driven boundary value problem.} This figure shows the solution fields obtained using the proposed modeling framework. These results show excellent agreement with the finite element solutions reported in \citet{joodat2018modeling}.}
    \label{fig:APINNs_1D_Pressure_Result}
\end{figure}

%===============================;
%  Subsubsection: 1D mixed BVP  ;
%===============================;
\subsubsection{1D mixed boundary value problem} 
In this problem, the macro-pore network is subjected to prescribed pressures at the left and right boundaries---similar to the previous case---while the micro-pore network is sealed (i.e., a no-flux boundary condition, meaning zero normal flow). \textbf{Figure~\ref{Fig:APINNS_1D_Mixed_BVP}} provides a pictorial illustration. Although this setup involves only minor modifications to the previous boundary value problem, it offers valuable insights for two key reasons. 
\begin{enumerate}
    \item It reveals that the micro-pore network can sustain internal discharge even in the absence of boundary outflow, highlighting the importance of characterizing internal pore architecture rather than relying solely on surface-level observations.
    \item It underscores a fundamental difference between the solution fields predicted by the classical Darcy equations and those from the DPP model. In a one-dimensional problem with prescribed velocity boundary conditions at both ends, the incompressibility constraint---together with the divergence theorem---requires the boundary velocities to be equal. Under this setup, the Darcy model predicts a uniform velocity field throughout the domain. In contrast, the DPP model permits a nonuniform velocity in the micro-pore network, due to internal mass exchange between the two pore systems \citep{nakshatrala2018modeling}.
\end{enumerate}

To validate these insights, we apply the proposed framework to this benchmark problem with simulation parameters provided in Table \ref{Tab:APINNS_1D_Mixed_BVP}. The neural network parameters remain as those of the earlier problem. \textbf{Figure \ref{Fig:APINNS_1D_Mixed_BVP_result}} shows the computed micro-velocity and volumetric transfer rate profiles, which closely match the analytical solutions reported by \citet{nakshatrala2018modeling}, demonstrating the framework’s accuracy and robustness.

%--------------------------;
%  Figure 6: 1D mixed BVP  ;
%--------------------------;
\begin{figure}[h!]
    \centering
    \includegraphics[width=0.5\linewidth]{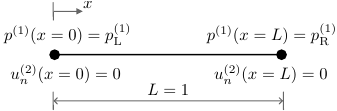}
    \caption{\textsf{1D mixed boundary value problem.} For the macro-pore network, pressure is prescribed at both ends of the domain, while there is no fluid discharge from the micro-pore network at either end. This boundary value problem is referred to as mixed, as pressure boundary conditions are applied to the macro-pore network and velocity boundary conditions to the micro-pore network. These mixed boundary conditions give rise to a fluid flow that is characteristically different from that in a single pore network governed by Darcy equations.}
    \label{Fig:APINNS_1D_Mixed_BVP}
\end{figure}

%----------------------------------------;
%  Table 2: Parameters for 1D mixed BVP  ;
%----------------------------------------;
\begin{table}[h!]
     \caption{\textsf{1D mixed boundary value problem.} Parameters used in the numerical simulation. All other parameter values remain the same as those listed in Table \ref{Tab:APINNS_1D_Pressure_BVP}, except for $p^{(2)}_{\mathrm{L}}$ and $p^{(2)}_{\mathrm{R}}$, which are not applicable for this boundary value problem. \label{Tab:APINNS_1D_Mixed_BVP}}
    \centering
    \begin{tabular}{|c|c|}
        \hline
        \textbf{Parameter} & \textbf{Value} \\ \hline
        $u^{(2)}_{n}(x= 0)$ & 0 \\ 
        $u^{(2)}_{n}(x= L)$ & 0 \\ \hline
    \end{tabular}
\end{table}

%------------------------------------;
%  Figure 7: 1D mixed BVP (results)  ;
%------------------------------------;
\begin{figure}[h!]
    \centering
    \includegraphics[width=0.9\linewidth]{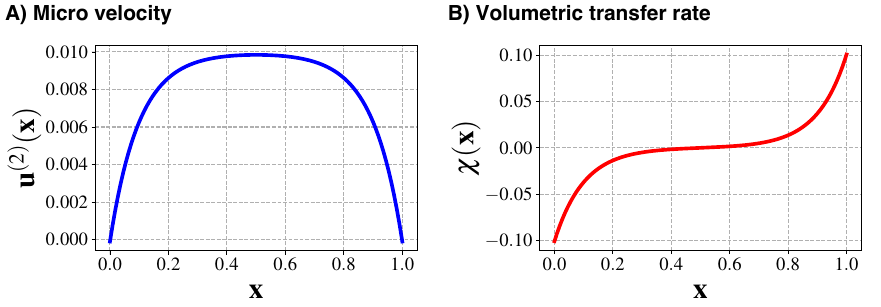}
    \caption{\textsf{1D mixed boundary value problem.} This figure shows the micro-velocity and mass transfer results obtained using the proposed modeling framework. These results are in good agreement with the analytical solution provided by \citet{nakshatrala2018modeling}. Furthermore, the results highlight a key feature of the DPP model: even in the absence of discharge at the boundary of the micro-pore network, internal flow can occur within it due to mass transfer between the pore networks inside the domain.}
    \label{Fig:APINNS_1D_Mixed_BVP_result}
\end{figure}

%===========================================;
%  Subsection: 2D BVP with radial symmetry  ;
%===========================================;
\subsection{Two-dimensional problem with radial symmetry}
\label{sec:APINNS_Radial_Symmetry}
This problem is selected to evaluate whether the proposed framework can capture the inherent radial (cylindrical) symmetry. It is motivated by the operating principle of candle filters, which are commonly used in water purification systems \citep{dickenson1997filters}.

The computational domain is an annular disc, representing a porous medium with two interconnected pore networks, bounded by an inner radius $r = r_i = 0.3$ and an outer radius $r_o = 1$. As illustrated in \textbf{Fig.~\ref{Fig:APINNs_Candle_filter_BVP}}, a uniform macro-pressure of unit magnitude is applied at the inner boundary (i.e., $p^{(1)}\vert_{r = r_i} = 1$), serving as the driving force for filtration. The outer boundary is exposed to ambient conditions, modeled by enforcing zero pressure (i.e., $p^{(1)}\vert_{r = r_o} = 0$). In the micro-pore network, no flow is permitted across either boundary; this condition is mathematically expressed as $\mathbf{u}^{(2)}(\mathbf{x})\bullet\widehat{\mathbf{n}}(\mathbf{x}) = u_{n}^{(2)} = 0$ at $r = r_i$ and $r = r_o$. 

%-----------------------------------;
%  Figure 8: Candle filter problem  ;
%-----------------------------------;
\begin{figure}[h!]
    \centering
    \includegraphics[width=0.55\linewidth]{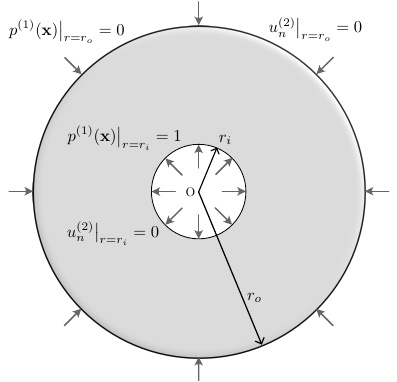}
    \caption{\textsf{Two-dimensional problem with radial symmetry.} The flow is driven by the pressure difference in the macro-pore network at the outer surface ($p^{(1)}\vert_{r = r_o} = 0$) and the inner surface ($p^{(1)}\vert_{r = r_i} = 1$). For the micro-pore network, both the inner and outer surfaces do not permit any fluid discharge. However, there is fluid exchange between the two networks within the domain.}
    \label{Fig:APINNs_Candle_filter_BVP}
\end{figure}

%-------------------------------------------;
%  Table 3: Parameters for radial symmetry  ;
%-------------------------------------------;
\begin{table}[h!]
     \caption{\textsf{Two-dimensional problem with radial symmetry.} Numerical simulation parameters. \label{Tab:APINNS_Params_for_radial_symmetry}}
    \centering
    \begin{tabular}{|c|c|}
        \hline
        \textbf{Parameter} & \textbf{Value} \\ \hline
        $r_i$ & 0.3 \\
        $r_o$ & 1.0 \\
        $\beta$ & 1\\
        $\mu$ & 1\\
        $\gamma\mathbf{b}$ & $\mathbf{0} = [0, 0]$\\
        $k^{(1)}$ & 1\\
        $k^{(2)}$ & 0.01\\
        $p^{(1)}_{\mathrm{p}}(r = r_i)$ & 1 \\
        $p^{(1)}_{\mathrm{p}}(r = r_o)$ & 0 \\
        $u^{(2)}_{n}(r= r_i)$ & 0 \\ 
        $u^{(2)}_{n}(r = r_o)$ & 0 \\ \hline
    \end{tabular}
\end{table}

The neural network uses eight hidden layers with 128 neurons each and the \texttt{ReLU} activation function. Table~\ref{Tab:APINNS_Params_for_radial_symmetry} summarizes the model parameters, including material properties, geometry, and boundary conditions. The total training time on an NVIDIA T4 GPU was 6.17 minutes.

\textbf{Figure~\ref{Fig:APINNs_Candle_filter_results}} displays the numerically computed pressure fields and velocity magnitudes, while \textbf{Fig.~\ref{Fig:APINNs_Candle_filter_volumetric_transfer_results}} illustrates $\chi(\mathbf{x})$, the volumetric transfer rate from the micro-pore to the macro-pore network. The results show excellent agreement with the analytical solution presented in \citet{nakshatrala2018modeling} and closely match the finite element simulations reported in \citet{joodat2018modeling}, thereby confirming the accuracy and robustness of the proposed method. \textbf{Figure~\ref{fig:L2-Error_Radial_Symmetry}} presents the $L_{2}$ error between the analytical solution and the neural-network-predicted solution, evaluated across varying combinations of network depth and width.

%---------------------------------------------;
%  Figure 9: Candle filter problem (results)  ;
%---------------------------------------------;
\begin{figure}
    \centering
    \includegraphics[width=0.8\linewidth]{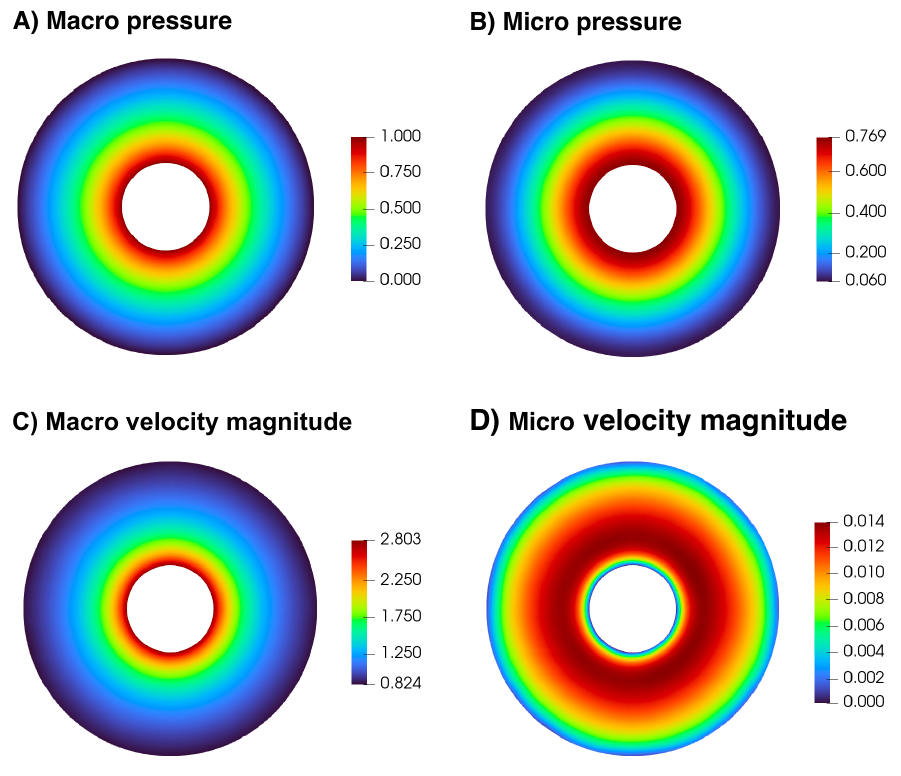}
    \caption{\textsf{Two-dimensional problem with radial symmetry.} This figure shows the solution fields obtained using the proposed modeling framework. Pressure and velocity boundary conditions are imposed in a soft manner---that is, they are not hard-coded into the architecture but are enforced through the loss function. Two key observations are evident: (a) radial symmetry is preserved, and (b) internal flow exists in the micro-pore network despite the absence of boundary discharge, highlighting inter-network exchange.}
\label{Fig:APINNs_Candle_filter_results}
\end{figure}

%-----------------------------------------------------;
%  Figure 10: Candle filter volumetric transfer rate  ;
%-----------------------------------------------------;
\begin{figure}
    \centering
    \includegraphics[width=0.425\linewidth]{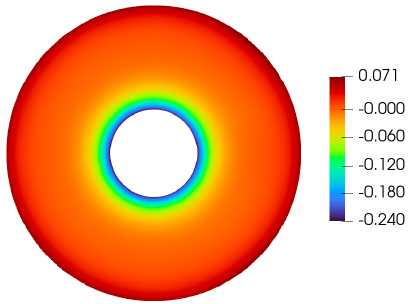}
    \caption{\textsf{Two-dimensional problem with radial symmetry.} This figure shows the volumetric transfer rate, $\chi(\mathbf{x})$, from the micro- to macro-pore network, computed using the proposed modeling  framework. The solution preserves radial symmetry.}
    \label{Fig:APINNs_Candle_filter_volumetric_transfer_results}
\end{figure}

%-------------------------------------;
%  Figure 11: Width-depth refinement  ;
%-------------------------------------;
\begin{figure}
    \centering
    \includegraphics[width=0.9\linewidth]{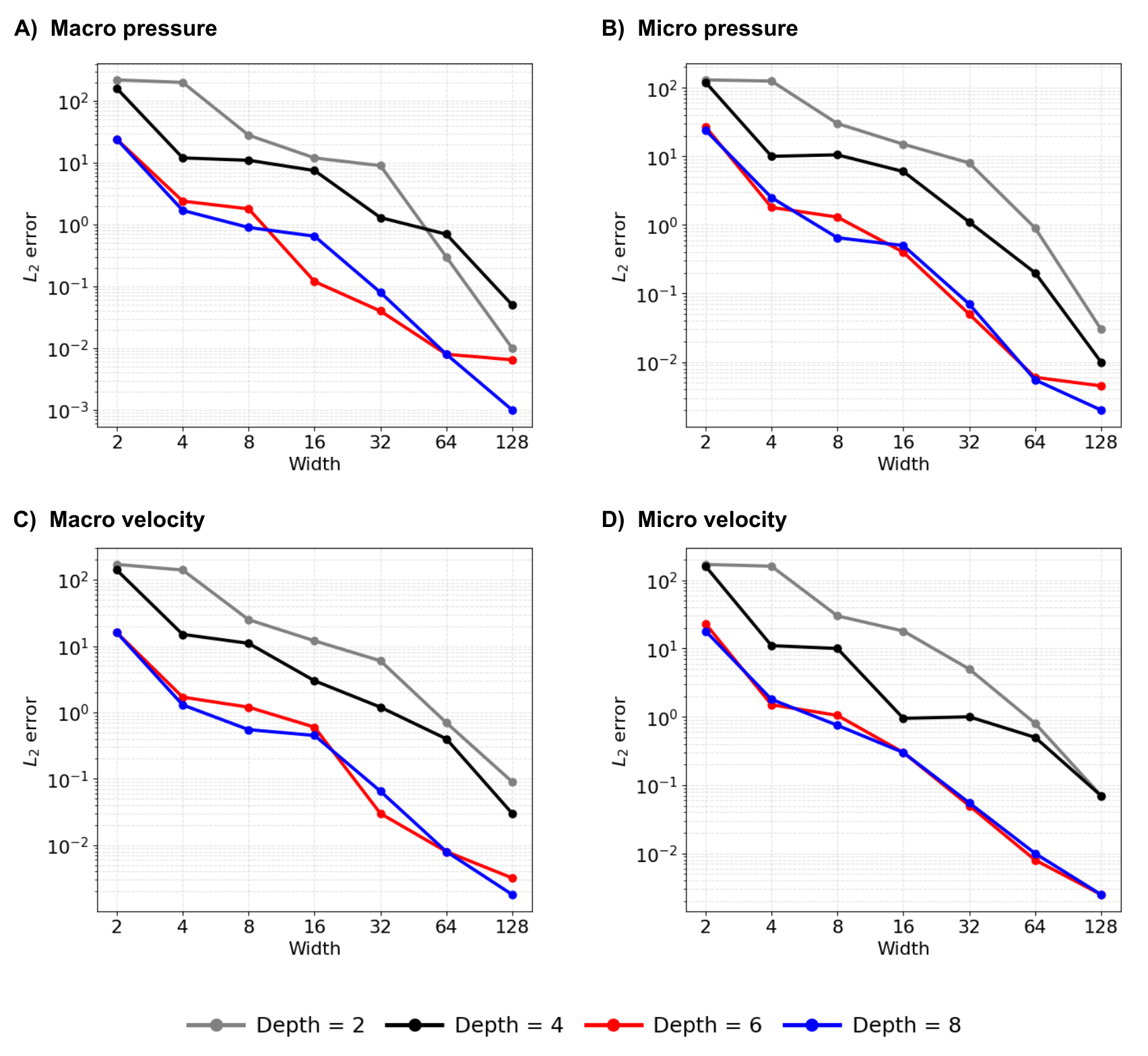}
    \caption{\textsf{Two-dimensional problem with radial symmetry:} Width--depth refinement study. The figure shows the $L_{2}$ errors in the predicted solution fields---macro pressure $p^{(1)}(r)$, micro pressure $p^{(2)}(r)$, and the corresponding velocities: $u^{(1)}(r)$ and $u^{(2)}(r)$---as functions of the neural network width and depth. The results show a systematic reduction in error as the network capacity increases. In particular, architectures with both greater width and depth consistently achieve the lowest errors, indicating improved agreement with the analytical solution.}
    \label{fig:L2-Error_Radial_Symmetry}
\end{figure}

%==============================;
%  Subsection: Layered medium  ;
%==============================;
\subsection{Two-dimensional layered media problem} Many geological formations exhibit abrupt property changes over small regions, varying by orders of magnitude. This boundary value problem highlights the proposed framework's robustness in handling such heterogeneity. Traditional methods like finite elements often require stabilized approaches---such as discontinuous Galerkin (DG) methods \citep{joshaghani2019stabilized}---to ensure stability and accuracy across high-contrast interfaces.

The computational domain consists of five horizontal layers, each 5.0 units wide and 0.8 units high, with distinct macro- and micro-permeabilities. As shown in \textbf{Fig.~\ref{Fig:APINNs_Layered_medium_BVP}}, constant normal velocities are applied at the left and right boundaries of each layer, computed as the permeability-to-viscosity ratio—negative on the left, positive on the right. The top and bottom boundaries impose zero normal components for both macro- and micro-velocities. Neural network parameters are identical to those used in the previous problem (\S\ref{sec:APINNS_Radial_Symmetry}). Model parameters are listed in Table~\ref{Tab:APINNS_2D_Layered_Media}.

 The results, shown in \textbf{Figs. \ref{Fig:APINNs_Layered_media_results_1}} and \textbf{\ref{Fig:APINNs_Layered_media_results_2}}, demonstrate excellent agreement between the proposed framework, the analytical solution, and the DG-based numerical results from \citet{joshaghani2019stabilized}. This strong alignment validates the framework's effectiveness in capturing discontinuities and layer-wise behavior in heterogeneous porous media, eliminating the need for a DG formulation. \textbf{Figure \ref{fig:APINNS_Layered_Media_Loss_vs_Ephos}} shows the total loss (log scale) over training epochs, with a sharp initial drop and gradual improvements during RAR rounds, highlighting the robust convergence of the proposed modeling framework. These results were obtained using a neural network with eight hidden layers and 128 neurons per layer, trained with the \texttt{ReLU} activation function; the total training time on an NVIDIA T4 GPU was 6.28 minutes.

%-----------------------------------;
%  Figure 12: Layered medium (BVP)  ;
%-----------------------------------;
\begin{figure}[h]
    \centering
    \includegraphics[width=0.9\linewidth]{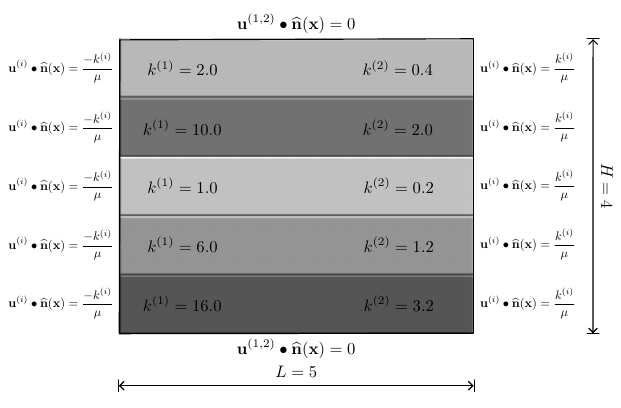}
    \caption{\textsf{Two-dimensional layered media problem.} This figure illustrates the computational domain, boundary conditions, and the macro- and micro-permeabilities of each layer.}
    \label{Fig:APINNs_Layered_medium_BVP}
\end{figure}

%------------------------------------;
%  Table 4: Layered medium (params)  ;
%------------------------------------;
\begin{table}[htbp]
    \centering
    \caption{\textsf{Two-dimensional layered media problem.} Numerical simulation parameters. \label{Tab:APINNS_2D_Layered_Media}}
    \begin{tabular}{|c|c|}
        \hline
        \textbf{Parameter} & \textbf{Value} \\ \hline
        $L_{x}$         & 5 \\ 
        $L_{y}$         & 4 \\ 
        $\beta$     & 1 \\ 
        $\mu$       & 1 \\ 
        $\gamma \mathbf{b}$  & $\mathbf{0} = [0, 0]$ \\ 
        $k^{(1)}$ and $k^{(2)}$ & Provided in Fig.~\ref{Fig:APINNs_Layered_medium_BVP} \\
        \hline
    \end{tabular}
\end{table}

%------------------------------------------------;
%  Figure 13: Layered medium (results contours)  ;
%------------------------------------------------;
\begin{figure}[h]
    \centering
    \includegraphics[width=0.85\linewidth]{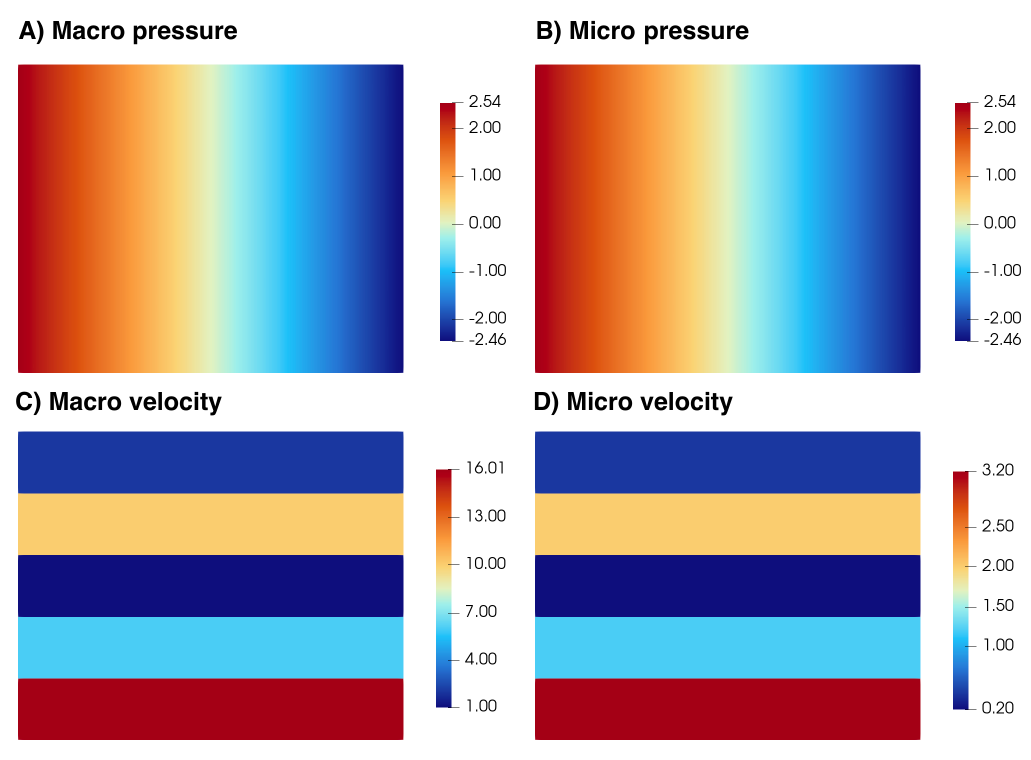}
    \caption{\textsf{Two-dimensional layered media problem.} This figure shows solution fields from the proposed model. Velocities are constant within each layer, and pressures vary linearly along the horizontal axis, matching the exact solution. This confirms the framework passes the velocity-driven patch test.}
    \label{Fig:APINNs_Layered_media_results_1}
\end{figure}

%--------------------------------------------------;
%  Figure 14: Layered medium (results line plots)  ;
%--------------------------------------------------;
\begin{figure}[h!]
    \centering
    \includegraphics[width=0.8\linewidth]{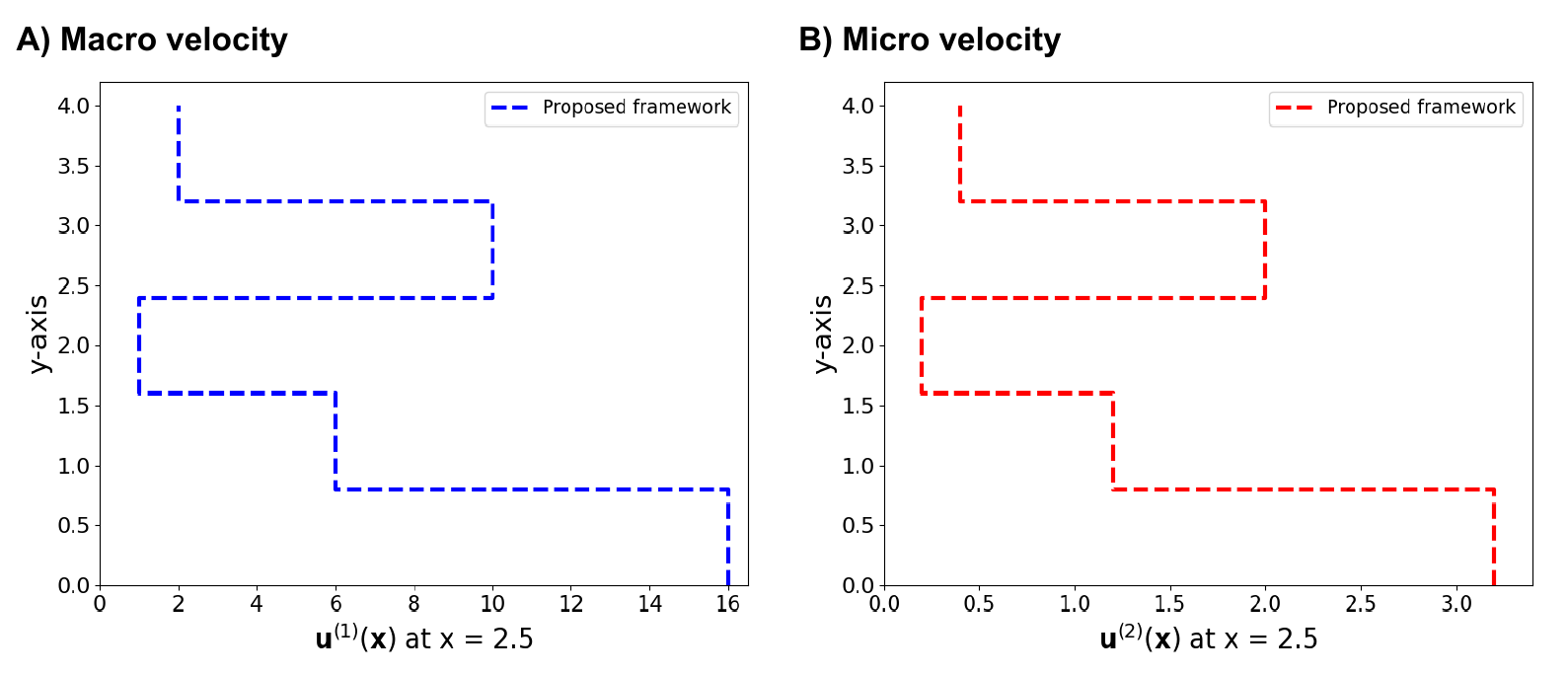}
    \caption{\textsf{Two-dimensional layered media problem.} This figure shows macro- and micro-pore velocity profiles at $x = 2.5$ from the proposed framework, aligning well with the DG formulation by \citet{joshaghani2019stabilized}.}
    \label{Fig:APINNs_Layered_media_results_2}
\end{figure}

%----------------------------------------------;
%  Figure 15: Layered medium (error log plot)  ;
%----------------------------------------------;
\begin{figure}
    \centering
    \includegraphics[width=0.8\linewidth]{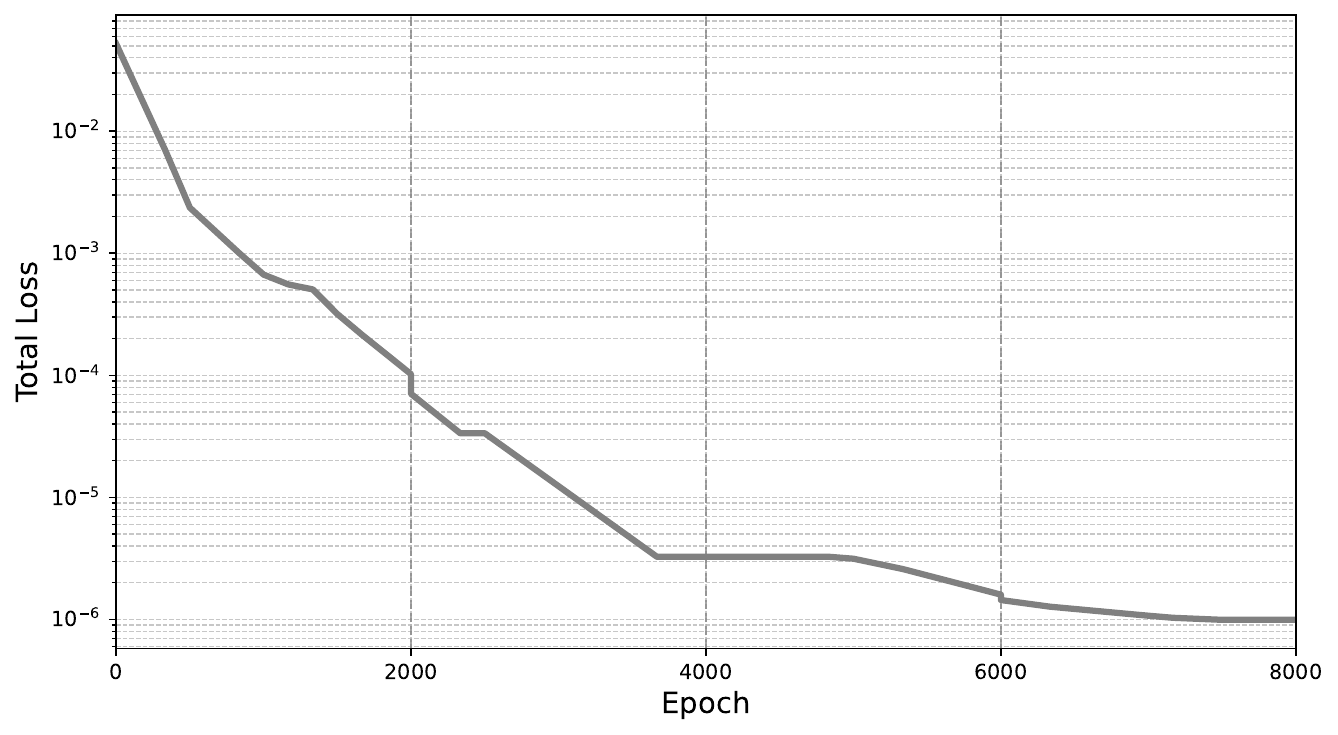}
    \caption{\textsf{Two-dimensional layered media problem.} The plot shows total loss for the layered-media problem under residual-based adaptive refinement (RAR). The gray curve shows a smooth, monotone decay of the training objective over epochs, with vertical dashed lines marking the transitions between successive RAR rounds. The network comprises eight hidden layers with 128 neurons per layer and a {\color{red}ReLU} activation function.}
    \label{fig:APINNS_Layered_Media_Loss_vs_Ephos}
\end{figure}

%========================================;
%  Subsection: 2D strip footing problem  ;
%========================================;
\subsection{Two-dimensional strip footing problem}
In geotechnical settings such as strip footings and retaining walls, overburden pressure induces flow-driving gradients. The boundary value problem considered here serves as a key benchmark for evaluating the proposed framework’s ability to handle such practical scenarios. 

Consider a rectangular domain of length $L = 10 \, \text{m}$ and height $H = 5 \, \text{m}$, as shown in \textbf{Fig.~\ref{Fig:APINNs_Footing_BVP}}. Zero normal velocity is imposed on the left, right, and bottom boundaries for both macro- and micro-pore flow. The top boundary is divided into three segments: $T_{1}\;(0 \le x \le 2.5\,\text{m})$, $T_{2}$ $(2.5 < x \le 7.5\,\text{m})$, and $T_{3}$ $(7.5 < x \le 10 \, \text{m})$, with boundary conditions specified in the figure. Model parameters used in the simulation are listed in Table~\ref{Tab:APINNS_2D_strip_footing}.

\textbf{Figure~\ref{Fig:APINNs_Footing_Results_P_and_V}} illustrates the predicted field variables. \textbf{Figure~\ref{Fig:APINNs_Footing_Results_Chi}} shows the predicted volumetric transfer rate and the corresponding absolute error with respect to the finite element formulation as in \citep{joodat2018modeling}. These results were obtained using a neural network with eight hidden layers and 256 neurons per layer, trained with the \texttt{ReLU} activation function; the total training time on an NVIDIA T4 GPU was 9.36 minutes.

%------------------------------------------;
%  Figure 16: Strip footing problem (BVP)  ;
%------------------------------------------;
\begin{figure}[h!]
    \centering
    \includegraphics[width=0.75\linewidth]{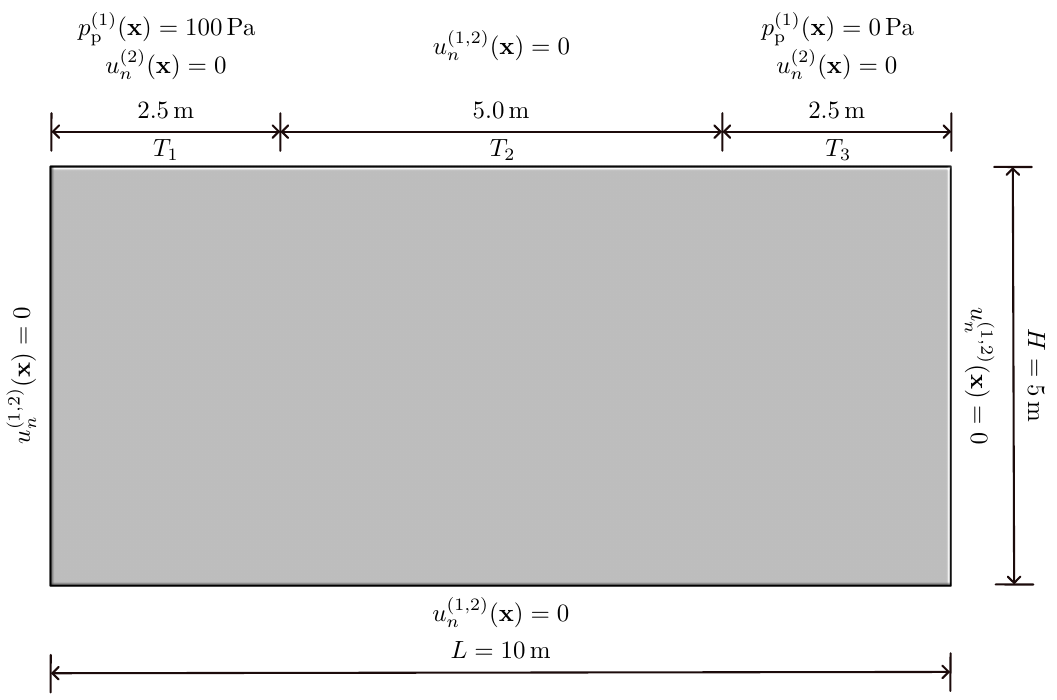}
    \caption{\textsf{Two-dimensional strip footing problem.} Schematic of the boundary value problem. Along the left, right, and bottom edges the normal components of the macro- and micro-pore velocities are zero ($\mathbf u^{(1)}(\mathbf{x})\, \bullet \, \widehat{\mathbf{n}}(\mathbf{x}) = \mathbf u^{(2)}(\mathbf{x})\, \bullet \, \widehat{\mathbf{n}}(\mathbf{x}) = 0$), rendering these boundaries impermeable. 
    The top boundary is divided into three segments: (i) $T_{1}\;(0 \le x \le 2.5\,\text{m})$, where a pressure $p^{(1)}(\mathbf{x}) =p^{(1)}_{\mathrm{p}}(\mathbf{x})  = 100 \,\text{Pa}$ is prescribed for the macro-pore network, and the micro-pore flux is set to zero ($ \mathbf{u}^{(2)}(\mathbf{x})\, \bullet \, \widehat{\mathbf{n}}(\mathbf{x}) = 0$). (ii) $T_{2}\;(2.5 < x \le 7.5\;\text{m})$, where the normal components of both macro/micro-pore velocities are zero ($\mathbf u^{(1,2)}(\mathbf{x})\, \bullet \, \widehat{\mathbf{n}}(\mathbf{x}) = 0$). (iii) $T_{3}\;(7.5 < x \le 10\;\text{m})$, where the macro-pore pressure is vented to the reference level $p^{(1)}(\mathbf{x}) = p^{(1)}_{\mathrm{p}}(\mathbf{x}) = 0\,\text{Pa}$, while the normal component of micro-pore velocity is zero ($\mathbf u^{(2)}(\mathbf{x})\, \bullet \, \widehat{\mathbf{n}}(\mathbf{x}) = 0$).}
    \label{Fig:APINNs_Footing_BVP}
\end{figure}

%--------------------------------------------;
%  Table 5: 2D strip footing problem params  ;
%--------------------------------------------;
\begin{table}[htbp]
    \centering
    \caption{\textsf{Two-dimensional strip footing problem.} Numerical simulation parameters. \label{Tab:APINNS_2D_strip_footing}}
    \begin{tabular}{|c|c|}
        \hline
        \textbf{Parameter} & \textbf{Value} \\ \hline
        $L$         & 10 \\ 
        $H$         & 5 \\ 
        $\beta$     & 1 \\ 
        $\mu$       & 1 \\ 
        $\gamma \mathbf{b}$  & $\mathbf{0} = [0, 0]$ \\ 
        $k^{(1)}$   & 1000 \\
        $k^{(2)}$   & 10 \\
        \hline
    \end{tabular}
\end{table}

%-----------------------------------------;
%  Figure 17: Footing (results: V and P)  ;
%-----------------------------------------;
\begin{figure}
    \centering
    \includegraphics[width=0.9\linewidth]{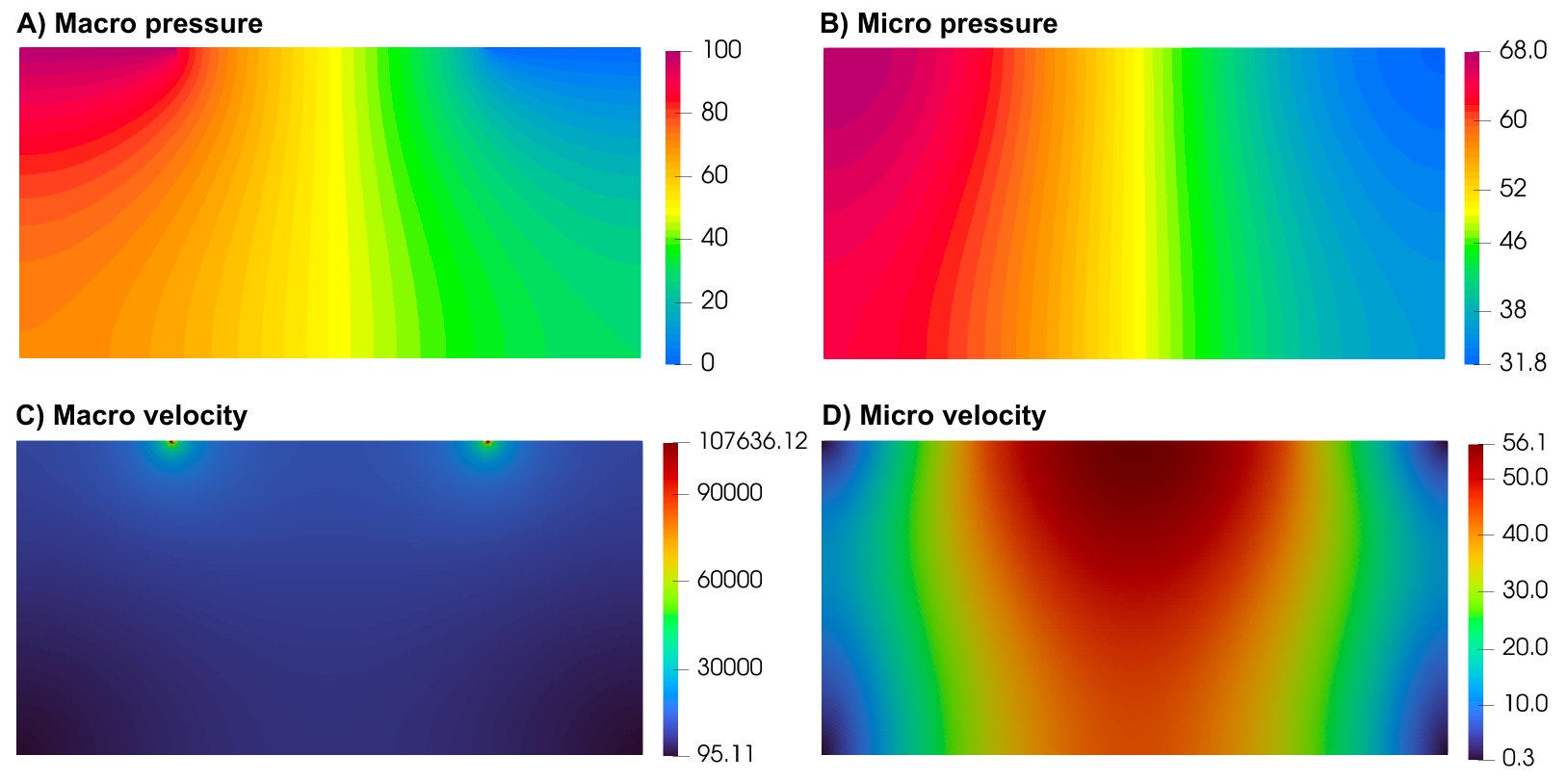}
    \caption{\textsf{Two-dimensional strip footing problem.} This figure shows the profiles of the solution fields under the proposed modeling framework. \label{Fig:APINNs_Footing_Results_P_and_V}}
\end{figure}

%-------------------------------------;
%  Figure 18: Footing (results: Chi)  ;
%-------------------------------------;
\begin{figure}
    \centering
    \includegraphics[width=0.9\linewidth]{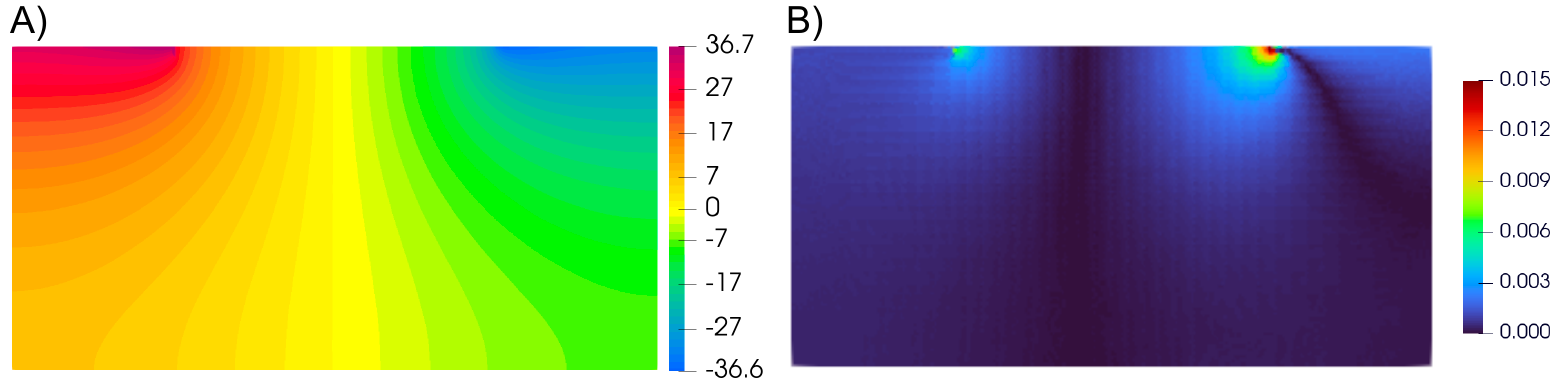}
    \caption{\textsf{Two-dimensional strip footing problem.} A) Volumetric transfer rate predicted by the proposed framework. B) Absolute error in the volumetric flow rate relative to the reference finite-element solution obtained using the stabilized mixed formulation of \citet{joodat2018modeling}, demonstrating close agreement.}
    \label{Fig:APINNs_Footing_Results_Chi}
\end{figure}

To further assess the robustness of the proposed framework, we performed an ablation study in which its principal ingredients are introduced progressively and their influence on the predicted solution is examined. The study is carried out using the layered-media problem, whose sharp permeability contrasts provide a sensitive test of interface resolution and convergence. The resulting velocity profiles and relative \(L^2\) errors quantify the contribution of the shared representation, Fourier encoding, adaptive sampling, and adaptive weighting. Details of the study are provided in Appendix~\ref{App:Ablation}. 
    \newpage
    %*********************************************;
%                                             ;
%  NAME                                       ;
%    S6_APINNs_Inversion.tex                  ;
%                                             ;
%  WRITTEN BY                                 ;
%    Kalyana B. Nakshatrala                   ;
%                                             ;
%*********************************************;
\section{INVERSION FRAMEWORK}
\label{Sec:S6_APINNs_IF}
As noted earlier, porous media models often involve parameters that are not directly measurable and must therefore be inferred indirectly from observable, measured quantities. Measurements such as the total flux at a production well reflect the integrated response of the coupled flow system and thereby encode information about the underlying transfer mechanisms. Leveraging such information enables the estimation of parameters that are otherwise difficult to measure directly. In the context of DPP models, one such difficult-to-measure parameter is the inter-porosity mass transfer coefficient $\beta$.

This motivates the development of an inversion framework to estimate $\beta$ from indirect observations while adhering to the governing equations of the system. Accordingly, we formulate the identification of $\beta$ as an inverse problem, wherein observational data are integrated with the governing equations to enable parameter estimation while maintaining consistency with the underlying physical laws.

Let $\mathbb{U}(\mathbf{x})$ denote the collection of field variables and $\hat{\mathbb{U}}_{\boldsymbol{\theta},\, \beta}(\mathbf{x})$ denote the neural network surrogate parameterized by $\boldsymbol{\theta}$ and conditioned on $\beta$, that is:
\begin{align}
    \hat{\mathbb{U}}_{\boldsymbol{\theta},\, \beta}(\mathbf{x}) = \Big(
\hat p^{(1)}(\mathbf{x};\theta,\beta),\,
\hat p^{(2)}(\mathbf{x};\theta,\beta),\,
\hat{\mathbf u}^{(1)}(\mathbf{x};\theta,\beta),\,
\hat{\mathbf u}^{(2)}(\mathbf{x} ;\theta,\beta)
\Big)
\end{align}
where $\boldsymbol{\theta}$ collects all trainable parameters while $\beta$ is treated as an additional trainable scalar constrained to the admissible set $\mathcal B\subset(0,\infty)$. In the inverse setting, $\beta$ enters the model through the governing-equation residuals and is inferred jointly with $\boldsymbol{\theta}$. The loss function formulation is the same as the forward case, as illustrated in Eq.~\eqref{Eqn:Abstarct_Total_Loss}; the only change is that these residuals are now evaluated using the parameter-conditioned surrogate $\hat{\mathbb{U}}_{\boldsymbol{\theta},\, \beta}$. We next augment the loss functional with an observation-misfit term through an appropriately defined observation operator. Let $y_{\mathrm{obs}}$ denote the available observations and let $\mathcal{Q}$ be an observation operator that maps the predicted state
\(\hat{\mathbb U}_{\boldsymbol{\theta},\beta}\) to the measured quantity.
We define the observation misfit as
\begin{align}
\mathcal L_{\mathrm{OBS}}(\boldsymbol{\theta},\beta)
=
\left\|
\mathcal{Q}\!\left(\hat{\mathbb U}_{\boldsymbol{\theta},\beta}\right)
-
y_{\mathrm{obs}}
\right\|_2^2
\end{align}
In particular, for flux-type observations over a measurement boundary segment \(\Gamma_{\mathrm{obs}}\subset\partial\Omega\),
we take
\begin{align}
\mathcal{Q}\!\left(\hat{\mathbb U}_{\boldsymbol{\theta},\beta}\right)
&=
\int_{\Gamma_{\mathrm{obs}}}
\mathbf u_{\mathrm{tot}}(\mathbf{x};\boldsymbol{\theta},\beta)\bullet\widehat{\mathbf{n}}(\mathbf{x})\, \mathrm{d}\Gamma
\\
\mathbf u_{\mathrm{tot}}(\mathbf{x};\boldsymbol{\theta},\beta)
&=
\hat{\mathbf u}^{(1)}(\mathbf{x};\boldsymbol{\theta},\beta)
+
\hat{\mathbf u}^{(2)}(\mathbf{x};\boldsymbol{\theta},\beta)
\end{align}
where $\widehat{\mathbf{n}}(\mathbf{x})$ denotes the outward unit normal vector. Now the final loss function is defined by
\begin{align}
\mathcal J_{\mathrm{inv}}(\boldsymbol{\theta},\beta)
=
\lambda_{\text{PDE}} \, \mathrm{Loss}_{\text{PDE}}(\boldsymbol{\theta})
+ \lambda_{\text{BC}} \, \mathrm{Loss}_{\text{BC}}(\boldsymbol{\theta})
+ \lambda_{\text{OBS}}\,\mathcal L_{\text{OBS}}(\boldsymbol{\theta},\beta)
\end{align}
where $\lambda_{\text{PDE}}$, $\lambda_{\text{BC}}$, and $\lambda_{\text{OBS}}$ are nonnegative weights
chosen according to the weighting strategy described in the forward formulation. The parameter identification problem is posed as the joint minimization
\begin{align}
    (\boldsymbol{\theta}^\star,\beta^\star)
    =
    \arg\min_{\boldsymbol{\theta}\in\Theta,\;\beta\in\mathcal B}
    \mathcal J_{\mathrm{inv}}(\boldsymbol{\theta},\beta)
\end{align}

%============================;
%  Subsection: Verification  ;
%============================;
\subsection{Verification of the inversion framework}
We consider a two-dimensional boundary value problem to assess the effectiveness of the proposed inversion framework. The flow domain is a rectangle where the motion is driven by pressure applied on the boundaries. \textbf{Figure~\ref{Fig:APINNs_Inverse_BVP}} illustrates the boundary conditions, including no-flow conditions on the top and bottom boundaries. Such configurations commonly arise in subsurface applications, including groundwater extraction, enhanced oil recovery, and filtration. The model parameters used in this study are summarized in Table~\ref{Table:APINNS_Inversion_Rectangle_BVP}. The results were obtained using a neural network with eight hidden layers and 256 neurons per layer, trained with the \texttt{ReLU} activation function; the total training time on an NVIDIA T4 GPU was 10.28 minutes.

\textbf{Figure~\ref{fig:APINNS_Pressure_Drive_Results}} evaluates the performance of the proposed inversion framework for the pressure-driven problem. Figure~\ref{fig:APINNS_Pressure_Drive_Results}(A) shows forward finite element simulations describing the relationship between the inter-porosity mass transfer coefficient $\beta$ and the production well flux $Q$. The results demonstrate a clear dependence of $Q$ on $\beta$, providing a reference mapping for validating the inversion approach. Figure~\ref{fig:APINNS_Pressure_Drive_Results}(B) presents the corresponding inversion results, where $\beta$ is estimated from a given flux value $Q$. The predicted $\beta$–$Q$ relationship closely agrees with the finite element solution, indicating that the inversion framework successfully captures the underlying parameter–response behavior.

\textbf{Figure~\ref{fig:APINNS_Pressure_Drive_Results_error}} further evaluates the accuracy of the recovered parameter by comparing the volumetric transfer rate predicted by the inversion framework with the reference finite element solution for $\beta = 1$. The predicted transfer rate closely matches the finite element result across the domain, with only minor discrepancies. This agreement confirms that the proposed framework accurately recovers the governing parameter and reproduces the corresponding flow behavior. This spatial agreement is obtained without supplying the transfer-rate distribution as training data. The proposed
formulation determines $\beta$ jointly with the pressure and velocity fields in both pore networks within a single
optimization, without requiring a precomputed parameter--flux map. A scalar lookup can estimate $\beta$ for the prescribed configuration, but recovery of the associated fields requires an additional forward solution or a separate field representation. The present verification therefore assesses both coefficient identification and spatial reconstruction
beyond the integrated observation used for training.

%--------------------------------;
%  Fig 19: APINNS_Inversion_BVP  ;
%--------------------------------;
\begin{figure}[h!]
    \centering
    \includegraphics[width=0.8\linewidth]{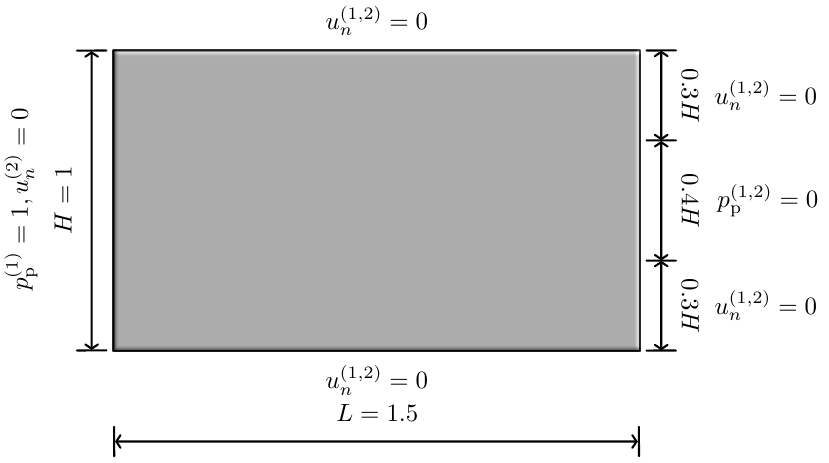}
    \caption{\textsf{Two-dimensional pressure-driven problem---Inversion.} Computational domain of length $L = 1.5$ and height $H = 1.0$. At the left boundary (i.e., $x=0$), a macro-pressure $p_{\mathrm{p}}^{(1)} = 1$ is prescribed with $u_n^{(2)} = 0$. At the right boundary (i.e., $x=L$), zero macro- and micro-pressures ($p_{\mathrm{p}}^{(1)} = p_{\mathrm{p}}^{(2)} = 0$) are imposed on the central segment of height $0.4H$, while $u_n^{(1,2)} = 0$ is applied on the remaining portion. The top and bottom boundaries satisfy no-flux conditions for both pore networks, representing impermeable surfaces. \label{Fig:APINNs_Inverse_BVP}}
\end{figure}

%-----------------------------------;
%  Table 6: APINNS_Inversion_Table  ;
%-----------------------------------;
\begin{table}[htbp]
    \centering
    \caption{\textsf{Two-dimensional pressure-driven problem---Inversion}. Numerical simulation parameters; see Fig.~\ref{Fig:APINNs_Inverse_BVP} for additional details. 
    \label{Table:APINNS_Inversion_Rectangle_BVP}}
    \begin{tabular}{|c|c|}
        \hline
        \textbf{Parameter} & \textbf{Value} \\ \hline
        $\mu$       & 1 \\ 
        $\gamma \mathbf{b}$  & $\mathbf{0} = [0, 0]$ \\ 
        $k^{(1)}$   & 1000 \\
        $k^{(2)}$   & 10 \\
        \hline
    \end{tabular}
\end{table}

%----------------------------------;
%  Figure 20: Inversion result #1  ;
%----------------------------------;
\begin{figure}[h!]
    \centering
    \includegraphics[width=0.9\linewidth]{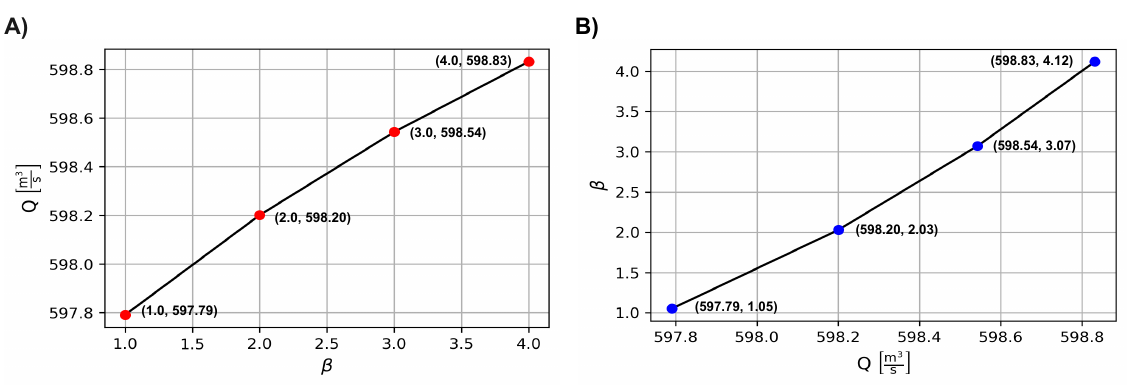}
    \caption{\textsf{Two-dimensional pressure-driven problem---Inversion.} The figure illustrates the accuracy of the proposed inversion framework by comparing its predicted solutions with those obtained from the forward problem, solved using the finite element method (FEM). \textbf{(A)} shows the forward solution, where for various values of $\beta$, the volumetric flux at the production well $Q$ is computed. \textbf{(B)} shows the corresponding inverse solution obtained with the proposed inversion framework, where for a given $Q$ the corresponding $\beta$ is recovered. The agreement between the $\beta$--$Q$ relationships from the inversion framework and the finite element-based forward problem demonstrates the effectiveness of the method in recovering system parameters from observed data.}
    \label{fig:APINNS_Pressure_Drive_Results}
\end{figure}

%----------------------------------;
%  Figure 21: Inversion result #2  ;
%----------------------------------;
\begin{figure}[h!]
    \centering
    \includegraphics[width=1.0\linewidth]{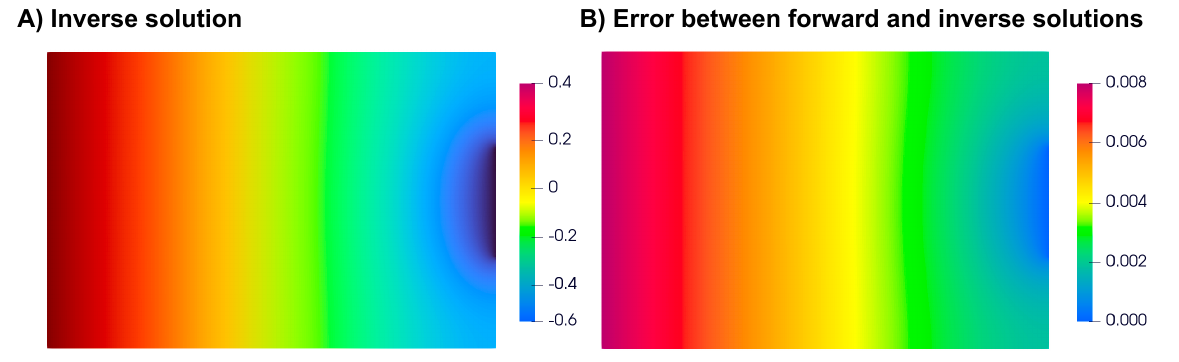}
    \caption{\textsf{Two-dimensional pressure-driven problem---Inversion.} The figure illustrates the estimated volumetric transfer rate along with the discrepancy between the forward and inverse solutions for $\beta = 1.0$. This comparison highlights the capability of the proposed framework to accurately recover transfer characteristics, as evidenced by the small error relative to the reference forward solution.}
    \label{fig:APINNS_Pressure_Drive_Results_error}
\end{figure}

%================================================;
%  Subsection: Sensitivity to observation noise  ;
%================================================;
\subsection{Sensitivity to observation noise}
\label{Sec:APINNS_Inverse_Noise}

The preceding verification concerns noise-free flux observations. We now examine how errors in these observations affect the identified transfer coefficient, with all remaining problem data held fixed. 

We perturb the production-flux observation for the preceding
pressure-driven problem at $\beta_{\mathrm{true}}=3$. Using the reference
finite element flux $Q_{\mathrm{ref}}=598.54$, we prescribe
\begin{equation}
    Q_{\mathrm{obs}}(\eta)
    :=
    (1+\eta)Q_{\mathrm{ref}}
    \label{Eqn:APINNS_Noise_Numerical_Observations}
\end{equation}
where $\eta$ is the signed relative observation error.
The seven cases in
Table~\ref{Table:APINNS_Noise_Recovery}
comprise the noise-free observation and positive and
negative perturbations of magnitude $0.01\%$, $0.05\%$,
and $0.1\%$. Each observation is held fixed during its
inversion.

For each perturbed observation $Q_{\mathrm{obs}}(\eta)$,
let $\beta_{\mathrm{FEM}}(\eta)$ denote the coefficient
estimated from the finite element parameter--flux relation
and $\beta_{\mathrm{NN}}(\eta)$ is the inverse solution from the proposed framework.

The comparison distinguishes the coefficient change
implied by the reference response from the discrepancy
between the two inversion procedures. The departure of
$\beta_{\mathrm{FEM}}(\eta)$ from $\beta_{\mathrm{true}}$ reflects
the response of the reference inverse problem to the
observation error, whereas the difference between
$\beta_{\mathrm{NN}}(\eta)$ and $\beta_{\mathrm{FEM}}(\eta)$ assesses agreement
with that response. The total relative recovery error
against the true coefficient is reported as
\begin{equation}
    e_\beta(\eta)
    :=
    100\,
    \frac{
        |\beta_{\mathrm{NN}}(\eta)-\beta_{\mathrm{true}}|
    }{
        \beta_{\mathrm{true}}
    }
    \label{Eqn:APINNS_Noise_Parameter_Error}
\end{equation}

\begin{table}[htbp]
    \centering
    \caption{\textsf{Sensitivity to production-flux
    observation errors.}
    Reference finite element estimates $\beta_{\mathrm{FEM}}(\eta)$
    and coefficients $\beta_{\mathrm{NN}}(\eta)$ recovered by the
    proposed framework from the same perturbed observations.
    The true coefficient remains $\beta_{\mathrm{true}}=3$;
    $e_\beta(\eta)$ measures departure from this value.}
    \label{Table:APINNS_Noise_Recovery}

    \begingroup
    \small
    \setlength{\tabcolsep}{8pt}
    \renewcommand{\arraystretch}{1.15}
    \begin{tabular}{@{}rrrrr@{}}
        \toprule
        $100\eta$
        & $Q_{\mathrm{obs}}(\eta)$
        & $\beta_{\mathrm{FEM}}(\eta)$
        & $\beta_{\mathrm{NN}}(\eta)$
        & $e_\beta(\eta)$
        \\
        $(\%)$
        & $(\mathrm{m}^{3}/\mathrm{s})$
        & & & $(\%)$
        \\
        \midrule
        $\mathbf{0}$
        & $\mathbf{598.540000}$
        & $\mathbf{3.0000}$
        & $\mathbf{3.0700}$
        & $\mathbf{2.33}$
        \\
        $-0.01$ & $598.480146$ & $2.8240$ & $2.8869$ & $3.77$  \\
        $0.01$  & $598.599854$ & $3.2064$ & $3.2867$ & $9.56$  \\
        $-0.05$ & $598.240730$ & $2.1198$ & $2.1546$ & $28.18$ \\
        $0.05$  & $598.839270$ & $4.0320$ & $4.1536$ & $38.45$ \\
        $-0.10$ & $597.941460$ & $1.3694$ & $1.4120$ & $52.93$ \\
        $0.10$  & $599.138540$ & $5.0639$ & $5.2371$ & $74.57$ \\
        \bottomrule
    \end{tabular}
    \endgroup
\end{table}

The results reflect the weak relative dependence of the
production flux on $\beta$ under the prescribed loading.
In Fig.~\ref{fig:APINNS_Pressure_Drive_Results}(A),
increasing $\beta$ from $3$ to $4$ changes the flux
by only approximately $0.048\%$. Consequently, small
relative flux perturbations can imply substantial
coefficient changes. For a $0.1\%$ increase in the
observation, the reference estimate is $5.0639$,
compared with $5.2371$ from the proposed framework.
Across all seven cases, the discrepancy between the
two estimates is at most $3.42\%$, relative to
$\beta_{\mathrm{FEM}}(\eta)$. The large departures from the true
coefficient are thus already present in the reference
inverse response, which the proposed framework closely
reproduces. They should not be interpreted as neural
approximation errors alone.

The present tests concern a spatially constant transfer
coefficient with prescribed permeabilities. They do not
establish identifiability of $\beta(\mathbf{x})$ or
$\mathbf{K}^{(i)}(\mathbf{x})$, $i\in\{1,2\}$.
Even a two-region scalar parameterization introduces
two unknown values, whereas one integrated flux supplies
at most one independent linearized constraint.
Future studies will examine spatial transfer-coefficient
and permeability identification using suitable
parameterizations and additional observations, including
local pressures and separately measured network fluxes
under distinct spatial loading configurations.
Their sufficiency must be assessed for the chosen
parameterization, particularly when changes in network
conductivity and inter-network exchange are inferred
simultaneously. Positivity constraints and regularization
restrict the admissible reconstructions but do not
supply additional observational information.
    %*********************************************;
%                                             ;
%  NAME                                       ;
%    S7_APINNs_Closure.tex                    ;
%                                             ;
%  WRITTEN BY                                 ;
%    Kalyana B. Nakshatrala                   ;
%                                             ;
%*********************************************;
\section{CLOSURE}
\label{Sec:S7_APINNs_Closure}

We have presented a PINN-based framework for modeling flow through porous media with dual pore networks. The underlying mathematical formulation, expressed in strong form, consists of four coupled partial differential equations supplemented by velocity and pressure boundary conditions.

The proposed framework incorporates several key features. First, boundary conditions---both velocity and pressure---are enforced in a soft manner, following standard practices in machine learning. Second, to handle the coupled structure of the governing equations, we employ a shared-trunk neural network architecture with adaptive weighting assigned to each residual term, including those associated with the boundary conditions. Finally, to better resolve sharp features in the solution fields, we utilize an adaptive distribution of collocation points, enabling improved accuracy in regions with high gradients.

Some of the main features of the framework are the following: 
\begin{enumerate}
   \item \textbf{Meshless formulation:} The approach is inherently meshless, making it particularly well-suited for problems involving complex geometries. It eliminates the need to generate meshes that conform to intricate domain boundaries—an advantage especially relevant to porous media, which often exhibit highly irregular structures.
   \item \textbf{Accurate handling of discontinuities:} It effectively captures discontinuities in the solution fields that arise across layered porous media. In contrast, finite element formulations based on the Bubnov--Galerkin method are prone to spurious oscillations—resembling the Gibbs phenomenon \citep{joodat2018modeling}.
   \item \textbf{Free from \emph{inf-sup} instabilities:} The proposed framework does not suffer from \emph{inf-sup} instabilities, which are typically addressed by the Ladyzhenskaya–Babu\v{s}ka–Brezzi (LBB) condition. In comparison, classical Galerkin finite element methods require specially constructed interpolation spaces (e.g., Raviart–Thomas spaces) to satisfy the LBB condition \citep{brezzi2012mixed}.
   \item \textbf{No need for stabilization:} To bypass the LBB condition, finite element methods often introduce stabilization terms into the Galerkin weak form \citep{masud2002stabilized,joodat2018modeling}. While these stabilized formulations can suppress spurious oscillations, the proposed PINN-based framework achieves stability and accuracy without the need for such additional terms.
  \item \textbf{Well-suited for both forward and inverse problems:} The framework is applicable to both forward and inverse problem settings. This is particularly advantageous in porous media modeling, where many parameters---such as those related to volumetric flow rate---are difficult to measure directly and are often estimated empirically. By enabling parameter identification, the framework reduces reliance on empirical approximations.
\end{enumerate}

Future work may be organized along two main directions: (i) theoretical developments, and (ii) the modeling of porous-media flow with increasingly complex physical mechanisms. 
\begin{enumerate}
    \item[(i)] From the theoretical perspective, a key priority is the development of a rigorous foundation for inverse problems, motivated in part by the growing use of PINN and related methods. In particular, there is a need to identify and characterize classes of boundary value problems that are well-posed in the inverse sense. Advancing this direction will require deeper results on existence, uniqueness, and stability of solutions, along with an improved understanding of how these properties influence the performance and reliability of PINN-based approaches. 
    \item[(ii)] From the modeling perspective, future efforts may focus on incorporating more realistic and nonlinear flow phenomena. This includes extensions to models accounting for pressure-dependent viscosity, such as the Barus model, and inertial effects as described by the Darcy--Forchheimer formulation \citep{chang2017modification, maduri2025flow}. Additionally, coupling between fluid flow and solid deformation in porous media---such as in poromechanical systems governed by Biot's equations---represents an important and challenging direction \citep{allen2021mathematics}.
\end{enumerate}

    %====================================;
    %  Include all the appendices below  ;
    %------------------------------------; 
    \appendix
    %*********************************************;
%                                             ;
%  APPENDIX: ABLATION STUDY                  ;
%                                             ;
%*********************************************;
\section{ABLATION STUDY}
\label{App:Ablation}
To isolate the contributions of the individual ingredients of the proposed framework, we perform an ablation study using the layered porous-medium problem described in Section~\ref{Sec:S5_APINNs_NR}. This problem provides a stringent test because the discontinuous permeability field produces sharp layer-wise variations in the macro- and micro-network velocities. The ablation study progressively activates the principal components of the proposed framework while using the same optimization budget. 

The five cases are defined cumulatively as follows:
%----------------------------;
%  Equation: Ablation cases  ;
%----------------------------;
\begin{align*}
    \mathrm{C1}:&\quad
    \text{baseline neural network} \\
    \mathrm{C2}:&\quad
    \text{shared trunk with field-specific heads} \\
    \mathrm{C3}:&\quad \mathrm{C2} + \text{Fourier encoder} \\
    \mathrm{C4}:&\quad \mathrm{C3}
    +\text{adaptive sampling} \\
    \mathrm{C5}:&\quad
    \mathrm{C4} +\text{adaptive loss weighting}
\end{align*}
The accuracy of the predicted velocity fields is measured using the relative
\(L^2(\Omega)\) error:
\begin{align}
    e_{\mathbf{u}^{(i)}}^{\mathrm{rel}}
    = \frac{
    \left\|
    \mathbf{u}^{(i)}_{\mathrm{NN}}
    - \mathbf{u}^{(i)}_{\mathrm{ref}}
    \right\|_{L^2(\Omega)}
    }{
    \left\|
    \mathbf{u}^{(i)}_{\mathrm{ref}}
    \right\|_{L^2(\Omega)}
    }
    \qquad i=1,2
\end{align}
where \(i=1\) and \(i=2\) denote the macro- and micro-pore networks, respectively,
\(\mathbf{u}^{(i)}_{\mathrm{NN}}\) is the predicted velocity field, and
\(\mathbf{u}^{(i)}_{\mathrm{ref}}\) is the corresponding reference solution.
\begin{table}[!htb]
\centering
\caption{Relative \(L^2(\Omega)\) errors in the macro- and micro-network velocity fields for the ablation cases.}
\label{Tab:Ablation_errors}

\renewcommand{\arraystretch}{1.20}

\begin{tabular}{|c | c | c |}
\hline
\textbf{Case}
&
\(e_{\mathbf{u}^{(1)}}^{\mathrm{rel}}\)
&
\(e_{\mathbf{u}^{(2)}}^{\mathrm{rel}}\)
\\
\hline

C1 & \(3.2847\times10^{-1}\) & \(1.1264\times10^{0}\) \\
C2 & \(1.1473\times10^{-2}\) & \(4.3852\times10^{-1}\) \\
C3 & \(7.8731\times10^{-3}\) & \(1.2765\times10^{-2}\) \\
C4 & \(8.7362\times10^{-4}\) & \(1.4837\times10^{-3}\) \\
C5 & \(4.1488\times10^{-5}\) & \(4.4486\times10^{-5}\) \\

\hline
\end{tabular}

\end{table}

The corresponding depthwise velocity profiles are evaluated at \(x=2.5\).
\textbf{Figure~\ref{Fig:Ablation_velocity_profiles}} compares the five cases for the
macro-pore and micro-pore networks. The comparison directly reveals the progressive
resolution of the layer-wise velocity variations and the behavior of the
approximations in the vicinity of the material interfaces.

%-----------------------------;
%  Figure: Velocity profiles  ;
%-----------------------------;
\begin{figure}[!htb]
    \centering

    \includegraphics[
        width=0.95\textwidth
    ]{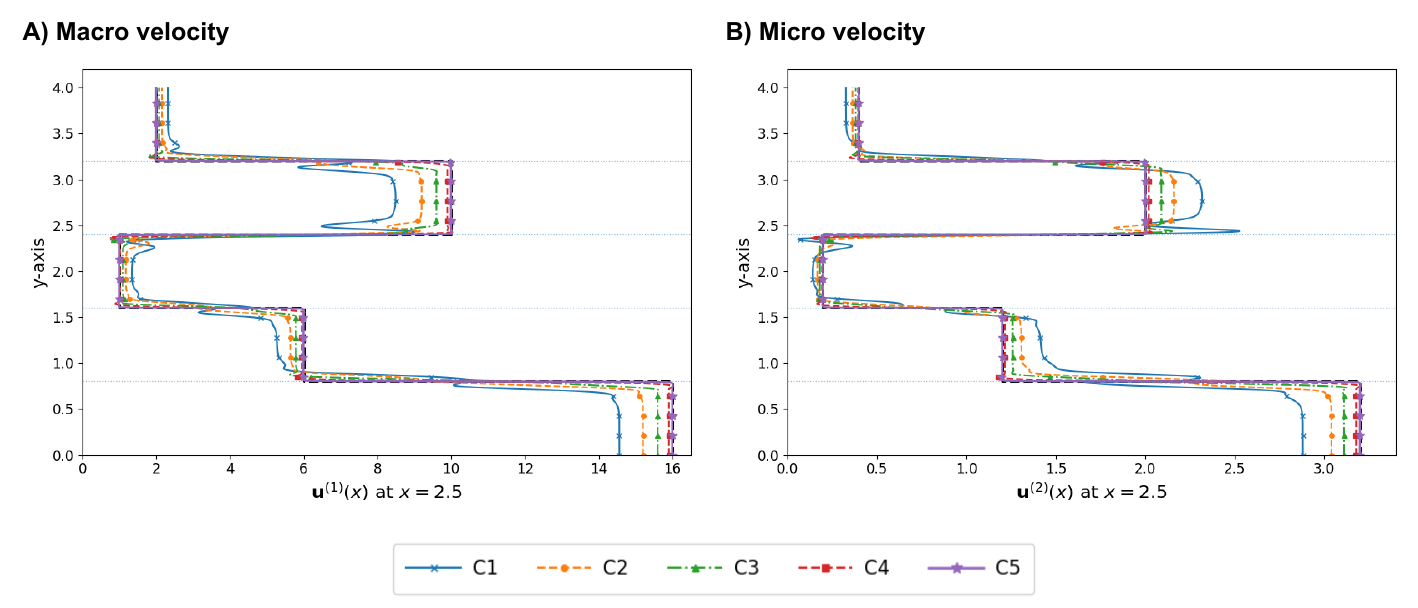}

    \caption{
    Ablation study for the layered porous-medium problem.
    Depthwise horizontal-velocity profiles at \(x=2.5\) for
    (A) the macro-network, \(\mathbf{u}^{(1)}\), and
    (B) the micro-network, \(\mathbf{u}^{(2)}\), for cases C1--C5.
    }

    \label{Fig:Ablation_velocity_profiles}
\end{figure}

Figure~\ref{Fig:Ablation_velocity_profiles} shows a clear progressive improvement in the predicted velocity fields as the individual ingredients of the framework are introduced. The basic neural approximation exhibits appreciable departures from the layer-wise response, particularly in the vicinity of the permeability interfaces. Introducing a shared representation
reduces these deviations, while Fourier encoding improves the resolution of the sharp velocity transitions. The most pronounced improvement follows from adaptive sampling, which substantially reduces the interface-localized errors. With adaptive weighting included, the predicted macro- and micro-pore-network velocities are virtually indistinguishable from the reference profiles. These observations are consistent with the relative \(L^2(\Omega)\) errors reported in Table~\ref{Tab:Ablation_errors}.

\emph{Taken together, the ablation results indicate that the improvement does not arise from a single architectural modification.} Rather, the shared representation, Fourier encoding, adaptive sampling, and adaptive weighting act cumulatively to improve both the global accuracy and the resolution of permeability-induced velocity jumps. The layered-medium problem therefore provides direct numerical evidence for the individual and combined roles of the principal ingredients of the proposed framework.
    %
    %============================;
%  Appendix: Computer Codes  ;
%============================;
\section{COMPUTER CODES}
\label{sec:computer-codes}
This appendix presents four code blocks that illustrate a shared-trunk neural network, an adaptive multi-objective weighting scheme, and a flexible training workflow for multi-task learning.

%======================================================%
%  Code 1: Shared-trunk multi-head network in PyTorch  %
%======================================================%
\subsection{Shared-trunk multi-head architecture}
The code block below specifies a shared-trunk neural network with a custom Swish activation and multiple lightweight heads. The trunk extracts common representations, while the heads produce task-specific outputs. This architecture is flexible, efficient, and well suited for multi-task or multiphysics applications.
\begin{lstlisting}[language=Python, style=mypython, caption={General shared-trunk multi-head architecture with a custom activation.}]
class Swish(nn.Module):
    def __init__(self, beta=1.0):
        super().__init__()
        self.beta = nn.Parameter(torch.tensor(beta))
    def forward(self, x):
        return x * torch.sigmoid(self.beta * x)

class MultiHeadNet(nn.Module):
    def __init__(self, in_dim, hidden=256, depth=15, head_dims=None):
        super().__init__()
        layers = []
        for i in range(depth):
            layers.append(nn.Linear(in_dim if i == 0 else hidden, hidden))
            layers.append(Swish())
            if i % 2 == 1:   # optional residual
                layers.append(nn.Identity())
        self.trunk = nn.Sequential(*layers)

        if head_dims is None:
            head_dims = {"head1": 1, "head2": 1}
        self.heads = nn.ModuleDict({
            name: nn.Linear(hidden, out_dim) 
            for name, out_dim in head_dims.items()
        })

    def forward(self, x):
        z = self.trunk(x)
        return {name: head(z) for name, head in self.heads.items()}

model = MultiHeadNet(in_dim=ENCODING_DIM,
                     hidden=256,
                     depth=12,
                     head_dims={"h1": 1, "h2": 2, "h3": 1}).to(DEVICE)
\end{lstlisting}

%======================================================;
%  Code 2: Inversion network with parameter injection  ;
%======================================================;
\subsection{Network architecture for inversion}

For the inversion task, the same shared-trunk architecture described in the previous subsection is used. The only modification is the inclusion of the unknown parameter $\beta$ as an additional input to the network. This parameter is concatenated with the encoded feature vector before entering the trunk, allowing the network to condition its predictions on the value of $\beta$.

Consequently, the first linear layer of the trunk receives $(d_{\text{enc}}+1)$ inputs, where $d_{\text{enc}}$ denotes the encoding dimension and the additional component corresponds to $\beta$. Aside from this change, the network architecture and activation functions remain unchanged. The relevant implementation changes are shown below.

\begin{lstlisting}[language=Python, style=mypython, caption={Minimal modifications required for the inversion network.}]
class TrunkNet(nn.Module):
    # First layer now takes ENCODING_DIM + 1 (extra input for beta)
    def __init__(self, in_dim, hidden=256, depth=12):
        super().__init__()
        layers = []
        for i in range(depth):
            in_features = (in_dim + 1) if i == 0 else hidden
            layers.append(nn.Linear(in_features, hidden))
            layers.append(Swish())
        self.body = nn.Sequential(*layers)

    def forward(self, enc, beta):
        # Concatenate beta with encoded input
        z_in = torch.cat([enc, beta], dim=-1)
        z = self.body(z_in)
\end{lstlisting}

%==================================;
%  Subsection: Adaptive weighting  ;
%==================================;
\subsection{Adaptive weighting}
\label{Subsec:APINNS_Adaptive_weighting}
The code below presents a general adaptive weighting mechanism for multi-objective optimization. It monitors the evolution of individual loss terms and dynamically adjusts task weights to maintain balance among objectives. This formulation can be broadly applied in settings where competing losses must be optimized jointly. 
\begin{lstlisting}[language=Python, style=mypython, caption={General adaptive weighting mechanism for balancing multiple objectives.}]
class AdaptiveWeights:
    def __init__(self, alpha=Alpha , window=Window, ratio_thr=Rho, device=DEVICE):
        self.alpha   = alpha
        self.window  = window
        self.ratio   = ratio_thr
        self.device  = device
        self._prev   = {}    # task -> previous loss value (float)
        self._buf    = {}    # task -> deque of recent fractional improvements
        self._lam    = {}    # task -> tensor weight (on device)

    def __getattr__(self, name):
        if name.startswith("w_") and name[2:] in self._lam:
            return self._lam[name[2:]]
        raise AttributeError(name)

    @torch.no_grad()
    def rebalance(self, **loss_dict):
        # lazy init
        if not self._lam:
            for t in loss_dict:
                self._prev[t] = None
                self._buf [t] = deque(maxlen=self.window)
                self._lam [t] = torch.tensor(1.0, device=self.device)

       
        for t, L in loss_dict.items():
            Lf = float(L.detach()) if hasattr(L, "detach") else float(L)
            if self._prev[t] is not None:
                prev = self._prev[t]
                denom = prev + 1e-12
                rel_improve = (prev - Lf) / denom
                if rel_improve < 0.0:
                    rel_improve = 0.0
                self._buf[t].append(rel_improve)
            self._prev[t] = Lf

        if any(len(buf) < self.window for buf in self._buf.values()):
            return  # weights remain as-is

        V = {t: (sum(buf)/self.window) for t, buf in self._buf.items()}
        vmax, vmin = max(V.values()), min(V.values())

        if vmax == vmin or vmax / max(vmin, 1e-12) <= self.ratio:
            return  # no reweighting

        # fastest (largest improvement) 
        f = max(V, key=V.get)
        for t in self._lam:
            if t == f:
                self._lam[t].fill_(1.0)
            else:
                # slower tasks get larger weights 
                Rj = (vmax - V[t]) / (vmax - vmin)
                self._lam[t].fill_(1.0 + self.alpha * Rj)

weights = AdaptiveWeights(alpha=Alpha, window=Window, ratio_thr=Rho)
\end{lstlisting}

%=================================;
%  Subsection: Adaptive sampling  ;
%=================================;
\subsection{Adaptive sampling}
\label{Subsec:APINNS_Adaptive_sampling}
Finally, the code below outlines a generic training framework that alternates between Adam optimization, L-BFGS refinement, and adaptive enrichment of training points. While multiple objectives are balanced through an adaptive weighting mechanism, residual-based sampling focuses new points in regions of high error.
\begin{lstlisting}[language=Python, style=mypython, caption={General training loop combining Adam pretraining, L-BFGS refinement, and adaptive sampling.}]
for round in range(NUM_ROUNDS + 1):
    print(f"\nRound {round} | interior pts = {len(train_cloud)}")

    #---------------- Adam phase ----------------;
    for epoch in range(EPOCHS_ADAM):
        # sample batch of training points
        idx = torch.randint(0, len(train_cloud), (BATCH_SIZE,), device=DEVICE)
        x   = train_cloud[idx].detach().clone().requires_grad_(True)

        # evaluate multiple objectives
        L1 = loss_fn1(x)       
        L2 = loss_fn2(x)       
        L3 = loss_fn3(x)       

        # adaptive weighting
        weights.rebalance(obj1=L1, obj2=L2, obj3=L3)
        w1, w2, w3 = weights.w_obj1, weights.w_obj2, weights.w_obj3

        # backward on weighted sum
        adam.zero_grad()
        loss = w1 * L1 + w2 * L2 + w3 * L3
        loss.backward()
        torch.nn.utils.clip_grad_norm_(model.parameters(), 1.0)
        adam.step()
        scheduler.step(loss.detach())

        if epoch % LOG_INTERVAL == 0 or epoch == EPOCHS_ADAM - 1:
            print(f"epoch {epoch:4d} | "
                  f"L1 {L1:.1e}  L2 {L2:.1e}  L3 {L3:.1e} | "
                  f"w1 {w1:.2f} w2 {w2:.2f} w3 {w3:.2f}")

    #---------------- L-BFGS refinement ----------------;
    def closure():
        lbfgs.zero_grad()
        L1, L2, L3 = loss_fn1(train_cloud), loss_fn2(train_cloud), loss_fn3(train_cloud)
        tot = weights.w_obj1 * L1 + weights.w_obj2 * L2 + weights.w_obj3 * L3
        tot.backward(retain_graph=True)
        return tot

    lbfgs.step(closure)

    #---------------- Adaptive enrichment ----------------;
    if round == NUM_ROUNDS:
        break

    cand = sample_domain_points(NEW_POINTS * 3, device=DEVICE).requires_grad_(True)
    residuals = compute_residuals(cand)
    with torch.no_grad():
        score = residuals.norm(dim=-1)

    top_idx    = torch.topk(score, NEW_POINTS).indices
    train_cloud = torch.cat([train_cloud, cand[top_idx]], dim=0)

    if len(train_cloud) > MAX_POINTS:
        perm = torch.randperm(len(train_cloud), device=DEVICE)[:MAX_POINTS]
        train_cloud = train_cloud[perm]
    train_cloud.requires_grad_(True)

print("Training finished.")
\end{lstlisting}

%===============================;
%  Data availability statement  ;
%===============================;
\section*{DATA AVAILABILITY STATEMENT}
    Data sharing is not applicable to this article as no new data were created or analyzed in this study. The code snippets used in this work are provided in the Appendix.

    \vspace{0.1in}
    %=======================;
    %  Funding declaration  ;
    %=======================;
    \noindent\textbf{Funding Declaration.} The authors acknowledge the support from the Environmental Molecular Sciences Laboratory (EMSL), a DOE Office of Science User Facility sponsored by the Biological and Environmental Research program under Contract No: DE-AC05-76RL01830 (Large-Scale Research User Project No: 60720, Award DOI: \url{10.46936/lser.proj.2023.60720/60008914}). The views and opinions of the authors expressed herein do not necessarily state or reflect those of the United States Government or any agency thereof.
    
     \vspace{0.1in}
     \noindent\textbf{Declaration of Competing Interests.} The authors declare that they have no known competing financial interests or personal relationships that could have appeared to influence the work reported in this paper.

%================;
%  Bibliography  ;
%================;
\bibliographystyle{plainnat}
\bibliography{Master_References}

@article{aavatsmark2002introduction,
  title={{An introduction to multipoint flux approximations for quadrilateral grids}},
  author={I.~Aavatsmark},
  journal={Computational Geosciences},
  volume={6},
  number={4},
  pages={405--432},
  year={2002},
  publisher={Springer},
  doi={10.1007/s10596-002-9001-2}
}

@book{adams2003sobolev,
  title     = {{Sobolev Spaces}},
  author    = {R.~A.~Adams and J.~J.~F.~Fournier},
  volume    = {140},
  year      = {2003},
  publisher = {Elsevier},
  series    = {Pure and Applied Mathematics},
  edition   = {2nd},
  address   = {Amsterdam}
}

@article{adhikari2025reactive,
  title={{Reactive transport modeling with physics-informed machine learning for critical minerals applications}},
  author={K.~Adhikari and M.~L.~Mamud and M.~K.~Mudunuru and K.~B.~Nakshatrala},
  journal={Transport in Porous Media},
  year    = {2026},
  volume  = {153},
  pages   = {45},
  doi     = {10.1007/s11242-026-02301-9}
}

@book{allen2021mathematics,
  title={{The Mathematics of Fluid Flow Through Porous Media}},
  author={M.~B.~Allen},
  year={2021},
  publisher={John Wiley \& Sons},
  address = {New Jersey}
}

@book{Axler2020_MeasureIntegrationRealAnalysis,
  author    = {S.~Axler},
  title     = {{Measure, Integration \& Real Analysis}},
  series    = {Graduate Texts in Mathematics},
  volume    = {282},
  publisher = {Springer Nature},
  address   = {Cham},
  year      = {2020}
}

@article{Barenblatt1960,
  author    = {G. I. Barenblatt and I. P. Zheltov and I. N. Kochina},
  title     = {Basic concepts in the theory of seepage of homogeneous liquids in fissured rocks},
  journal   = {Journal of Applied Mathematics and Mechanics},
  volume    = {24},
  number    = {5},
  pages     = {1286--1303},
  year      = {1960},
  doi       = {10.1016/0021-8928(60)90107-6}
}

@book{brezis2011functional,
  title={{Functional Analysis, Sobolev Spaces and Partial Differential Equations}},
  author={H.~Brezis},
  year={2011},
  publisher={Springer},
  address={New York}
}

@book{brezzi2012mixed,
  title={{Mixed and Hybrid Finite Element Methods}},
  author={F.~Brezzi and M.~Fortin},
  volume={15},
  year={2012},
  publisher={Springer Science \& Business Media},
  address={New York}
}

@article{caruana1997multitask,
  title={Multitask learning},
  author={R.~Caruana},
  journal={Machine Learning},
  volume={28},
  pages={41--75},
  year={1997},
  publisher={Springer},
  address = {New York},
  doi={10.1023/A:1007379606734}
}

@article{chang2017modification,
  title     = {{Modification to Darcy–Forchheimer model due to pressure-dependent viscosity: consequences and numerical solutions}},
  author    = {J.~Chang and K.~B.~Nakshatrala and J.~N.~Reddy},
  journal   = {Journal of Porous Media},
  volume    = {20},
  number    = {3},
  pages     = {263--285},
  year      = {2017},
  doi       = {10.1615/JPorMedia.v20.i3.60},
}

@book{chen2006computational,
  title={{Computational Methods for Multiphase Flows in Porous Media}},
  author={Z.~Chen and G.~Huan and Y.~Ma},
  year={2006},
  publisher={Society for Industrial and Applied Mathematics},
  address={Philadelphia}
}

@article{deng2011dual,
  title     = {A dual-continuum model for simulating reactive transport in fractured porous media},
  author    = {H.~Deng and P.~H.~Stauffer and Z.~Dai and A.~V.~Wolfsberg},
  journal   = {Advances in Water Resources},
  volume    = {34},
  number    = {9},
  pages     = {1101--1111},
  year      = {2011},
  publisher = {Elsevier},
  doi       = {10.1016/j.advwatres.2011.06.003}
}

@book{dickenson1997filters,
  title={{Filters and Filtration Handbook}},
  author={T.~C.~Dickenson},
  year={1997},
  publisher={Elsevier},
  address = {New York}
}

@article{dissanayake1994neural,
  title={Neural-network-based approximations for solving partial differential equations},
  author={M.~W.~M.~G.~Dissanayake and N.~Phan-Thien},
  journal={Communications in Numerical Methods in Engineering},
  volume={10},
  number={3},
  pages={195--201},
  year={1994},
  doi = {10.1002/cnm.1640100303},
  publisher={Wiley Online Library}
}

@book{evanspartial,
  title={{Partial Differential Equations}},
  author={L.~C.~Evans},
  year={1998},
  publisher={American Mathematical Society},
  address={Providence}
}

@article{gao2025physics,
  title={Physics-informed neural networks with adaptive loss weighting algorithm for solving partial differential equations},
  author={B.~Gao and R.~Yao and Y.~Li},
  journal={Computers \& Mathematics with Applications},
  volume={181},
  pages={216--227},
  year={2025},
  doi={10.1016/j.camwa.2025.01.007}
}

@article{Gerke1993,
  author    = {H. H. Gerke and M. Th. van Genuchten},
  title     = {A dual-porosity model for simulating the preferential movement of water and solutes in structured porous media},
  journal   = {Water Resources Research},
  volume    = {29},
  number    = {2},
  pages     = {305--319},
  year      = {1993},
  doi       = {10.1029/92WR02339}
}

@book{goodfellow2016deep,
  title={{Deep Learning}},
  author={I.~Goodfellow and Y.~Bengio and A.~Courville},
  year={2016},
  publisher={MIT Press Cambridge}, 
  address={Cambridge},
  note={\url{http://www.deeplearningbook.org}}
}

@article{jagtap2020adaptive,
  title={Adaptive activation functions accelerate convergence in deep and physics-informed neural networks},
  author={A.~D.~Jagtap and K.~Kawaguchi and G.~E.~Karniadakis},
  journal={Journal of Computational Physics},
  volume={404},
  pages={109136},
  year={2020},
  publisher={Elsevier},
  doi={10.1016/j.jcp.2019.109136}
}

@article{jagtap2023coolpinns,
  title={{CoolPINNs: A physics-informed neural network modeling of active cooling in vascular systems}},
  author={N.~V.~Jagtap and M.~K.~Mudunuru and K.~B.~Nakshatrala},
  journal={Applied Mathematical Modelling},
  volume={122},
  pages={265--287},
  year={2023},
  publisher={Elsevier},
  doi={10.1016/j.apm.2023.04.020}
}

@article{joodat2018modeling,
  title={{Modeling flow in porous media with double porosity/permeability: A stabilized mixed formulation, error analysis, and numerical solutions}},
  author={S.~H.~S.~Joodat and K.~B.~Nakshatrala and R.~Ballarini},
  journal={Computer Methods in Applied Mechanics and Engineering},
  volume={337},
  pages={632--676},
  year={2018},
  doi={10.1016/j.cma.2018.04.004}
}

@article{joshaghani2019stabilized,
  title={{A stabilized mixed discontinuous Galerkin formulation for double porosity/permeability model}},
  author={M.~S.~Joshaghani and S.~H.~S.~Joodat and K.~B.~Nakshatrala},
  journal={Computer Methods in Applied Mechanics and Engineering},
  volume={352},
  pages={508--560},
  year={2019},
  doi={10.1016/j.cma.2019.04.010}
}

@article{joshaghani2019composable,
  title={Composable block solvers for the four-field double porosity/permeability model},
  author={M.~S.~Joshaghani and J.~Chang and K.~B.~Nakshatrala and M.~G.~Knepley},
  journal={Journal of Computational Physics},
  volume={386},
  pages={428--466},
  year={2019},
  doi={10.1016/j.jcp.2019.02.020}
}

@article{katende2025stability,
  author  = {R.~Katende},
  title   = {Stability Analysis of Physics-Informed Neural Networks via Variational Coercivity, Perturbation Bounds, and Concentration Estimates},
  journal = {arXiv preprint},
  doi     = {10.48550/arXiv.2506.13554},
  year    = {2025}
}

@book{ketkar2021deep,
  title={Deep Learning with Python: Learn Best Practices of Deep Learning Models with PyTorch},
  author={N.~Ketkar and J.~Moolayil},
  year={2021},
  publisher={Springer},
  address ={New York}
}

@article{kingma2014adam,
  title={{Adam: A method for stochastic optimization}},
  author={D.~P.~Kingma and J.~Ba},
  journal={arXiv preprint},
  doi = {10.48550/arXiv.1412.6980},
  year={2014}
}

@article{kim2022fast,
  title={A fast and accurate physics-informed neural network reduced order model with shallow masked autoencoder},
  author={Y.~ Kim and Y.~Choi and D.~Widemann and T.~Zohdi},
  journal={Journal of Computational Physics},
  volume={451},
  year={2022},
  doi = {10.1016/j.jcp.2021.110841},
  publisher={Elsevier}
}

@article{kissas2020machine,
  title={{Machine learning in cardiovascular flows modeling: Predicting arterial blood pressure from non-invasive measurements}},
  author={G.~Kissas and Y.~Yang and E.~Hwuang and W.~R.~Witschey and J.~A.~Detre and P.~Perdikaris},
  journal={Computer Methods in Applied Mechanics and Engineering},
  volume={358},
  pages={112623},
  year={2020},
  doi={10.1016/j.cma.2019.112623}
}

@inproceedings{krishnapriyan2021characterizing,
    title={{Characterizing possible failure modes in physics-informed neural networks}},
  author={A.~Krishnapriyan and A.~Gholami and S.~Zhe and R.~M.~Kirby and M.~W.~Mahoney},
  booktitle = {Advances in Neural Information Processing Systems 34 (NeurIPS 2021)},
  year      = {2021},
  doi = {10.48550/arXiv.2109.01050},
  url       = {https://proceedings.neurips.cc/paper/2021/hash/df438e5206f31600e6ae4af72f2725f1-Abstract.html}
}

@article{lagaris2000neural,
  title={Neural-network methods for boundary value problems with irregular boundaries},
  author={I.~E.~Lagaris and A.~C.~Likas and D.~G.~Papageorgiou},
  journal={IEEE Transactions on Neural Networks},
  volume={11},
  number={5},
  pages={1041--1049},
  doi = {10.1109/72.870037},
  year={2000},
  publisher={IEEE}
}

@article{lu2021deepxde,
  title={{DeepXDE: A deep learning library for solving differential equations}},
  author={L.~Lu and X.~Meng and X.~Mao and G.~E.~Karniadakis},
  journal={SIAM Review},
  volume={63},
  number={1},
  doi = {10.1137/19M1274067},
  pages={208--228},
  year={2021},
  publisher={SIAM}
}

@article{manav2024phase,
  title={Phase-field modeling of fracture with physics-informed deep learning},
  author={M.~Manav and R.~Molinaro and S.~Mishra and L.~De Lorenzis},
  journal={Computer Methods in Applied Mechanics and Engineering},
  volume={429},
  pages={117104},
  year={2024},
  doi ={10.1016/j.cma.2024.117104},
  publisher={Elsevier}
}

@article{mao2020physics,
  title={Physics-informed neural networks for high-speed flows},
  author={Z.~Mao and A.~D.~Jagtap and G.~E.~Karniadakis},
  journal={Computer Methods in Applied Mechanics and Engineering},
  volume={360},
  pages={112789},
  year={2020},
  publisher={Elsevier},
  doi={10.1016/j.cma.2019.112789}
}

@article{masud2002stabilized,
  title={{A stabilized mixed finite element method for Darcy flow}},
  author={A.~Masud and T.~J.~R.~Hughes},
  journal={Computer Methods in Applied Mechanics and Engineering},
  volume={191},
  number={39--40},
  pages={4341--4370},
  year={2002},
  doi={10.1016/S0045-7825(02)00398-2}
}

@article{molins2015reactive,
  title     = {Reactive transport modeling of coupled processes in porous media},
  author    = {S.~Molins},
  journal   = {Reviews in Mineralogy and Geochemistry},
  volume    = {80},
  number    = {1},
  pages     = {461--481},
  year      = {2015},
  publisher = {Mineralogical Society of America},
  doi       = {10.2138/rmg.2015.80.14}
}

@article{maduri2025flow,
  title        = {{Flow through porous media: A Hopf--Cole transformation approach for modeling pressure-dependent viscosity}},
  author       = {V.~S.~Maduri and K.~B.~Nakshatrala},
  journal      = {Journal of Applied Mechanics},
  volume       = {92},
  number       = {11},
  pages        = {111010},
  year         = {2025},
  doi          = {10.1115/1.4069108}
}

@article{maduri2026deepLS,
  title        = {{A machine learning--enhanced Hopf--Cole formulation for nonlinear gas flow in porous media}},
  author       = {V.~S.~Maduri and K.~B.~Nakshatrala},
  journal = {Computer Methods in Applied Mechanics and Engineering},
  volume = {461},
  pages = {119--185},
  year = {2026},
  publisher = {Elsevier},
  doi = {https://doi.org/10.1016/j.cma.2026.119185}
}

@article{nabian2019deep,
  title={A deep learning solution approach for high-dimensional random differential equations},
  author={M.~A.~Nabian and H.~Meidani},
  journal={Probabilistic Engineering Mechanics},
  volume={57},
  pages={14--25},
  doi ={10.1016/j.probengmech.2019.05.001},
  year={2019},
  publisher={Elsevier}
}

@article{nakshatrala2018modeling,
  title={{Modeling flow in porous media with double porosity/permeability: Mathematical model, properties, and analytical solutions}},
  author={K.~B.~Nakshatrala and S.~H.~S.~Joodat and R.~Ballarini},
  journal={Journal of Applied Mechanics},
  volume={85},
  number={8},
  pages={081009},
  year={2018},
  publisher={American Society of Mechanical Engineers Digital Collection},
  doi={10.1115/1.4040116}
}

@book{nrc2008critical,
  title     = {Minerals, Critical Minerals, and the U.S. Economy},
  author    = {{National Research Council}},
  year      = {2008},
  publisher = {National Academies Press},
  address   = {Washington, DC},
  doi       = {10.17226/12034}
}

@inproceedings{paszke2017automatic,
  title={Automatic differentiation in {PyTorch}},
  author={A.~Paszke and S.~Gross and S.~Chintala and G.~Chanan and E.~Yang and Z.~DeVito and Z.~Lin and A.~Desmaison and L.~Antiga and A.~Lerer},
  booktitle={Proceedings of the 31st Conference on Neural Information Processing Systems (NIPS 2017), Autodiff Workshop},
  year={2017}
}

@article{pornea2022,
  author       = {A.~G.~M.~Pornea and J.~M.~C.~Puguan and J.~L.~A.~Ruello and H.~Kim},
  title        = {{Multifunctional dual‐pore network aerogel composite material for broadband sound absorption, thermal insulation, and fire repellent applications}},
  journal      = {ACS Applied Polymer Materials},
  year         = {2022},
  volume       = {4},
  number       = {4},
  pages        = {2880--2895},
  doi          = {10.1021/acsapm.2c00139}
}

@article{raissi2017physics,
  title={Physics informed deep learning (part i): Data-driven solutions of nonlinear partial differential equations},
  author={M.~Raissi and P.~Perdikaris and G.~E.~Karniadakis},
  journal={arXiv preprint}, 
  year={2017},
  doi = {10.48550/arXiv.1711.10561}
}

@article{raissi2019physics,
  title={{Physics-informed neural networks: A deep learning framework for solving forward and inverse problems involving nonlinear partial differential equations}},
  author={M.~Raissi and P.~Perdikaris and G.~E.~Karniadakis},
  journal={Journal of Computational Physics},
  volume={378},
  pages={686--707},
  year={2019},
  doi={10.1016/j.jcp.2018.10.045}
}

@article{raissi2020hidden,
  title={{Hidden fluid mechanics: Learning velocity and pressure fields from flow visualizations}},
  author={M.~Raissi and A.~Yazdani and G.~E.~Karniadakis},
  journal={Science},
  volume={367},
  number={6481},
  pages={1026--1030},
  year={2020},
  publisher={American Association for the Advancement of Science},
  doi={10.1126/science.aaw4741}
}

@book{brenner2008mathematical,
  title        = {{The Mathematical Theory of Finite Element Methods}},
  author       = {S.~C.~Brenner and L.~R.~Scott},
  year         = {2008},
  edition      = {3rd},
  series       = {Texts in Applied Mathematics},
  volume       = {15},
  publisher    = {Springer},
  address      = {New York}
}

@article{seo2024solving,
  title={Solving real-world optimization tasks using physics-informed neural computing},
  author={S.~Seo},
  journal={Scientific Reports},
  volume={14},
  number={1},
  pages={202},
  year={2024},
  publisher={Nature Publishing Group UK London},
  doi = {10.1038/s41598-023-49977-3}
}

@article{steefel2005reactive,
  title     = {{Reactive transport modeling: An essential tool and a new research approach for the Earth sciences}},
  author    = {C.~I.~Steefel and D.~J.~DePaolo and P.~C.~Lichtner},
  journal   = {Earth and Planetary Science Letters},
  volume    = {240},
  number    = {3-4},
  pages     = {539--558},
  year      = {2005},
  publisher = {Elsevier},
  doi       = {10.1016/j.epsl.2005.09.017}
}

@book{strack2017analytical,
  title     = {{Analytical Groundwater Mechanics}},
  author    = {O.~D.~L.~Strack},
  year      = {2017},
  publisher = {Cambridge University Press},
  address   = {Cambridge}
}

@article{tancik2020fourier,
  title={Fourier features let networks learn high frequency functions in low dimensional domains},
  author={M.~Tancik and P.~Srinivasan and B.~Mildenhall and S.~Fridovich-Keil and N.~Raghavan and U.~Singhal and R.~Ramamoorthi and J.~Barron and R.~Ng},
  journal={Advances in Neural Information Processing Systems},
  volume={33},
  pages={7537--7547},
  year={2020}
}

@article{tartakovsky2020physics,
  title={Physics-informed deep neural networks for learning parameters and constitutive relationships in subsurface flow problems},
  author={A.~M.~Tartakovsky and C.~O.~Marrero  and P.~Perdikaris and G.~D.~Tartakovsky and D.~Barajas-Solano},
  journal={Water Resources Research},
  volume={56},
  number={5},
  pages={e2019WR026731},
  year={2020},
  doi={10.1029/2019WR026731}
}

@article{Vogel2000,
  author    = {T. Vogel and H. H. Gerke and R. Zhang and M. Th. van Genuchten},
  title     = {Modeling flow and transport in a two-dimensional dual-permeability system with spatially variable hydraulic properties},
  journal   = {Journal of Hydrology},
  volume    = {238},
  pages     = {78--89},
  year      = {2000},
  doi       = {10.1016/S0022-1694(00)00327-9}
}

@article{wang2021when,
  title={{When and why PINNs fail to train: A neural tangent kernel perspective}},
  author={S.~Wang and X.~Yu and P.~Perdikaris},
  journal={Journal of Computational Physics},
  volume={449},
  pages={110768},
  year={2021},
  publisher={Elsevier},
  doi={10.1016/j.jcp.2021.110768}
}

@article{wang2024piratenets,
  title={{PirateNets: Physics-informed deep learning with residual adaptive networks}},
  author={S.~Wang and B.~Li and Y.~Chen and P.~Perdikaris},
  journal={Journal of Machine Learning Research},
  volume={25},
  number={402},
  pages={1--51},
  year={2024},
  url     = {http://jmlr.org/papers/v25/24-0313.html}

}

@article{warren1963behavior,
  title = {The behavior of naturally fractured reservoirs},
  author = {J.~E.~Warren and P.~J.~Root},
  journal = {Society of Petroleum Engineers Journal},
  volume = {3},
  number = {03},
  pages = {245--255},
  year = {1963},
  publisher = {Society of Petroleum Engineers},
  doi = {10.2118/426-PA}
}

@article{zou2023hydra,
  title={{L-HYDRA: Multi-head physics-informed neural networks}},
  author={Z.~Zou and G.~E.~Karniadakis},
  journal={arXiv preprint},
  year={2023},
  doi = {10.48550/arXiv.2301.02152}
}
%%%
\end{document}